\documentclass{statsoc}
\usepackage{geometry}

\usepackage{etoolbox}
\makeatletter
\patchcmd{\@makecaption}
  {\parbox}
  {\advance\@tempdima-\fontdimen2} 
  {}{}
\makeatother

\usepackage[utf8]{inputenc}
\RequirePackage{amsthm,amsmath,amsfonts,amssymb}
\RequirePackage[colorlinks,citecolor=blue,urlcolor=blue]{hyperref}
\RequirePackage{graphicx}
\usepackage{mathtools}
\usepackage{csquotes}
\usepackage{enumitem}
\usepackage{bbm}
\usepackage[capitalize,nameinlink,noabbrev]{cleveref}
\usepackage[dvipsnames]{xcolor}
\usepackage{tikz}
\usetikzlibrary{calc}
\usepackage{bm}
\usepackage{tkz-euclide}
\usepackage{svg}
\usepackage[caption=false]{subfig}
\usepackage{array}
\usepackage{multirow}
\usepackage{appendix}
\usepackage{doi}

\usepackage{comment}
\usepackage[textwidth=2.8cm]{todonotes}
\setuptodonotes{color=blue!30}

\usepackage{booktabs} 
\usepackage[linesnumbered,ruled,lined,boxed,commentsnumbered]{algorithm2e}

\numberwithin{equation}{section}
\theoremstyle{plain}

\newtheorem{theorem}{Theorem}[section]
\newtheorem{lemma}[theorem]{Lemma}
\newtheorem{corollary}[theorem]{Corollary}
\newtheorem{proposition}[theorem]{Proposition}
\newtheorem*{informal}{\cref{cor:conv_xcut} (informal)}

\theoremstyle{remark}

\newtheorem{definition}{Definition}
\newtheorem{example}{Example}

\crefname{assumptions}{Assumption}{Assumptions}

\crefname{algocfline}{Algorithm}{Algorithms}

\newcommand{\xistref}{\hyperref[alg:xist]{Xist}\xspace}

\newcommand{\xvstnameref}{\hyperref[alg:xvst]{basic Xvst algorithm}\xspace}
\newcommand{\xistnameref}{\hyperref[alg:xist]{Xist algorithm}\xspace}
\newcommand{\aref}[1]{\hyperref[#1]{Assumption \ref*{#1}}}
\newcommand{\firstassumptions}{\hyperref[assumptions:iid]{\ref*{assumptions:iid}--\ref*{assumptions:diff}~}}

\newlength{\subcolumnwidth}

\newcommand{\nextsubcolumn}[1][]{%
  \cr\noalign{\hfill}
  \if\relax\detokenize{#1}\relax\else\hsize=#1\setlength{\subcolumnwidth}{\hsize}\fi
}

\DeclareMathOperator*{\argmin}{arg\,min}

\DeclareMathOperator{\cov}{Cov}
\DeclareMathOperator{\var}{Var}
\DeclareMathOperator{\bal}{Bal}
\DeclareMathOperator{\ball}{bal}
\DeclareMathOperator{\XC}{\mathrm{XC}}
\DeclareMathOperator{\MC}{\mathrm{MC}}
\DeclareMathOperator{\RC}{\mathrm{RC}}
\DeclareMathOperator{\NC}{\mathrm{NC}}
\DeclareMathOperator{\CC}{\mathrm{CC}}
\DeclareMathOperator{\kXC}{\mathrm{kXC}}
\DeclareMathOperator{\kMC}{\mathrm{kMC}}

\DeclareMathOperator{\xvst}{\mathrm{Xvst}}

\DeclareMathOperator{\xist}{\mathrm{Xist}}
\DeclareMathOperator{\Vloc}{V^{\mathrm{loc}}}

\DeclareMathOperator{\vol}{vol}
\DeclareMathOperator{\mult}{Mult}
\DeclareMathOperator{\minn}{min}
\DeclareMathOperator{\dist}{dist}
\DeclareMathOperator{\diag}{diag}

\newcolumntype{x}[1]{>{\centering\arraybackslash\hspace{0pt}}p{#1}}

\newcommand{\comp}[1]{{#1}^{\mathsf{c}}}
\newcommand{\norm}[1]{\left\lVert#1\right\rVert}
\newcommand{\scalprod}[2]{\langle#1,#2\rangle}
\newcommand{\mysubref}[2]{\hyperref[#1:#2]{\cref*{#1}~\ref*{#1:#2}}}
\newcommand{\assumptionref}[1]{\hyperref[#1]{Assumption \ref*{#1}}}
\renewcommand{\eqref}[1]{\hyperref[#1]{(\ref*{#1})}}
\newenvironment{tttenv}{\ttfamily}{\par}

\definecolor{niceblue}{rgb}{0.15,0.15,1} 
\definecolor{nicegreen}{rgb}{0.1,0.9,0.1} 

\renewcommand{\vec}[1]{{\boldsymbol#1}} 
\newcommand{\mat}[1]{{\boldsymbol#1}}    
\newcommand{\R}{\mathbb{R}} 
\newcommand{\N}{\mathbb{N}} 
\newcommand{\BO}{\mathcal{O}} 
\renewcommand{\P}{\mathbb{P}} 
\newcommand{\E}{\mathbb{E}} 
\newcommand{\dto}{\stackrel{D}{\longrightarrow}} 
\newcommand{\nto}{\stackrel{n\to\infty}{\longrightarrow}} 
\newcommand{\dom}{\Omega} 
\newcommand{\eps}{\varepsilon} 
\newcommand{\iidsim}{\stackrel{\mathrm{i.i.d.}}{\sim}} 
\newcommand{\eqd}{\stackrel{D}{=}} 
\newcommand{\abs}[1]{|#1|} 
\newcommand{\babs}[1]{\bigl|#1\bigr|} 

\title[Limits of discretized graph cuts]{Distributional limits of graph cuts on discretized grids}
\author{Housen Li}
\address{Institute for Mathematical Stochastics, Georg-August-University of Göttingen, Germany.}
\email{housen.li@mathematik.uni-goettingen.de}
\author{Axel Munk}
\address{Institute for Mathematical Stochastics, Georg-August-University of Göttingen, Germany.}
\email{munk@math.uni-goettingen.de}
\author[Li, Munk, Suchan]{Leo Suchan}
\address{Institute for Mathematical Stochastics, Georg-August-University of Göttingen, Germany.}
\email{leo.suchan@uni-goettingen.de}

\begin{document}

%
%

\begin{abstract}
Graph cuts are among the most prominent tools for clustering and classification analysis. While intensively studied from geometric and algorithmic perspectives, graph cut-based statistical inference still remains elusive to a certain extent. Distributional limits are fundamental in understanding and designing such statistical procedures on randomly sampled data. We provide explicit limiting distributions for \textit{balanced graph cuts} in general on a fixed but arbitrary discretization. In particular, we show that Minimum Cut, Ratio Cut and Normalized Cut behave asymptotically as the minimum of Gaussians as sample size increases. Interestingly, our results reveal a dichotomy for Cheeger Cut: The limiting distribution of the optimal objective value is the minimum of Gaussians only when the optimal partition yields two sets of unequal volumes, while otherwise the limiting distribution is the minimum of a random mixture of Gaussians. Further, we show the bootstrap consistency for all types of graph cuts by utilizing the directional differentiability of cut functionals. We validate these theoretical findings by Monte Carlo experiments, and examine differences between the cuts and the dependency on the underlying distribution. Additionally, we expand our theoretical findings to the \emph{Xist} algorithm, a computational surrogate of graph cuts recently proposed in Suchan, Li and Munk (arXiv, 2023), thus demonstrating the practical applicability of our findings e.g.\ in statistical tests.

\end{abstract}

%
%


\section{Introduction}
\label{sec:intro}

Graph partitioning remains one of the most fundamental tasks in graph-based data analysis, with applications found in cluster analysis \citep{FlakeTarjanTsioutsiouliklis2004, ZabihKolmogorov2004, Long_etal2022}, image segmentation \citep{ShiMalik2000, TaoJinZhang2007, PaulhacTaMegret2012}, machine learning \citep{ZhouBurges2007, Huang_etal2009, OnKimHeo2020}, and it is relevant to a wide range of domain applications, including nuclear physics, cell biology, and the analysis of social networks, to mention a few \citep{Yang_etal2014, XingKarp2001, Rahimian_etal2015}.

While a large body of work has been devoted to its mathematical understanding (e.g.\ \citealp{Trillos_etal2020,Mulas2021}) and algorithmic advances (e.g.\ \citealp{Chuzhoy_etal2020,SaranurakWang2019,ChuzhoyKhanna2019,SuchanLiMunk2023arXiv}), the understanding of prominent balanced graph cuts such as Ratio Cut (RCut; \citealp{HagenKahng1992}), Normalized Cut (NCut; \citealp{ShiMalik2000}), Cheeger Cut (CCut; \citealp{Cheeger1971}) from a statistical perspective has been the subject of only a few papers.
Among these, \citet{Trillos_etal2016} established the $\Gamma$-convergence of the empirical cut towards a continuum cut on the underlying distribution for various balancing terms. More recently, as a first step in the asymptotic analysis of graph cuts, \citet{TrillosMurrayThorpe2022} derived upper and lower bounds for the convergence rates for CCut in the same setting which, as the authors note, are in multiple ways not yet optimal. So while these results are valuable by assuring that through refinement of the graph structure the corresponding partitions stay valid asymptotically, they are still far from providing a limiting distribution, say, usable for statistical inference.

Although conceptually appealing, the original balanced graph cuts are of limited practical use as computing them is known to be NP-hard (e.g.\ \citealp{ShiMalik2000, SimaSchaeffer2006, Mohar1989, BuiJones1992}). Hence, attention on graph cuts has largely focused on surrogates, most prominently on spectral clustering \citep{vonLuxburg2007}, a well-known convex relaxation of NCut (and, by extension, of RCut) which reduces the graph cut to an eigenvalue problem on the graph Laplacian. Here, the asymptotic perspective has, for instance, been taken by \citet{LuxburgBelkinBousqet2008} who showed that under mild conditions the normalized graph Laplacian of a neighbourhood graph built from a sample converges toward a continuous operator depending on the underlying distribution. \citet{MaierLuxburgHein2013} explored the impact of the choice of neighbourhood structure on the cut value and its asymptotic behaviour.

While spectral clustering seemingly enables fast computation even for large graphs, by design it merely computes an \textit{imitation} of the true graph cut and can thus yield inaccurate results (see e.g.\ \citet{GuatteryMiller1998}, in particular, their famous \enquote{cockroach graph}). In this paper we turn the spectral clustering idea on its head by utilizing discretization: Instead of relaxation of a cut functional of a graph built from a (large) point cloud, we discretize the points first, construct a graph out of this (much smaller) discretized sample and then cut it using the exact graph cut. Or, to rephrase: While spectral clustering computes an imitation of the graph cut at a high resolution, we consider the exact cut at a low resolution. This approach will allow us to derive asymptotic distributions which, to the best of our knowledge, are the first such results for graph cuts.


\begin{figure}[htbp]
    \subfloat[Original sample]{
	\centering
	\begin{tikzpicture}[x=.65cm,y=.65cm]
		\coordinate (O) (0,0);
		\draw[help lines,xstep=2,ystep=2] (0,0) grid (10,6);
		\draw[black,line width=.5mm] (0,0) rectangle (10,6);
		\node [circle,fill,inner sep=2pt] (B) at (0.6,2.2) {};
		\node [circle,fill,inner sep=2pt] (C) at (0.7,5.2) {};
		\node [circle,fill,inner sep=2pt] (D) at (1.5,3.3) {};
		\node [circle,fill,inner sep=2pt] (E) at (2.7,2.6) {};
		\node [circle,fill,inner sep=2pt] (F) at (2.2,3.4) {};
		\node [circle,fill,inner sep=2pt] (G) at (3.2,4.3) {};
		\node [circle,fill,inner sep=2pt] (H) at (2.7,1.2) {};
		\node [circle,fill,inner sep=2pt] (I) at (3.6,3.1) {};
		\node [circle,fill,inner sep=2pt] (J) at (1.1,4.5) {};
		\node [circle,fill,inner sep=2pt] (K) at (1.9,4.8) {};
		\node [circle,fill,inner sep=2pt] (L) at (0.5,1.0) {};
		\node [circle,fill,inner sep=2pt] (M) at (1.3,2.4) {};
		\node [circle,fill,inner sep=2pt] (N) at (0.5,3.5) {};
		\node [circle,fill,inner sep=2pt] (O) at (1.6,1.7) {};
		\node [circle,fill,inner sep=2pt] (P) at (7.1,1.4) {};
		\node [circle,fill,inner sep=2pt] (Q) at (8.7,1.5) {};
		\node [circle,fill,inner sep=2pt] (R) at (5.0,2.7) {};
		\node [circle,fill,inner sep=2pt] (S) at (5.4,3.7) {};
		\node [circle,fill,inner sep=2pt] (T) at (6.4,2.3) {};
		\node [circle,fill,inner sep=2pt] (U) at (7.6,2.5) {};
		\node [circle,fill,inner sep=2pt] (V) at (6.7,3.3) {};
		\node [circle,fill,inner sep=2pt] (W) at (8.5,3.7) {};
		\node [circle,fill,inner sep=2pt] (X) at (6.6,4.4) {};
		\node [circle,fill,inner sep=2pt] (Y) at (7.5,4.7) {};
		\node [circle,fill,inner sep=2pt] (Z) at (4.3,4.4) {};
		\node [circle,fill,inner sep=2pt] (A) at (8.4,2.5) {};
	\end{tikzpicture}}%
	\hspace{.3cm}
    \subfloat[Our discretization approach]{
	\centering
	\begin{tikzpicture}[x=.65cm,y=.65cm]
		\coordinate (O) (0,0);
		\draw[help lines,xstep=2,ystep=2] (0,0) grid (10,6);
		\draw[black,line width=.5mm] (0,0) rectangle (10,6);
		\node [circle,fill,black!40,inner sep=2pt] (B) at (0.6,2.2) {};
		\node [circle,fill,black!40,inner sep=2pt] (C) at (0.7,5.2) {};
		\node [circle,fill,black!40,inner sep=2pt] (D) at (1.5,3.3) {};
		\node [circle,fill,black!40,inner sep=2pt] (E) at (2.7,2.6) {};
		\node [circle,fill,black!40,inner sep=2pt] (F) at (2.2,3.4) {};
		\node [circle,fill,black!40,inner sep=2pt] (G) at (3.2,4.3) {};
		\node [circle,fill,black!40,inner sep=2pt] (H) at (2.7,1.2) {};
		\node [circle,fill,black!40,inner sep=2pt] (I) at (3.6,3.1) {};
		\node [circle,fill,black!40,inner sep=2pt] (J) at (1.1,4.5) {};
		\node [circle,fill,black!40,inner sep=2pt] (K) at (1.9,4.8) {};
		\node [circle,fill,black!40,inner sep=2pt] (L) at (0.5,1.0) {};
		\node [circle,fill,black!40,inner sep=2pt] (M) at (1.3,2.4) {};
		\node [circle,fill,black!40,inner sep=2pt] (N) at (0.5,3.5) {};
		\node [circle,fill,black!40,inner sep=2pt] (O) at (1.6,1.7) {};
		\node [circle,fill,black!40,inner sep=2pt] (P) at (7.1,1.4) {};
		\node [circle,fill,black!40,inner sep=2pt] (Q) at (8.7,1.5) {};
		\node [circle,fill,black!40,inner sep=2pt] (R) at (5.0,2.7) {};
		\node [circle,fill,black!40,inner sep=2pt] (S) at (5.4,3.7) {};
		\node [circle,fill,black!40,inner sep=2pt] (T) at (6.4,2.3) {};
		\node [circle,fill,black!40,inner sep=2pt] (U) at (7.6,2.5) {};
		\node [circle,fill,black!40,inner sep=2pt] (V) at (6.7,3.3) {};
		\node [circle,fill,black!40,inner sep=2pt] (W) at (8.5,3.7) {};
		\node [circle,fill,black!40,inner sep=2pt] (X) at (6.6,4.4) {};
		\node [circle,fill,black!40,inner sep=2pt] (Y) at (7.5,4.7) {};
		\node [circle,fill,black!40,inner sep=2pt] (Z) at (4.3,4.4) {};
		\node [circle,fill,black!40,inner sep=2pt] (A) at (8.4,2.5) {};
		\filldraw (1,1) circle (4pt);
		\filldraw (1,3) circle (8pt);
		\filldraw (1,5) circle (6pt);
		\filldraw (3,1) circle (2.5pt);
		\filldraw (3,3) circle (6pt);
		\filldraw (3,5) circle (2.5pt);
		\filldraw (5,3) circle (4pt);
		\filldraw (5,5) circle (2.5pt);
		\filldraw (7,1) circle (2.5pt);
		\filldraw (7,3) circle (6pt);
		\filldraw (7,5) circle (4pt);
		\filldraw (9,1) circle (2.5pt);
		\filldraw (9,3) circle (4pt);
	\end{tikzpicture}}%
    \\
    \subfloat[Graph of the original sample]{
	\centering
	\begin{tikzpicture}[x=.65cm,y=.65cm]
		\coordinate (O) (0,0);
		\draw[help lines,xstep=2,ystep=2] (0,0) grid (10,6);
		\draw[black,line width=.5mm] (0,0) rectangle (10,6);
		\node [circle,fill,inner sep=2pt] (B) at (0.6,2.2) {};
		\node [circle,fill,inner sep=2pt] (C) at (0.7,5.2) {};
		\node [circle,fill,inner sep=2pt] (D) at (1.5,3.3) {};
		\node [circle,fill,inner sep=2pt] (E) at (2.7,2.6) {};
		\node [circle,fill,inner sep=2pt] (F) at (2.2,3.4) {};
		\node [circle,fill,inner sep=2pt] (G) at (3.2,4.3) {};
		\node [circle,fill,inner sep=2pt] (H) at (2.7,1.2) {};
		\node [circle,fill,inner sep=2pt] (I) at (3.6,3.1) {};
		\node [circle,fill,inner sep=2pt] (J) at (1.1,4.5) {};
		\node [circle,fill,inner sep=2pt] (K) at (1.9,4.8) {};
		\node [circle,fill,inner sep=2pt] (L) at (0.5,1.0) {};
		\node [circle,fill,inner sep=2pt] (M) at (1.3,2.4) {};
		\node [circle,fill,inner sep=2pt] (N) at (0.5,3.5) {};
		\node [circle,fill,inner sep=2pt] (O) at (1.6,1.7) {};
		\node [circle,fill,inner sep=2pt] (P) at (7.1,1.4) {};
		\node [circle,fill,inner sep=2pt] (Q) at (8.7,1.5) {};
		\node [circle,fill,inner sep=2pt] (R) at (5.0,2.7) {};
		\node [circle,fill,inner sep=2pt] (S) at (5.4,3.7) {};
		\node [circle,fill,inner sep=2pt] (T) at (6.4,2.3) {};
		\node [circle,fill,inner sep=2pt] (U) at (7.6,2.5) {};
		\node [circle,fill,inner sep=2pt] (V) at (6.7,3.3) {};
		\node [circle,fill,inner sep=2pt] (W) at (8.5,3.7) {};
		\node [circle,fill,inner sep=2pt] (X) at (6.6,4.4) {};
		\node [circle,fill,inner sep=2pt] (Y) at (7.5,4.7) {};
		\node [circle,fill,inner sep=2pt] (Z) at (4.3,4.4) {};
		\node [circle,fill,inner sep=2pt] (A) at (8.4,2.5) {};
		\draw [black!50, line width=.2mm] (A) -- (U) -- (P) -- (T) -- (U) -- (V) -- (T) -- (R);
		\draw [black!50, line width=.2mm] (X) -- (V) -- (S) -- (X) -- (Y) -- (W) -- (A) -- (Q);
		\draw [black!50, line width=.2mm] (O) -- (E) -- (M) -- (F) -- (I) -- (E) -- (H) -- (O) -- (L) -- (B) -- (O) -- (M) -- (B);
		\draw [black!50, line width=.2mm] (F) -- (K) -- (C) -- (J) -- (D) -- (F) -- (M) -- (D) -- (E) -- (F) -- (G) -- (Z);
		\draw [black!50, line width=.2mm] (M) -- (N) -- (B) -- (D) -- (N) -- (J) -- (K) -- (G) -- (I) -- (R) -- (S) -- (Z) -- (I);
	\end{tikzpicture}}%
    \label{img:discretization:c}
	\hspace{.3cm}
    \subfloat[Graph of the discretized sample]{
	\centering
	\begin{tikzpicture}[x=.65cm,y=.65cm]
		\coordinate (O) (0,0);
		\draw[help lines,xstep=2,ystep=2] (0,0) grid (10,6);
		\draw[black,line width=.5mm] (0,0) rectangle (10,6);
		\draw[black!40,line width=2.4mm] (1,1) -- (1,3);
		\draw[black!40,line width=3.6mm] (1,3) -- (1,5);
		\draw[black!40,line width=1.8mm] (3,1) -- (3,3);
		\draw[black!40,line width=0.9mm] (3,3) -- (3,5);
		\draw[black!40,line width=0.6mm] (5,3) -- (5,5);
		\draw[black!40,line width=0.9mm] (7,1) -- (7,3);
		\draw[black!40,line width=1.8mm] (7,3) -- (7,5);
		\draw[black!40,line width=0.6mm] (9,1) -- (9,3);
		\draw[black!40,line width=0.6mm] (1,1) -- (3,1);
		\draw[black!40,line width=3.6mm] (1,3) -- (3,3);
		\draw[black!40,line width=0.9mm] (1,5) -- (3,5);
		\draw[black!40,line width=1.8mm] (3,3) -- (5,3);
		\draw[black!40,line width=0.3mm] (3,5) -- (5,5);
		\draw[black!40,line width=1.8mm] (5,3) -- (7,3);
		\draw[black!40,line width=0.6mm] (5,5) -- (7,5);
		\draw[black!40,line width=0.3mm] (7,1) -- (9,1);
		\draw[black!40,line width=1.8mm] (7,3) -- (9,3);
		\filldraw (1,1) circle (4pt);
		\filldraw (1,3) circle (8pt);
		\filldraw (1,5) circle (6pt);
		\filldraw (3,1) circle (2.5pt);
		\filldraw (3,3) circle (6pt);
		\filldraw (3,5) circle (2.5pt);
		\filldraw (5,3) circle (4pt);
		\filldraw (5,5) circle (2.5pt);
		\filldraw (7,1) circle (2.5pt);
		\filldraw (7,3) circle (6pt);
		\filldraw (7,5) circle (4pt);
		\filldraw (9,1) circle (2.5pt);
		\filldraw (9,3) circle (4pt);
	\end{tikzpicture}}%
    \\
    \subfloat[Normalized Cut of the original sample]{
	\centering
	\begin{tikzpicture}[x=.65cm,y=.65cm]
		\coordinate (O) (0,0);
		\fill[niceblue,opacity=0.25] (0,0) -- (0,6) -- (5.45,6) -- (5.45,4.992857) -- (4.448148,3.418518) -- (3.974178,1.759613) -- (4.9469055,0.2680782) -- (5.121739,0);
		\fill[nicegreen,opacity=0.3] (10,0) -- (10,6) -- (5.45,6) -- (5.45,4.992857) -- (4.448148,3.418518) -- (3.974178,1.759613) -- (4.9469055,0.2680782) -- (5.121739,0);
		\draw[red,dashed,line width=.7mm] (5.1,4.442857) -- (4.448148,3.418518) -- (4.2,2.55);
		\draw[black,line width=.5mm] (0,0) rectangle (10,6);
		\node [circle,fill,niceblue,inner sep=2pt] (B) at (0.6,2.2) {};
		\node [circle,fill,niceblue,inner sep=2pt] (C) at (0.7,5.2) {};
		\node [circle,fill,niceblue,inner sep=2pt] (D) at (1.5,3.3) {};
		\node [circle,fill,niceblue,inner sep=2pt] (E) at (2.7,2.6) {};
		\node [circle,fill,niceblue,inner sep=2pt] (F) at (2.2,3.4) {};
		\node [circle,fill,niceblue,inner sep=2pt] (G) at (3.2,4.3) {};
		\node [circle,fill,niceblue,inner sep=2pt] (H) at (2.7,1.2) {};
		\node [circle,fill,niceblue,inner sep=2pt] (I) at (3.6,3.1) {};
		\node [circle,fill,niceblue,inner sep=2pt] (J) at (1.1,4.5) {};
		\node [circle,fill,niceblue,inner sep=2pt] (K) at (1.9,4.8) {};
		\node [circle,fill,niceblue,inner sep=2pt] (L) at (0.5,1.0) {};
		\node [circle,fill,niceblue,inner sep=2pt] (M) at (1.3,2.4) {};
		\node [circle,fill,niceblue,inner sep=2pt] (N) at (0.5,3.5) {};
		\node [circle,fill,niceblue,inner sep=2pt] (O) at (1.6,1.7) {};
		\node [circle,fill,nicegreen,inner sep=2pt] (P) at (7.1,1.4) {};
		\node [circle,fill,nicegreen,inner sep=2pt] (Q) at (8.7,1.5) {};
		\node [circle,fill,nicegreen,inner sep=2pt] (R) at (5.0,2.7) {};
		\node [circle,fill,nicegreen,inner sep=2pt] (S) at (5.4,3.7) {};
		\node [circle,fill,nicegreen,inner sep=2pt] (T) at (6.4,2.3) {};
		\node [circle,fill,nicegreen,inner sep=2pt] (U) at (7.6,2.5) {};
		\node [circle,fill,nicegreen,inner sep=2pt] (V) at (6.7,3.3) {};
		\node [circle,fill,nicegreen,inner sep=2pt] (W) at (8.5,3.7) {};
		\node [circle,fill,nicegreen,inner sep=2pt] (X) at (6.6,4.4) {};
		\node [circle,fill,nicegreen,inner sep=2pt] (Y) at (7.5,4.7) {};
		\node [circle,fill,niceblue,inner sep=2pt] (Z) at (4.3,4.4) {};
		\node [circle,fill,nicegreen,inner sep=2pt] (A) at (8.4,2.5) {};
		\draw [black!50, line width=.2mm] (A) -- (U) -- (P) -- (T) -- (U) -- (V) -- (T) -- (R);
		\draw [black!50, line width=.2mm] (X) -- (V) -- (S) -- (X) -- (Y) -- (W) -- (A) -- (Q);
		\draw [black!50, line width=.2mm] (O) -- (E) -- (M) -- (F) -- (I) -- (E) -- (H) -- (O) -- (L) -- (B) -- (O) -- (M) -- (B);
		\draw [black!50, line width=.2mm] (F) -- (K) -- (C) -- (J) -- (D) -- (F) -- (M) -- (D) -- (E) -- (F) -- (G) -- (Z);
		\draw [black!50, line width=.2mm] (M) -- (N) -- (B) -- (D) -- (N) -- (J) -- (K) -- (G) -- (I) -- (R) -- (S) -- (Z) -- (I);
	\end{tikzpicture}}%
    \hspace{.3cm}
    \subfloat[Normalized Cut of the discretized graph]{
	\centering
	\begin{tikzpicture}[x=.65cm,y=.65cm]
		\coordinate (O) (0,0);
		\fill[niceblue,opacity=0.25] (0,0) -- (0,6) -- (4,6) -- (4,2) -- (5,1) -- (5,0);
		\fill[nicegreen,opacity=0.3] (10,6) -- (4,6) -- (4,2) -- (5,1) -- (5,0) -- (10,0);
		\draw[black!50,line width=2.4mm] (1,1) -- (1,3);
		\draw[black!50,line width=3.6mm] (1,3) -- (1,5);
		\draw[black!50,line width=1.8mm] (3,1) -- (3,3);
		\draw[black!50,line width=0.9mm] (3,3) -- (3,5);
		\draw[black!50,line width=0.6mm] (5,3) -- (5,5);
		\draw[black!50,line width=0.9mm] (7,1) -- (7,3);
		\draw[black!50,line width=1.8mm] (7,3) -- (7,5);
		\draw[black!50,line width=0.6mm] (9,1) -- (9,3);
		\draw[black!50,line width=0.6mm] (1,1) -- (3,1);
		\draw[black!50,line width=3.6mm] (1,3) -- (3,3);
		\draw[black!50,line width=0.9mm] (1,5) -- (3,5);
		\draw[black!50,line width=1.8mm] (3,3) -- (5,3);
		\draw[black!50,line width=0.3mm] (3,5) -- (5,5);
		\draw[black!50,line width=1.8mm] (5,3) -- (7,3);
		\draw[black!50,line width=0.6mm] (5,5) -- (7,5);
		\draw[black!50,line width=0.3mm] (7,1) -- (9,1);
		\draw[black!50,line width=1.8mm] (7,3) -- (9,3);
		\filldraw[niceblue] (1,1) circle (4pt);
		\filldraw[niceblue] (1,3) circle (8pt);
		\filldraw[niceblue] (1,5) circle (6pt);
		\filldraw[niceblue] (3,1) circle (2.5pt);
		\filldraw[niceblue] (3,3) circle (6pt);
		\filldraw[niceblue] (3,5) circle (2.5pt);
		\filldraw[nicegreen] (5,3) circle (4pt);
		\filldraw[nicegreen] (5,5) circle (2.5pt);
		\filldraw[nicegreen] (7,1) circle (2.5pt);
		\filldraw[nicegreen] (7,3) circle (6pt);
		\filldraw[nicegreen] (7,5) circle (4pt);
		\filldraw[nicegreen] (9,1) circle (2.5pt);
		\filldraw[nicegreen] (9,3) circle (4pt);
		\draw[red,dashed,line width=.7mm] (4,6) -- (4,2);
		\draw[black,line width=.5mm] (0,0) rectangle (10,6);
	\end{tikzpicture}}%
	\caption{Illustrative example showcasing our discretization with Normalized Cut (NCut) on $D=[0,5]\times[0,3]$. (a) Original sample of a fictitious point cloud. (b) Bin-wise accumulation by discretizing the (greyed out) original sample. The discretized sample is displayed as black nodes with their sizes proportional to the number of (raw) sample points lying in the bins. (c) (Unweighted) graph built out of the original sample using $t$-neighbourhood for $t=0.55$. (d) Similar graph on discretized samples but with weights proportional to the size of the districts where directly adjacent vertices are connected. (e) NCut of the graph of original sample. (f) NCut of the graph of discretized samples. In both (e) and (f), the underlying space was divided into two parts (marked in blue and green) using 1-nearest neighbourhood.}
	\label{img:discretization}
\end{figure}
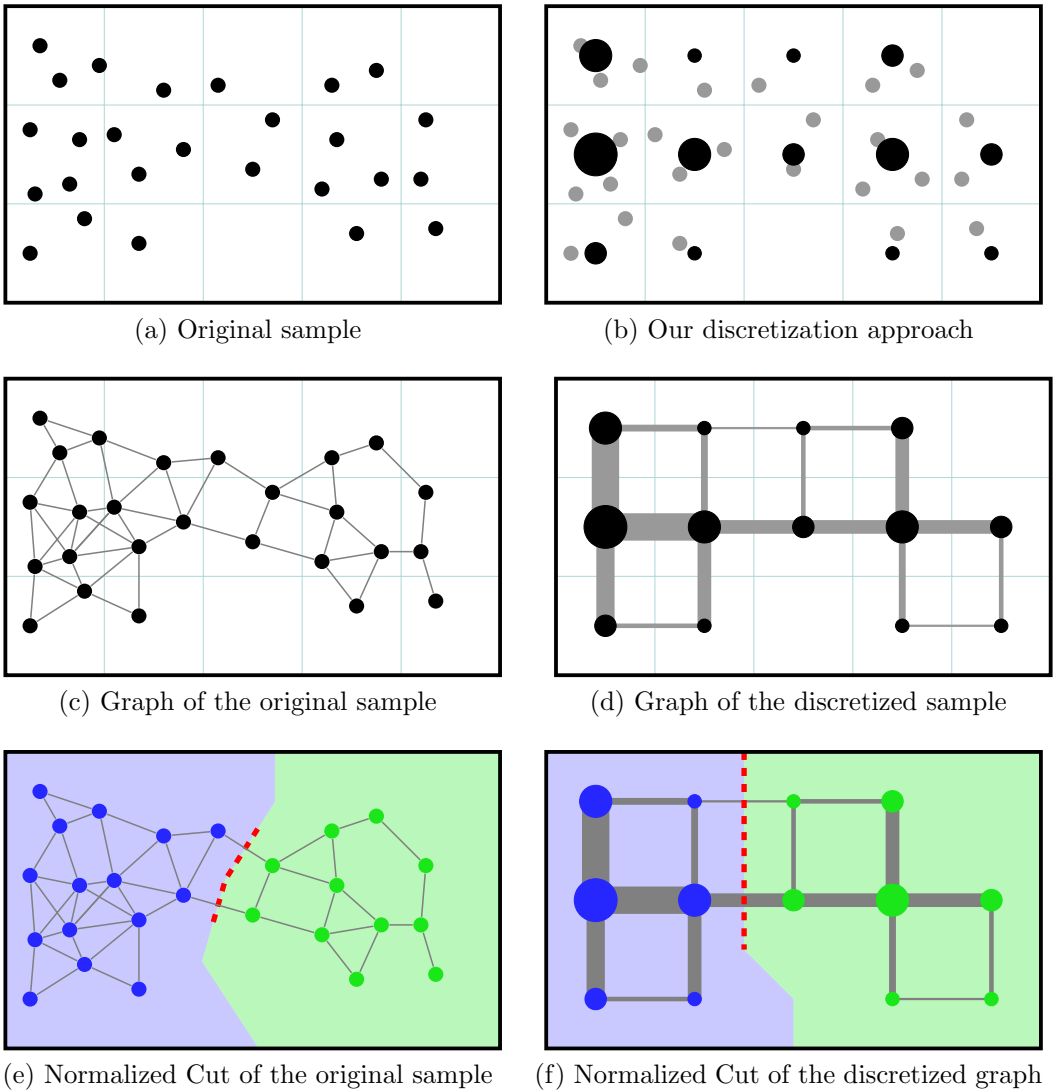

\cref{img:discretization} depicts a toy example to illustrate our approach on a subset of $\R^2$. Panels (c) and (e) display the usual approach of constructing a graph out of the point cloud in (a), taking the points as nodes and connecting them -- in this case using $t$-neighborhood -- to receive the graph shown in (c). Then in (e), the graph is partitioned into a blue and a green region by computing its NCut, where, for illustrative purposes, we divide the underlying space by assigning each point the partition its nearest-neighboring node belongs to. Compare this to the approach we will utilize in the following as depicted in (b), (d) and (f): First, as in (b), we divide the underlying space into disjoint subsets (here, a $3\times5$ grid). For each of these subsets, we accumulate the (greyed out) original data points inside to obtain new (black) nodes located in the center of the grid cells. When constructing the graph in (d), we connect those nodes that are no more than $t$ apart (which here, for $t=0.55$, reduces to direct horizontal and vertical adjacency). Additionally, each edge is assigned a weight (indicated by its thickness) proportional to the number of observations in the two nodes it connects. Finally, in (f) we partition the resulting graph by computing its NCut and, again for illustration, divide the underlying space using $1$-neighborhood analogous to (e). Note that NCut, CCut and RCut yield the same partition in this example, both in (e) and (f). To debias the results from possible numerical errors, in this small example we computed the graph cuts through brute-force iteration over all partitions. We emphasize that our results apply in much more general situations, see \cref{sec:modelling}.
%
%

The impact of discretizing a sample on the statistical properties of the resulting graph cut has not been explored in much detail in the literature. Existing discretization approaches have been focused exclusively on utilizing random subsampling in order to approximate the graph. For instance, \citet{FetayaShamirUllman2015} investigated the impact of vertex subsampling on the Laplacian for spectral clustering. Somewhat complimentary, \citet{TsayLovejoyKarger1999} preserved the set of vertices and instead constructed a graph out of a random set of edges in order to approximate graph cuts. Both approaches, however, only considered the impact of subsampling on the sample graph, and not its relation to the underlying distribution. When one seeks to analyze the limiting behaviour of graph cuts, subsampling merely introduces another random component leading to more complications. In contrast, we take a completely different perspective on discretization by superimposing a fixed structure on top of any point cloud, taking into account the sample as a whole. Namely, we construct a new graph out of a predefined set of vertices and edges, where the sample itself merely influences the weighting of the edges. Utilizing this type of discretization has multiple distinct advantages:

\begin{enumerate}
    \item The resulting fixed, discrete structure avoids the need for any condition on the support of the underlying distribution such as boundary smoothness or certain manifold requirements.
    \item If chosen properly, our discretization provides a simple multinomial model without any requirement from the underlying distribution itself.
    \item The control over the discretization remains with the user, thus allowing for much flexibility depending on the context, e.g.\ including prior information about the underlying distribution.
    \item It often arises naturally in practical applications, e.g.\ as image pixels or voxels.
\end{enumerate}

\subsection{Informal statement of the main theorem}

We consider a general setup of observing a sample from an arbitrary probability measure on a metric space. The asymptotic distributions of the objective values of the discretized empirical graph cuts are characterized as minima over normal variates indexed in all possible partitions.
%
In their original form, the \emph{balanced graph cuts} (sometimes called \emph{sparsest cuts} to distinguish them from cuts with enforced balancing constraints, e.g.\ bipartitions) can be expressed as
\[
\XC(\mathcal{G}):= \min_{\emptyset \neq S\subsetneq V} \sum_{{i\in S,\,j\in\comp{S},\;\{i,j\}\in E}} \frac{w_{ij}}{\bal(S,\comp{S})}
\]
for a graph $\mathcal{G}=(V,E)$ with nodes in $V$ and edges $\{i,j\}$ in $E$ that have positive weights $w_{ij}$, and for a \emph{balancing term} $\bal(S,\comp{S})$. The balancing terms of the most well-known cuts, e.g.\ Minimum Cut (MCut), Ratio Cut (RCut), Normalized Cut (NCut) and Cheeger Cut (CCut), are listed in \cref{tbl:cuts}.

\begin{table}
	\caption{\label{tbl:cuts} Definition of the most commonly used (balanced) graph cuts.}
	\centering
	\begin{tabular}{ccccc}
		\toprule
		Cut name & Short & $\XC$ & $\bal(S,\comp{S})$ & Reference\\
		\midrule
		Minimum Cut & MCut & $\MC$ & $1$ & E.g.~\cite{Cook_etal1998}\\
		Ratio Cut & RCut & $\RC$ & $\abs{S}|\comp{S}|$ & \cite{HagenKahng1992}\\
		Normalized Cut & NCut & $\NC$ & $\vol(S)\,\vol(\comp{S})$ & \cite{ShiMalik2000}\\
		Cheeger Cut & CCut & $\CC$ & $\displaystyle\min\hspace{-0.03cm}\big\{\hspace{-.08cm}\vol(S),\,\vol(\comp{S})\big\}$ & \cite{Cheeger1971}\\
		\bottomrule
	\end{tabular}
\end{table}

Informally, the main theorem of this paper can be stated as follows:

\begin{informal}
	\label{thm:main_informal}
	Let $\mathcal{G}$ be the $t$-neighbourhood graph built from the probability vector $\vec{p}$ obtained from discretizing the distribution $\nu$, and let $\mathcal{G}_n$ be its empirical counterpart obtained by discretizing i.i.d.\ samples $X_1,\ldots X_n$ drawn from $\nu$. Let $\mathcal{S}_{\vec{p}}$ be the subset of partitions attaining $\XC(\mathcal{G})$, i.e.\ the minimum XCut of $\mathcal{G}$, where XCut stands for any balanced graph cut with Hadamard directionally differentiable objective function.
	\begin{enumerate}[label=(\roman*)]
	\item Then, for the empirical graph cut $\XC(\mathcal{G}_n)$ of $\mathcal{G}_n$ it holds
	\begin{equation}
	\sqrt{n}\bigl(\XC(\mathcal{G}_n) - \XC(\mathcal{G})\bigr)\dto\min_{S\in\mathcal{S}_{\vec{p}}}Z_S^{\XC}, \label{eq:main_informal}
	\end{equation}
	where $Z_S^{\XC}$ is a random mixture of Gaussians. \label{thm:main_informal:hadamard}
	\item If the XCut objective funtion has a linear Hadamard derivative (incl.\ MCut, RCut and NCut), $\vec{Z}_S$ in \eqref{eq:main_informal} satisfies $(Z_S^{\XC})_{S\in\mathcal{S}_{\vec{p}}} \sim\mathcal{N}(0,\mat{\Sigma}^{\XC})$ with an explicit covariance matrix $\mat{\Sigma}^{\XC}$. For CCut, the same is true if additionally $\vol(S)\neq\vol(\comp{S})$ for all partitions $S\in\mathcal{S}_{\vec{p}}$. \label{thm:main_informal:diff}
	\end{enumerate}
\end{informal}

The detailed assumptions of this theorem are stated in \cref{sec:modelling}. We also obtain explicit expressions (see \cref{thm:conv_mrcut_S,cor:mrnccut_explicit}) for the aforementioned limiting distributions $Z_S^{\XC}$. An important finding of our work is the dependency of the asymptotic distribution of CCut on partition volume equality. While MCut, RCut and NCut are asymptotically normal when the partition attaining the minimum graph cut value is unique, the asymptotic distribution of CCut additionally depends on knowledge about the volume of the partition minimizing the cut value of the underlying distribution, i.e.\ if this minimizing partition $S_*$ is unique and satisfies $\vol(S_*) = \vol(\comp{S_*})$. In that case, CCut is asymptotically distributed according to a weighted sum of normals where the weighting coefficients are random themselves (see \cref{cor:mrnccut_explicit}). 

We display the QQ plots (of empirical distributions of $n$ samples against asymptotic distributions) of a discretized uniform distribution onto a $3\times 3$ grid in \cref{fig:intro_ccut_unif}. The dashed lines represent (quantiles of) asymptotic distributions obtained under false assumptions, i.e.\ by wrongly assuming either that the volumes of the minimizing partition are not equal or that the partition attaining the minimal graph cut value is unique. If one operates under false (and in general unknown) assumptions on the underlying distribution, the empirical CCut does not behave as expected in the limit. 

\begin{figure}[!htp]
    \centering
    \includegraphics[width=.6\linewidth]{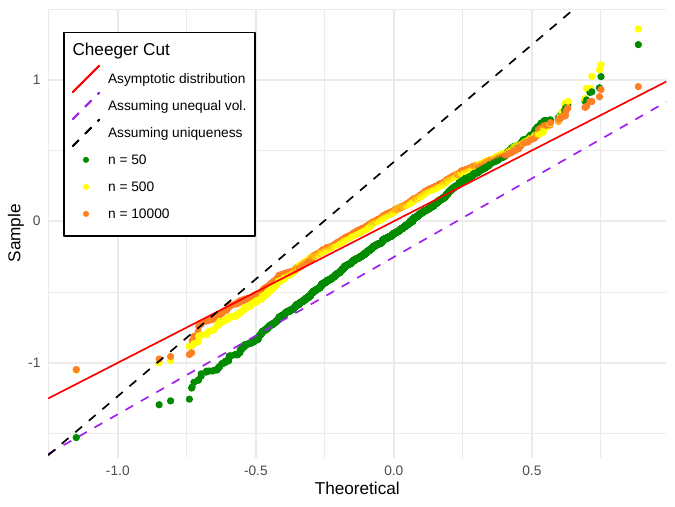}
    \caption{QQ plots of empirical distributions of $\sqrt{n}(\CC(\mathcal{G}_n) - \CC(\mathcal{G}))$ against theoretical limiting distributions for various sample sizes $n$. False assumptions of unequal volumes and of partition uniqueness on the data distribution (which is uniform in this case) lead to wrong limits as plotted by the two dashed lines in purple and black, respectively. For details, see \cref{img:qq_cbimodal_xcut_overn}.}
    \label{fig:intro_ccut_unif}
\end{figure}

\subsection{Outline}
We introduce the modelling and basic assumptions in \cref{sec:modelling}, where we also motivate and justify our discretization approach. \cref{sec:main_results} presents the main theorem (\cref{cor:conv_xcut}) of this paper that characterizes the asymptotic behaviour of general balanced graph cuts (incl.\ MCut, RCut, NCut and CCut). We show bootstrap consistency, i.e.\ the bootstrap estimator attaining the same limiting distribution as the empirical graph cut. In \cref{sec:applications} we extend our findings to the \emph{Xist} algorithm, a computationally efficient surrogate for the balanced graph cuts proposed in \cite{SuchanLiMunk2023arXiv}, and to multiway cuts. In particular the extension to \emph{Xist} makes our limit theorems usable in practical applications, as it is computable in low polynomial time. \cref{sec:simulations} then proceeds to demonstrate via simulations that the distributions we determined previously are indeed attained in the limit, and it outlines the use of our limit theorems for testing of clustering structure. \cref{sec:outlook} concludes the paper with discussions. The technical details and all the proofs are given in the \nameref{appendix}.

\section{Modelling and assumptions}
\label{sec:modelling}

%

By $\mathcal{G} = (V, E, \mat{W})$ we denote a simple weighted undirected graph. Here $V$ is the set of nodes, $E$ the set of edges, and $\mat{W}=(w_{ij})_{i,j\in V}$ the weight matrix with $w_{ij}$ equal to the weight of edge $\{i,j\}$ if $\{i,j\} \in E$ and $w_{ij}=0$ otherwise. Clearly, $\mat{W}$ is symmetric and has zero diagonals since a simple graph disallows self-loops, i.e.\ $\{i,i\} \not\in E$. An edge $\{i,j\} \in E$ is also denoted as $i \sim j$.

\begin{definition}
	\label{def:cuts}
	For a graph $\mathcal{G}=(V,E,\mat{W})$, the \emph{balanced graph cut} $\XC$ of the graph $\mathcal{G}$ is
	\[
 \XC(\mathcal{G})\coloneqq \min_{S\in\mathcal{S}}\XC_S(\mathcal{G})\quad \text{where}\quad\XC_S(\mathcal{G}) \coloneqq  \sum_{{i\in S, j\in \comp{S}}}\frac{w_{ij}}{\bal(S,\comp{S})}\quad \text{for}\,\,S\in\mathcal{S}.
	\]
	Here $\comp{S}\coloneqq V\setminus S$ denotes the complement of $S$, $\mathcal{S}\coloneqq \{S\subseteq V: V\neq S\neq\emptyset\}$ is the set of partitions of $V$,  and $\XC$ and $\bal(S,\comp{S})$ serve as placeholders for the graph cut and its corresponding balancing term, respectively.
\end{definition}

\cref{tbl:cuts} lists the balancing terms for MCut, RCut, NCut and CCut. In general, the balancing term can depend on the underlying graph structure, the partition $S$ as well as the weights $w_{ij}$. Further, for any partition $S\in\mathcal{S}$ define
\[
\abs{S} \coloneqq  \#\{i\in S\}\quad\text{and}\quad \vol(S)\coloneqq \sum_{i\in S}\sum_{j\in V} w_{ij}.
\]
To avoid some balancing terms becoming infinitely large we assume that $\abs{V}<\infty$. This assumption has no further restriction on our model since it comes naturally by means of the discretization of the underlying space into $m$ subsets (see \ref{assumptions:centers} below). This will become clear as this section progresses.
Note that while we use the balancing term $\bal(S,\comp{S})=\vol(S)\vol(\comp{S})$ for NCut, some authors use $\bal^*(S,\comp{S})\coloneqq (\vol(S)^{-1} + \vol(\comp{S})^{-1})^{-1}$ instead. This difference does not impact the resulting partition since $\bal^*(S,\comp{S})=\bal(S,\comp{S})/\vol(V)$, and $\vol(V)$ does not depend on $S\in\mathcal{S}$. It does, however, influence the cut value as well as the limiting distribution that we present later in \cref{thm:conv_mrcut_S,cor:conv_xcut} by a scaling factor of $\vol(V)$, see  \cref{cor:conv_ncut_alt} in the \nameref{appendix}.

%
We now specify the assumptions which our model requires and then the discretization itself. Throughout this paper, we make the following assumptions:
\begin{enumerate}[label={(M\arabic*)}]
    \item $\vec{X}\coloneqq (X_1,\ldots,X_n)$ is a vector of i.i.d.\ samples from a probability distribution $\nu$ on a metric space $(\mathcal{D},\dist)$, equipped with the Borel $\sigma$-algebra, with finite second moment. \label{assumptions:iid}
    \item The support of $\nu$, denoted by $D$, can be divided into finitely many (fixed) disjoint measurable sets $D_i$, i.e.\ $D=\sum_{i=1}^m D_i \subseteq \mathcal{D}$, $m\in \N$. \label{assumptions:centers}
    \item The balancing terms $\bal(S,\comp{S})$ are Hadmard directionally differentiable at the underlying discretized probability vector $\vec{p}\coloneqq \bigl(\nu(D_i)\bigr)_{i=1,\ldots,m}$. \label{assumptions:diff}
\end{enumerate}

As a natural assumption, \ref{assumptions:iid} justifies treating the number of sample points that fall into each $D_i$ as a sample from a multinomial distribution, which is the backbone of our model. The requirement on the second moment is necessary for a distributional limit, and the metric structure of the support $D$ of $\nu$ allows quantification of the distance between discretization districts. In \ref{assumptions:centers} we assume the existence of a discretization of $D$ into finitely many $D_i$. Typical examples include a regular partition by equally spaced grids (e.g.\ image pixels) and a Voroni tessellation using a finite set of points.
In \ref{assumptions:diff} we require the Hadamard directional differentiability of the balancing terms at $\vec{p}$. Here we view  $\bal(S,\comp{S})$ as functions $\R^m\to\R$ that depend on the discretization, see also the discussion at the end of this section. We stress that the Hadamard directional differentiability is a rather weak form of differentiability, see Appendix~\ref{apdx:sub:hadamard}. In fact, \ref{assumptions:diff} is satisfied for all graph cuts under consideration.

In the following we further specify our discretization model given data $X_1,\ldots, X_n$. This will also shed more light on why we need certain assumptions.

\begin{definition}
\label{def:discretization}
Under \ref{assumptions:iid}--\ref{assumptions:centers}, we define the \emph{discretized sample} $\vec{Y}=(Y_1,\ldots,Y_m)$ by
\[
Y_i\coloneqq \babs{\big\{j : X_j\in D_i,\; j = 1,\ldots,n\big\}}.
\]
We define the corresponding \emph{empirical discretized graph} $\mathcal{G}_n\coloneqq (V,E,\widehat{\mat{W}})$ as follows:
\begin{enumerate}[label={\roman*.}]
    \item The set of nodes is $V\coloneqq \{1\ldots,m\}$ with $i\in V$ denoting the index of $D_i$. 
    \item The edges are defined by a $t$-neighbourhood structure for some $t > 0$. That is, $\{i,j\}\in E$ if and only if $\mathrm{diss}(D_i,D_j)\leq t$. Here, $\mathrm{diss}(D_i,D_j)$ is a certain measure of dissimilarity between $D_i$ and $D_j$, for instance, the distance between two arbitrarily picked points (e.g.\ geometric centers) in $D_i$ and $D_j$.
    \item The weight matrix $\widehat{\mat{W}}=(\widehat{w}_{ij})_{i,j=1,\ldots,m}$ consists of entries
	\[
	\widehat{w}_{ij}\coloneqq \frac{Y_i}{n} \frac{Y_j}{n} \mathbbm{1}_{i\sim j}\qquad \text{ with }\quad\mathbbm{1}_{i\sim j} = \begin{cases}
 1 &\text{ if } \{i,j\}\in E,\\
 0 &\text{otherwise.}
 \end{cases}
	\]
\end{enumerate}
\end{definition}

As a discretization of $\vec{X}=(X_1,\ldots,X_n)$ onto the districts $D_1,\ldots,D_m$, $\vec{Y}=(Y_1,\ldots,Y_m)$ is a sample from $\mult(n,\vec{p})$ for $\vec{p}=(p_1,\ldots,p_m)$ with $p_i = \nu(D_i)$. 
Our choice of weights $\widehat{w}_{ij} = Y_i Y_j n^{-2} \mathbbm{1}_{i\sim j}$ can be seen as an \enquote{importance weighting}, where the strength of a connecting edge is proportional to the \enquote{importance} (i.e.\ probability mass) of both vertices involved. To illustrate, consider the simple scenario that every observation in $D_i$ connects with every observation in $D_i$ if $\mathrm{dis}(D_i,D_j) < t$ in the graph built from the original sample. Then, the weight $\widehat{w}_{ij}$ is proportion to the total number of connections between observations in $D_i$ and those in $D_j$, see \cref{img:discretization}. The choice in scaling (in terms of $n$) serves as a standardization, noting that
\begin{equation}
\label{eq:unbiased_weights}
    \tfrac{n}{n-1}\mathbb{E}\bigl[\widehat{w}_{ij}\bigr] = p_ip_j\mathbbm{1}_{i\sim j} \quad\text{for all}\quad i,j\in V.
\end{equation}
Moreover, a benefit of choosing $\widehat{w}_{ij}$ in this way is that it connects $D_i$ and $D_j$ more strongly if $Y_i$ and $Y_j$ have similar values. More precisely, the function $f:(y_i,y_j)\mapsto y_i y_j$ is minimized by $y_i=y_j$ under the condition that $y_i+y_j$ is constant. This means that if both districts contain, say, $10$ observations in total, then $D_i$ and $D_j$ will be connected most closely if $Y_i=Y_j=5$ (with $\widehat{w}_{ij}=25/n
^2$). If, on the other hand, $Y_i=1$ and $Y_j=9$, $\widehat{w}_{ij}=9/n^2$ will result in a much weaker connection. As an analogy in a different context, the product structure in $\widehat{w}_{ij}$ resembles the Physics laws that characterize interactions between objects, such as Coulomb’s law or Newton’s law of universal gravitation, where probability mass represents the charge or the mass of the objects involved. This formal similarity coincides with the intuition that $\widehat{w}_{ij}$ measures the degree of connectivity between $D_i$ and $D_j$.

\begin{definition}
\label{def:population_graph}
	We define the \emph{discretized population graph} $\mathcal{G}=(V,E, \mat{W})$ with the same $V$ and $E$ as in the discretized graph $\mathcal{G}_n$ from \cref{def:discretization}, but a different weight matrix  $\mat{W}=(w_{ij})_{i,j=1,\ldots,m}$, which is defined as $w_{ij}\coloneqq p_i p_j \mathbbm{1}_{i\sim j}$ for the probability vector $\vec{p} = (p_1, \ldots, p_m)\in[0,1]^m$ and $p_i=\nu(D_i)$.
\end{definition}

We are mainly interested in the relation between the empirical graph cuts $\XC(\mathcal{G}_n)$ and its population counterpart $\XC(\mathcal{G})$, where the only difference lies in the associated weight matrices. The empirical graph $\mathcal{G}_n$ is an empirical version of $\mathcal{G}$ using the plug-in estimator $\vec{Y}/n$ of $\vec{p}$. Note that the XCut objective function (in \cref{def:cuts}) can be seen as a functional of the weight matrix. Thus, for every partition $S \in \mathcal{S}$, there exists a functional $g_S^{\XC}:\R^m\to\R$ such that 
\begin{subequations}
\label{eq:xcut_functional}
\begin{equation}
   g_S^{\XC}\Bigl(\tfrac{\vec{Y}}{n}\Bigr) = \XC_S(\mathcal{G}_n)\quad \text{and} \quad g_S^{\XC}(\vec{p}) = \XC_S(\mathcal{G}). 
\end{equation}
More precisely, it takes the form of
\begin{equation}
    g_S^{\XC}(x_1,\ldots,x_m) \coloneqq \sum_{i\sim j,\, i\in S, \,j\in \comp{S} } \frac{x_i x_j}{\ball_S^{\XC}(x_1,\ldots,x_m)},
\end{equation}
\end{subequations}
where $\ball_S^{\XC}:\R^m\to\R$ encodes the balance term $\bal(S,\comp{S})$. For instance,  in case of NCut, 
\[
 \ball_S^{\NC}(x_1,\ldots,x_m) = \left(\sum_{i\in S}\sum_{j: j\sim i} x_i x_j\right)\left(\sum_{k\in \comp{S}}\sum_{l: l\sim k} x_j x_l\right).
\]
Thus, \ref{assumptions:diff} actually requires the existence of the  Hadamard directional derivative of the balance functional $\ball_S^{\XC}(\cdot)$ at $\vec{p}$, which we denote by $\bigl(\ball_S^{\XC}\bigr)_{\vec{p}}'$. 

\section{Main results}
\label{sec:main_results}


\subsection{Limiting distributions for graph cuts}

We start with the asymptotic distributions of graph cuts over all partitions.

\begin{theorem}
\label{thm:conv_mrcut_S}
Under \firstassumptions, fix a distance threshold $t > 0$. For $S \in \mathcal{S}$, define $\vec{q}_S\in\R^m$ as
\begin{equation}\label{e:dqiS}
(\vec{q}_S)_i\coloneqq q_{i,S} \coloneqq 
\begin{cases}
\sum_{j \in \comp{S}, j\sim i} p_j& \text{ if } i \in S,\\
\sum_{j \in {S}, j\sim i} p_j& \text{ if } i \in \comp{S},
\end{cases}    
\end{equation}
and define the balance functional $\ball_S^{\XC}(\cdot)$ as in \eqref{eq:xcut_functional}. Then:
\begin{enumerate}[label=(\roman*)]
	\item As $n\to \infty$, it holds that\label{thm:conv_mrcut_S:Hadamard}
 
	\[
	\Bigl(\sqrt{n}\left({\XC}_S(\mathcal{G}_n) - \XC_S(\mathcal{G})\right)\Bigr)_{S\in\mathcal{S}}\; \dto\; \left(\frac{\scalprod{\vec{Z}}{\vec{q}_S} - \XC_S(\mathcal{G})\,\bigl(\ball_S^{\XC}\bigr)_{\vec{p}}'(\vec{Z})}{\bal(S,\comp{S})}\right)_{S \in \mathcal{S}},
	\]
 
	where $\vec{Z} \sim \mathcal{N}_m(\vec{0},\mat{\Sigma})$ and $\mat{\Sigma}\in \R^{m \times m}$ is the covariance matrix of $\mult(1,\vec{p})$.
 
	\item Further, if $\ball_S^{\XC}(\cdot)$ is differentiable at $\vec{p}$ for all $S\in \mathcal{S}$, then the limiting distribution above is the centered Gaussian $\mathcal{N}_{\abs{\mathcal{S}}}\big(\vec{0},\mat{\Sigma}^{\XC}\big)$ with $\mat{\Sigma}^{\XC} = \bigl({\Sigma}_{T,S}^{\XC}\bigr)_{T,S\in\mathcal{S}}\in\R^{\abs{\mathcal{S}}\times\abs{\mathcal{S}}}$ given by
	\[
	\Sigma_{T,S}^{\XC} \coloneqq  \sum_{i=1}^m p_i c_{i,S}c_{i,T} - \biggl(\sum_{i=1}^m p_i c_{i,S}\biggr)\biggl(\sum_{i=1}^m p_i c_{i,T}\biggr),
	\] 
	where $c_{i,S} \coloneqq  \bigl({q_{i,S} -\XC_S(\mathcal{G}) \, \partial_{p_i} \ball_S^{\XC}(\vec{p})}\bigr)/{\ball_S^{\XC}(\vec{p})}$ for $i \in V$ and $S \in \mathcal{S}$.\label{thm:conv_mrcut_S:diff}
\end{enumerate}
\end{theorem}

\cref{thm:conv_mrcut_S} presents a general result on the distributional limits of the $\R^{\abs{\mathcal{S}}}$-valued vector $({{\XC}}_S(\mathcal{G}_n))_{S\in\mathcal{S}}$ indexed by the set of partitions $\mathcal{S}$ under fairly weak assumptions. As a consequence, one can readily obtain distributional limits for the early discussed graph cuts (see \cref{tbl:cuts}).

\begin{corollary}
\label{cor:mrnccut_explicit}
Let \firstassumptions hold and fix a distance threshold $t > 0$ s.t.\ the discretized graphs $\mathcal{G}_n$ and $\mathcal{G}$ (in \cref{def:population_graph,def:discretization}) are connected.  Introduce $q_{i,S}$ as in \eqref{e:dqiS} for $i \in V$ and $S \in \mathcal{S}$. Then:
\begin{enumerate}[label=(\roman*)]
\item  In case of MCut, RCut and NCut, i.e.\ $\XC \in \{\MC, \RC, \NC\}$, it holds that
\[
\left(\sqrt{n}\bigl(\XC_S(\mathcal{G}_n) - \XC_S(\mathcal{G})\bigr)\right)_{S\in\mathcal{S}}\quad \dto\quad\mathcal{N}_{\abs{\mathcal{S}}}\big(\vec{0},\mat{\Sigma}^{\XC}\big).
\]
Here the covariance matrix $\mat{\Sigma}^{\XC} = \bigl({\Sigma}^{\XC}_{T,S}\bigr)_{T,S \in \mathcal{S}}$ has the following explicit forms
\begin{align*}
\Sigma_{T,S}^{\MC} &\; \coloneqq  \;\sum_{i=1}^m p_i q_{i,T} q_{i,S} - 4 \MC_T(\mathcal{G})\MC_S(\mathcal{G}), \\
\Sigma_{T,S}^{\RC} &\; \coloneqq  \; \frac{\sum_{i=1}^m p_i\,q_{i,T}\,q_{i,S}}{|T| |\comp{T}| |S| |\comp{S}|} - 4 \RC_T(\mathcal{G})\RC_S(\mathcal{G}), \\
\Sigma_{T,S}^{\NC} &\; \coloneqq  \;\frac{\displaystyle\sum_{i=1}^m p_i q_{i,T} q_{i,S} \left(1 - \NC_T(\mathcal{G}) r_{i,T}\right)\left(1 - \NC_S(\mathcal{G}) r_{i,S}\right)}{\vol(T)\vol(\comp{T})\vol(S)\vol(\comp{S})} - 4 \NC_T(\mathcal{G})\NC_S(\mathcal{G}),
\end{align*}

where 
$r_{i,R} \coloneqq \vol(R) + 3\vol(\comp{R})$ if $i \in R$, and  $r_{i,R} \coloneqq\vol(\comp{R}) + 3\vol(R)$ otherwise.
\label{cor:mrnccut_explicit:diff}

\item  In case of CCut, i.e.\ $\XC = \CC$, it holds that
\begin{multline}\label{eq:cor:mrncut_explicit:ccut}
    \left(\sqrt{n}\bigl(\CC_S(\mathcal{G}_n) - \CC_S(\mathcal{G})\bigr)\right)_{S \in \mathcal{S}}  \\
\dto
\left(\frac{\scalprod{\vec{Z}}{\vec{q}_S} -\CC_S(\mathcal{G})\min\limits_{T \in \argmin\{\vol(S), \vol(\comp{S})\}} \sum\limits_{i \in T}\sum\limits_{j\sim i}(p_i Z_j + p_j Z_i)}{\min\{\vol(S),\, \vol(\comp{S})\}}\right)_{S \in \mathcal{S}}
\end{multline}
for $\vec{Z} \sim \mathcal{N}_m(\vec{0},\mat{\Sigma})$, where $\mat{\Sigma}\in \R^{m \times m}$ is the covariance matrix of $\mult(1,\vec{p})$.\label{cor:mrnccut_explicit:cheeger}

\item For partitions of unequal volumes, i.e.\ $\mathcal{S}_{\mathrm{u}} \coloneqq \left\{S \in \mathcal{S}: \vol(S) < \vol(\comp{S})\right\}$, in particular it holds 
\[
\left(\sqrt{n}\left(\CC_S(\mathcal{G}_n) - \CC_S(\mathcal{G})\right)\right)_{S \in \mathcal{S}_{\mathrm{u}}}\quad \dto\quad\mathcal{N}_{\abs{\mathcal{S}_{\mathrm{u}}}}\big(\vec{0},\mat{\Sigma}^{\CC}\big),
\]

where the covariance matrix $\mat{\Sigma}^{\CC} = \bigl({\Sigma}^{\CC}_{T,S}\bigr)_{T,S \in \mathcal{S}_{\mathrm{u}}}$ is given as 
\[
{\Sigma}^{\CC}_{T,S} \coloneqq  
\sum_{i = 1}^m\frac{p_i\bigl(q_{i,T} -\CC_T(\mathcal{G})(q_{i,T} + 2b_{i,T}\mathbbm{1}_{i\in T})\bigr)\bigl(q_{i,S} -\CC_S(\mathcal{G})(q_{i,S} + 2b_{i,S}\mathbbm{1}_{i\in S})\bigr)}{\vol(T)\vol(S)}
\]

with $b_{i,S} \coloneqq  \sum_{j \in S,\, j\sim i}p_j$ for $i \in V$ and $T,S \in \mathcal{S}_{\mathrm{u}}$.\label{cor:mrnccut_explicit:cheegerunique}
\end{enumerate}
\end{corollary}

As $S$ and $\comp{S}$ give the same cut value, it is sufficient to consider $\mathcal{S}_{\mathrm{u}}$ in cases of partitions with unequal volumes. We stress that a distinct feature of CCut -- as opposed to MCut, RCut and NCut -- is its non-differentiability (in the usual sense) at $\vec{p}$ for partitions $S$ that satisfy $\vol(S) = \vol(\comp{S})$. However, it can be shown to be (nonlinear) Hadamard directionally differentiable (see Appendix~\ref{apdx:sub:hadamard}) which leads to the non-Gaussian limits in \eqref{eq:cor:mrncut_explicit:ccut}. More precisely, consider
\[
\mathcal{S}_{\mathrm{e}} \coloneqq  \bigl\{S \in \mathcal{S}:\vol(S) = \vol(\comp{S})\bigr\}.
\]
Then, \cref{cor:mrnccut_explicit}~(ii) implies that the limit in \eqref{eq:cor:mrncut_explicit:ccut} restricted to $\mathcal{S}_{\mathrm{e}}$
is a linear combination of a Gaussian and a minimum of two Gaussian distributions, thus being non-Gaussian. 
In addition, similarly to \cref{cor:mrnccut_explicit}, one can derive limiting distributions for other balanced graph cuts from \cref{thm:conv_mrcut_S}, e.g.\ with balancing term $\bal(S,\comp{S})=\min\{|S|,|\comp{S}|\}$ (considered in \cite{Mohar1989}, for instance).

Next we state the limiting distribution of graph cuts with the optimal partition(s).  
\begin{theorem}
	\label{cor:conv_xcut}
	Assume the same setup as in \cref{thm:conv_mrcut_S}, and let $\vec{Z}^{\XC} = (Z_S^{\XC})_{S\in\mathcal{S}}$ be the limiting distribution of $\bigl(\sqrt{n}\left(\XC_S(\mathcal{G}_n) - \XC_S(\mathcal{G})\right)\bigr)_{S\in\mathcal{S}}$ therein. Define $\mathcal{S}_{*}$ as the subset of partitions attaining $\min_{S\in\mathcal{S}}\XC_S(\mathcal{G}) = \XC(\mathcal{G})$. Then, as $n\to \infty$, $\sqrt{n}\bigl({\XC}(\mathcal{G}_{n}) - \XC(\mathcal{G})\bigr)  \dto  \min_{S\in\mathcal{S}_{*}}Z_S^{\XC}.$
\end{theorem}

Combining \cref{cor:mrnccut_explicit} and \cref{cor:conv_xcut} yields the limiting distribution of the optimal cut values for MCut, RCut, NCut and CCut. It further follows that for MCut, RCut and NCut, the limiting distribution of $\XC(\mathcal{G}_{n})$ is Gaussian if and only if the optimal partition is unique, i.e.\ $\abs{\mathcal{S}_{*}} = 1$. For CCut, the limiting distribution of $\CC(\mathcal{G}_{n})$ is  Gaussian if and only if the optimal partition, denoted as $S_*$, is unique and additionally satisfies that $\vol(S_*) \neq \vol(\comp{S}_*)$.


\subsection{Utilizing the bootstrap}
\label{sub:bootstrap}

In order to utilize the asymptotic distributions presented in \cref{sec:main_results} it is necessary to know the XCut value of the graph constructed out of the (discretized) underlying distribution $\vec{p}$. In practice, however, this is usually not possible apart from utilizing this assumption for testing as we will briefly discuss in \cref{sec:outlook}. The bootstrap provides a means to approximate the distribution of $\sqrt{n}(\XC(\mathcal{G}_n) - \XC(\mathcal{G}))$ without any knowledge of $\vec{p}$.
More precisely, let $\vec{Y}\coloneqq \vec{Y}_n\coloneqq (Y_1,\ldots,Y_m)$ be the sample obtained by our discretization, and let $M_n\in\N$ be the bootstrap sample size. That is, we draw $M_n$ i.i.d.\ samples from the \emph{empirical} distribution obtained from $\vec{Y}/n$, thus constituting the bootstrap sample $\vec{Y}^*\coloneqq \vec{Y}_{M_n}^*\coloneqq (Y_1^*,\ldots,Y_m^*)\sim\mult(M_n,\vec{Y}/n)$. With this, one can estimate the quantity of interest of
$\sqrt{n}\bigl(\XC(\mathcal{G}_n) - \XC(\mathcal{G})\bigr)$ by$ \sqrt{M_n}\bigl(\XC(\mathcal{G}_{M_n}^*) - \XC(\mathcal{G}_n)\bigr),$
where $\mathcal{G}_{M_n}^*$ is defined as the graph constructed from $\vec{Y}^*$ in the same way that $\mathcal{G}_n$ is constructed from $\vec{Y}/n$, i.e.\ using edge weights $\widehat{w}_{ij} \coloneqq  Y_i^* Y_j^* M_n^{-2} \mathbbm{1}_{i\sim j}$, see \cref{def:discretization}.

\begin{theorem}
    \label{thm:bootstrap_limit_xcut}
    Let \firstassumptions hold, fix $t>0$ and define the bootstrap samples as above, i.e.\ by sampling $\vec{Y}^*\sim\mult(M_n,\vec{Y}/n)$ out of $\vec{Y}$, and constructing $\mathcal{G}_{M_n}^*$ out of $\vec{Y}^*$. There exists a sequence of sets $(A_n)_{n\in\N}\subseteq\R^d$ with $\lim_{n\to\infty}\P(\vec{Y}\in A_n)= 1$ such that, if $\vec{Y}\in A_n$ for all large enough $n$, it holds:
    \begin{enumerate}[label=(\roman*)]
    \item Let $M_n\to\infty$ with $M_n=o(n)$  as $n\to\infty$. Then, the $M_n$-out-of-$n$ bootstrap converges toward the same limit as $\sqrt{n}({\XC}(\mathcal{G}_{n}) - \XC(\mathcal{G}))$ in \cref{cor:conv_xcut}, i.e.\ conditioned on $\vec{Y}$
	\begin{equation}
        \sqrt{M_n}\bigl({\XC}(\mathcal{G}_{M_n}^*) - \XC(\mathcal{G}_n)\bigr) \quad \dto \quad \min_{S\in\mathcal{S}_{*}}Z_S^{\XC},
        \label{eq:xcut_bootstrap_convergence}
	\end{equation}
	where $\vec{Z}^{\XC}$ is the limiting distribution of $\bigl(\sqrt{n}\left(\XC_S(\mathcal{G}_n) - \XC_S(\mathcal{G})\right)\bigr)_{S\in\mathcal{S}}$ from \cref{thm:conv_mrcut_S}, and $\mathcal{S}_{*}$ denotes the subset of partitions attaining $\min_{S\in\mathcal{S}}\XC_S(\mathcal{G}) = \XC(\mathcal{G})$. \label{thm:bootstrap_limit_xcut:Mn}
	\item If $M_n = n$, \eqref{eq:xcut_bootstrap_convergence} still holds if $\abs{\mathcal{S}_{*}} = 1$ and if the XCut balancing term is Hadamard differentiable with linear Hadamard derivative at $\vec{p}$. \label{thm:bootstrap_limit_xcut:n}
	\end{enumerate}
\end{theorem}

Simulations underpinning these theoretical results, particularly the difference between the $n$-out-of-$n$ bootstrap and the $M_n$-out-of-$n$ bootstrap procedures can be found in \cref{sub:bootstrap_simulations}.


\section{Extensions and variations}
\label{sec:applications}

We now present several extensions of the asymptotic results of \cref{sec:main_results}, and consider a computational surrogate (\xistref, introduced in \citealp{SuchanLiMunk2023arXiv}) for graph cuts that can be computed in polynomial time, more precisely in $\mathcal{O}(Nnm)$ for MCut, RCut, NCut and CCut, with $N$ the number of \emph{local maxima} (see \cref{def:locmax} below), thus opening the door for a plethora of applications of our limit theorems. 

\subsection{A fast algorithm to approximate balanced graph cuts via \texorpdfstring{$st$}{st}-min cuts}
\label{sub:xvst}

The disadvantage of balanced graph cuts is their computational complexity that is rooted in their combinatorial nature. It has been shown that the problem of computing NCut is NP-complete, see \cite[Appendix A, Proposition~1]{ShiMalik2000}; the idea behind their proof can be adapted for the other balanced cuts in \cref{tbl:cuts}, namely, RCut and  CCut (see \cite{SimaSchaeffer2006} for an alternative proof of the latter). Several other balanced graph cuts are also known to be NP-complete to compute, for instance, the graph cut with the balancing term $\bal(S,\comp{S})=\min\{|S|,|\comp{S}|\}$, see \cite{Mohar1989}. Thus, the balanced graph cuts become computationally infeasible even for problems of moderate sizes in practice. This remains the case even for relaxed versions of graph cuts \citep{WagnerWagner1993} and also for approximations of graph cuts \citep{BuiJones1992}. Moreover, the computation of balanced multiway graph cuts are also NP-complete \citep{Mohar1989}, even if the graph is a grid or tree \citep{AndreevRaecke2004}, or if the graph cut value is approximated by a finite factor \citep{FeldmannFoschini2012}. 

Spectral clustering is a popular method that circumvents this issue by relaxing the problem, but this approach can lead to qualitative drawbacks (e.g.\ \citealp{GuatteryMiller1998,NadlerGalun2007}). A different approach was presented in \cite{SuchanLiMunk2023arXiv}, where the \xistnameref (\textbf{X}Cut \textbf{i}mitation through $\mathbf{st}$-MinCuts) was introduced. In essence, this algorithm computes the minimum XCut value over a certain subset of partitions, namely that of $st$-MinCuts, and it only considers vertices $s$ and $t$ that satisfy a local maximality property, i.e.\ their degree being larger than that of their neighbors. These restrictions allow a version of our limit results to still apply to the \xistref algorithm, see \cref{thm:xist_limit} later. Further, by \cite[Theorem 3.2]{SuchanLiMunk2023arXiv}, \xistref is polynomial in time: Its computational complexity is $\BO(\abs{\Vloc}\max\{\abs{V}\abs{E}, \kappa\})$, where $\BO(\kappa)$ is set to be the complexity of computing $\XC_S(G)$ for a fixed partition $S\in\mathcal{S}$ and graph $G$, and $\Vloc\subseteq V$ is defined below. Note that for graph cuts such as RCut, NCut or CCut, $\kappa = \abs{V}\eqqcolon m$, meaning that in many scenarios, \xistref would be able to run faster than spectral clustering (which has a worst-case complexity of $\BO(\abs{V}^3)$, see \citealp{vonLuxburg2007}). Note that the polynomial complexity of \xistref does not contradict the NP-completeness of computing balanced graph cuts since the algorithm essentially minimizes the graph cut functional over a certain subset of partitions only (and not the full set of partitions $\mathcal{S}$).

In order to state the \xistref algorithm we first recall the notion of $st$-MinCuts, i.e.\ the partitions of minimal cut value that separate two nodes $s$ and $t$ in $V$. 

\begin{definition}
\label{def:stmincut}
    Given a graph $\mathcal{G}=(V,E,\mat{W})$, define  the \emph{$st$-MinCut} $\MC^{st}$ of $\mathcal{G}$, for distinct $s,t\in V$, as
    \begin{equation}
    \MC^{st}(\mathcal{G})\coloneqq \min_{S \in \mathcal{S}_{st}} \MC_S(\mathcal{G}), \quad\text{where}\quad \mathcal{S}_{st}\coloneqq \{S\in\mathcal{S}:  s \in S,\ t\not\in S\}. \label{eq:defSst}
    \end{equation}
\end{definition}

We also require the notion of \emph{local maxima}. For a general graph, this subset of $V$ would be defined as those vertices whose degree is larger than that of their neighbors. Our discretized graph, however, already provides us with a natural notion of \enquote{local maximality}, namely to choose those vertices $i$ (or, more precisely, districts $D_i$) which contain more observations $Y_j$ than their neighbors. 

\begin{definition}
\label{def:locmax}
    For the population graph $\mathcal{G}$, the \emph{set of local maxima} is defined as $\mathrm{V}_{\mathcal{G}}^{\mathrm{loc}} \coloneqq \bigl\{i\in V :  p_i\geq p_j\text{ for all } j\in V\text{ with } i\sim j\bigr\},$ 
    and for the empiricial graph $\mathcal{G}_n$ accordingly as
    $\mathrm{V}_{\mathcal{G}_n}^{\mathrm{loc}} \coloneqq \bigl\{i\in V :  Y_i\geq Y_j\text{ for all } j\in V\text{ with } i\sim j\bigr\}.$
\end{definition}

If the corresponding graph is clear from context or irrelevant, we often simply write $\Vloc$. We now may state \xistnameref below. Both \texttt{R} and \texttt{Python} implementations of the \xistref algorithm and related tools as well as application examples are available on Github \citep{Xist_Github}.


\begin{algorithm}[!htpb]
\label{alg:xist}
\SetAlgorithmName{Xist algorithm}{Xist}
\SetAlgoLined
\DontPrintSemicolon
\SetKwInOut{Input}{input}
\SetKwInOut{Output}{output}
\SetKw{KwTerminate}{terminate}
\Input{weighted graph $G=(V,E,\mat{W})$}
\Output{XC value $c_{\min} = \xist(G)$ and its associated partition $S_{\min}$}
\BlankLine

set $c_{\min}\leftarrow\infty$ and $S_{\min}\leftarrow\emptyset$\;
determine the set of local maxima $\Vloc\subseteq V$\label{alg:xist:vloc}\;
\lIf{$N\coloneqq \abs{\Vloc} = 1$}{\KwTerminate \label{alg:xist:cs}}
set $\tau\leftarrow (1,\ldots,1)\in\R^N$

\For{$i\in\{2,\ldots,N\}$\label{alg:xist:forloop}}{
    let $s$ denote the $i$-th, and $t$ the $\tau_i$-th vertex in $\Vloc$\;
    compute an $st$-MinCut partition $S_{st}$ on $G$ \label{alg:xist:st-mincut}
    \tcp*{By definition $s\in S_{st}$ and $t\in\comp{S}_{st}$}
    compute the XCut value $\XC_{S_{st}}(G)$ of partition $S_{st}$ \label{alg:xist:xcut}\;
	\If{$\XC_{S_{st}}(G) < c_{\min}$}{\label{alg:xist:ifmin}
        $c_{\min} \leftarrow \XC_{S_{st}}(G)$\;
        $S_{\min} \leftarrow S_{st}$\;
	}\label{alg:xist:ifmin_end}
	\For{$j\in\{i,\ldots,N\}$}{
	    let $v_j$ be the $j$-th vertex in $\Vloc$\;
	    \lIf{$v_j\in S_{st}$ and $\tau_j = \tau_i$}{$\tau_j\leftarrow i$}
	}
}
\caption{XCut imitation through $st$-MinCuts \citep{SuchanLiMunk2023arXiv}}
\end{algorithm}

The \xistref algorithm consists of three parts: First, the computation of the set of local maxima $\Vloc$. The intuition behind restricting \xistref to this subset of vertices is as follows: In the absence of a balancing term, an $st$-MinCut minimizes the sum of the weights of the cut edges in order to separate $s$ and $t$. If $s$ and $t$ are local maxima belonging to different and, say, neighboring clusters, the $st$-MinCut is can be expected to separate $s$ and $t$ by cutting between the two clusters since there, connections are not as tight as within clusters. An additional benefit of the restriction to $\Vloc$ is the reduction in computation time, see Appendix~\ref{apdx:sub:algorithm_proofs} for a brief discussion, or Section~4.2 in \cite{SuchanLiMunk2023arXiv} for an empirical comparison with related algorithms.

Second, \xistref computes $st$-MinCuts according to a selection procedure determined by $\tau$. In its original form first described in \cite{GomoryHu1961}, all $st$-MinCuts are computed by constructing (what is now called) a Gomory--Hu tree through selective vertex contraction, though it was later discovered to be usable without contractions entirely \citep{Gusfield1990}. In short, through $\tau$ the procedure ensures that for each distinct $st$-MinCut value, $s,t\in\Vloc$, at least one partition attaining this value is considered by the algorithm, while still only computing $\abs{\Vloc}-1$ $st$-MinCuts. For more details regarding the cut pair selection procedure of \xistref see \cite{SuchanLiMunk2023arXiv}, in particular Theorem 3.1 and the corresponding proof.

Third, \xistref outputs the minimum XCut among all $st$-MinCuts it computed. This gives rise to the intuition that \xistref acts as a computational surrogate of the original graph cuts by retaining computation of the exact XCut values, only not over all partitions $S\in\mathcal{S}$ but a certain subset.

Additionally, the design of the algorithm guarantees that the resulting partition is reasonable in the sense that it separates two vertices $s$ and $t$ through the $st$-MinCut partition $S_{st}$ while also taking cluster size into account via the balancing term in the XCut value $\XC_{S_{st}}(\mathcal{G})$. This makes the NP-complete problem of computing balanced graph cuts approachable from a qualitative perspective since the output of \xistref is guaranteed to be an $st$-MinCut for some $s,t\in\Vloc$. 

In order to state a limit theorem for \xistref one has to overcome one major theoretic complication, namely the possibility of nonuniqueness of the $st$-MinCuts. By design, \xistref only requires one partition that attains the $st$-MinCut value for each pair $s,t\in\Vloc$, $s\neq t$; this allows \xistref to run in polynomial time. Further, in the particular case of the empirical graph $\mathcal{G}_n$, it is not clear whether the event of multiple partitions attaining the $st$-MinCut can even occur with non-zero probability. For instance, if the limiting distribution of $(\sqrt{n}(\MC_S(\mathcal{G}_n) - \MC_S(\mathcal{G})))_{S\in\mathcal{S}}$ is non-degenerate, the optimal $st$-partition on $\mathcal{G}_n$ is almost surely unique (which can be shown similarly as \cref{lem:multinomial_components_distinct} in the \nameref{appendix}). In case of such a limiting distribution is degenerate, the optimal  $st$-partition on $\mathcal{G}_n$ may still be unique, as demonstrated later in \cref{exmp:xvst_practice}. In general, a precise condition on the uniqueness of optimal $st$-partitions is unknown and presents a direction for future research.

To deal with the technical complexity caused by the nonuniqueness of $st$-MinCuts, we require a mild restriction to establish a limiting distribution for \xistref. 

\begin{enumerate}[label={(A\arabic*)}]
\setcounter{enumi}{3}
\item For all $s,t\in V$, $s\neq t$, if $S, T \in \mathcal{S}_{st}$ satisfy $\MC_S(\mathcal{G}) = \MC_T(\mathcal{G}) = \MC^{st}(\mathcal{G})$ and $q_{i,S} = q_{i,T}$ for all $i\in V$, with $q_{i,S},q_{i,T}$ defined in \eqref{e:dqiS}, then it holds that $S = T$.
\label{assumptions:additional}
\end{enumerate}

To provide an intuition, this assumption essentially requires that, for any two vertices $s$ and $t$, no two distinct partitions $S$ and $T$ can attain the $st$-MinCut while inducing the same neighbourhood weight structure along the edges they cut. \aref{assumptions:additional} can be shown to be equivalent to a uniqueness condition on the components of the limiting distribution $\vec{Z}^{\MC}$ (see \cref{lem:q_equality_uniqueness} in the \nameref{appendix}). We stress that this assumption is weak, mainly to simplify the technical details. Constructing a scenario where \ref{assumptions:additional} fails to hold requires a situation where the graph is highly symmetrical. Even in such a case, however, we may still obtain a nondegenerate limit which requires a different scaling factor as \cref{exmp:xvst_practice} below demonstrates.

Under \aref{assumptions:additional}, we may already obtain a central limit theorem-type result for \xistref which requires substitution of the population graph cut value $\xist(\mathcal{G})$ for $\XC_{S_{*,n}}(\mathcal{G})$, where $S_{*,n}$ is the partition attaining $\xist(\mathcal{G}_n)$. The resulting asymptotic distribution is non-Gaussian in general (see \mysubref{thm:xist_limit}{without_uniqueness}). Under the following additional assumption (i.e.\ Assumption~1 in \citealp{SuchanLiMunk2023arXiv}), however, a Gaussian limit can be obtained (see \mysubref{thm:xist_limit}{with_uniqueness}). 

\begin{enumerate}[label={(A\arabic*)}]
\setcounter{enumi}{4}
\item There are $s,t\in \mathrm{V}_{\mathcal{G}}^{\mathrm{loc}}$, $s\neq t$, with a unique $st$-MinCut partition $S_{st}$ that attains the minimum of
\[
\XC_{S_{st}}(\mathcal{G})\; = \;\min_{u,v, S_{u,v}} \XC_{S_{uv}}(\mathcal{G}),
\]
where the minimum is over all $u\neq v\in \mathrm{V}_{\mathcal{G}}^{\mathrm{loc}}$, and \emph{all} optimal partitions $S_{uv}$ of the $uv$-MinCuts.
\label{assumptions:uniqueness}
\end{enumerate}

In short, \ref{assumptions:uniqueness} requires that the partition attaining the best XCut among all possible $st$-MinCuts, $s,t\in\Vloc$, also attains its own $st$-MinCut value uniquely for its pair $s,t\in\Vloc$, see also \cite{SuchanLiMunk2023arXiv}.

We now derive a limiting distribution for \xistref, similar to the theoretical findings in \cref{sec:main_results}. This is not straightforward: We cannot apply \cref{cor:conv_xcut} directly because the subset of $st$-MinCuts that is computed in line~\ref{alg:xist:st-mincut} of the \xistnameref depends on the graph and thus on the sample. We tackle such a dependence by a continuity argument (see Appendix~\ref{apdx:sub:algorithm_proofs}) and obtain the next theorem.


\begin{theorem}
    \label{thm:xist_limit}
    Assume the same setup as in \cref{thm:conv_mrcut_S}, and further \ref{assumptions:additional}. Introduce 
    \begin{align*}
    \mathcal{S}^{\MC} &\coloneqq \bigl\{(s,t,S)\in V\times V\times\mathcal{S} : \MC_S(\mathcal{G}) = \MC^{st}(\mathcal{G})\bigr\},\\
\text{and }\,    \mathcal{S}_*^{\MC} &\coloneqq \bigl\{(s,t,S)\in \mathcal{S}^{\MC} : \XC_S(\mathcal{G}) = \min_{(u,v,T)\in\mathcal{S}^{\MC}}\XC_T(\mathcal{G})\bigr\}.
    \end{align*}
    For $S\in\mathcal{S}$, denote the limiting distribution of $\bigl(\sqrt{n}\left(\XC_S(\mathcal{G}_n) - \XC_S(\mathcal{G})\right)\bigr)_S$ from \cref{thm:conv_mrcut_S} by $(Z_S^{\XC})_S$, and let $(Z_r^{\vec{p}})_{r \in V}$ be the limiting distribution of $\sqrt{n}(\vec{Y}/n - \vec{p})$. Introduce 
    \begin{align*}
    V_1 &\coloneqq \{r\in V: p_r > p_i\;\text{for all }i\sim r\}\quad\text{as well as}\\
    \widehat{V}_2 &\coloneqq \{r\in V\setminus V_1 : p_r > p_i,\text{ or } p_r = p_i \text{ and }Z_r^{\vec{p}} > Z_i^{\vec{p}},\; \text{ for all }i\sim r\}.
    \end{align*}
    \begin{enumerate}[label=(\roman*)]
        \item If $S_{*,n}$ and $\xist(\mathcal{G}_n)$ are the partition and XC value, respectively, of \xistref applied to $\mathcal{G}_n$, then
        \begin{align*}
        \sqrt{n}\bigl(\xvst^{\mathrm{loc}}(\mathcal{G}_n) - \XC_{S_n^{\mathrm{loc}}}(\mathcal{G})\bigr) \\
        \dto \min_{(s,t,S)\in\mathcal{S}_*^{\MC}} \bigl\{Z_S^{\XC} \mid\; &Z_S^{\MC} < Z_T^{\MC}\;\forall\,T\in\mathcal{S}_{st}\setminus\{S\},\text{ and, for }r=s,t: \\[-3mm]
        &\;r \in V_1\;\text{ or }\; (r\in V_2\,\text{ and }\, Z_r^{\vec{p}} > Z_i^{\vec{p}}\;\forall\,i\in V_2, i\sim r)\bigr\},
        \end{align*}
        \label{thm:xist_limit:without_uniqueness}
        \item If additionally \ref{assumptions:uniqueness} holds, $\sqrt{n}\bigl(\xist(\mathcal{G}_n) - \xist(\mathcal{G})\bigr)\dto \mathcal{N}\bigl(0,\Sigma_{S_*,S_*}^{\XC}\bigr)$, where the variance $\Sigma_{S_*,S_*}^{\XC}$ is given by \mysubref{cor:mrnccut_explicit}{diff}, for the (unique) minimizing partition $S_{\min}$ of $\xist(\mathcal{G})$.
        \label{thm:xist_limit:with_uniqueness}
    \end{enumerate}
\end{theorem}


It is important to point out that $Z_r^{\vec{p}}$, $Z_S^{\MC}$ and $Z_T^{\XC}$, $r\in V$, $S,T\in\mathcal{S}$, are all dependent as they arise from joint convergence of $\sqrt{n}(\vec{Y}/n - \vec{p})$, $\sqrt{n}(\MC_S(\mathcal{G}_n) - \MC_S(\mathcal{G}))$ and $\sqrt{n}(\XC_T(\mathcal{G}_n) - \XC_S(\mathcal{G}))$, taken as a vector over all such $r\in V$, $S,T\in\mathcal{S}$ simultaneously. A more precise formulation can be found in the proof of this corollary (cf.\ Appendix~\ref{apdx:sub:algorithm_proofs}, specifically \eqref{eq:xvst_loc_max_join_convergence}).


For \cref{thm:xist_limit} note that the extra requirement that $Z_r^{\vec{p}}$ and $Z_S^{\XC}$ converge jointly for all $r\in V$ and $S\in\mathcal{S}$ is natural and arises directly from \xistref. There, we use one sample $\vec{X}=(X_1,\ldots,X_n)$ to construct multiple statistics, namely $\vec{Y}$, $\MC_S(\mathcal{G}_n)$ and $\XC_S(\mathcal{G}_n)$. Consequently, the joint convergence assumption only asserts that one uses the same sample $\vec{X}$ to construct all these statistics.

Note that under \ref{assumptions:additional}, the set $\{Z_S^{\XC} : Z_S^{\MC} < Z_T^{\MC}\;\text{ for all }T\in\mathcal{S}_{st}\setminus\{S\}\bigr\}$ is almost surely non-empty (see \cref{lem:q_equality_uniqueness} in the \nameref{appendix}). Further note that the partition $S_{*,n}$, while dependent upon the sample $\vec{Y}$, is necessary in $\sqrt{n}\bigl(\xist(\mathcal{G}_n) - \XC_{S_{*,n}}(\mathcal{G})\bigr)$ to obtain a limiting distribution at all -- one could initially think that the \enquote{usual} population version $\xist(\mathcal{G})$ sufficed (see \mysubref{lem:xvst_theory}{without_uniqueness} in the \nameref{appendix}). The problem with that, however, is that it is not guaranteed that $S_{*,n}$ attains its corresponding $st$-MinCut uniquely, and if another partition of different XCut value also attained this $st$-MinCut, a limiting distribution would not exist.

\aref{assumptions:uniqueness} simplifies the asymptotic distribution significantly due to the partition $S_{*,n}$ attaining $\xist(\mathcal{G}_n)$ being fixed with probability approaching one (see \mysubref{lem:xvst_theory}{with_uniqueness} in the \nameref{appendix}). In practical applications, we expect this uniqueness condition (as well as the more technical \ref{assumptions:additional}) to be satisfied in almost all cases, especially if the underlying dataset contains (theoretically) non-discrete data, such as for image segmentation.

The following example shows that \ref{assumptions:additional} is, in a sense, necessary to ensure limiting distributions of the forms stated in \cref{thm:xist_limit}, and meanwhile serves as an illustration on how to derive limiting distributions in more general scenarios beyond \ref{assumptions:additional}. More precisely, we will see that the limiting distribution in \cref{thm:xist_limit} can be degenerate, but a nondegenerate limit can still be obtained by increasing the scaling from $\sqrt{n}$ to $n$. For simplicity, we only discuss this strategy and postpone the technical details in Appendix~\ref{apdx:sub:example}.

\begin{example}[Uniform graph]
\label{exmp:xvst_practice}
    Consider the graph $\mathcal{G}=(V,E,\mat{W})$ in \cref{img:xvst_theory_example}, where $V\coloneqq \{1,2,3,4\}$, $E\coloneqq\bigl\{\{1,2\},\{1,3\},\{2,4\},\{3,4\}\bigr\}$, and $\mat{W}$ is defined using the uniform probability vector $\vec{p} = \tfrac{1}{4}\cdot\mathbf{1} \in [0,1]^4$ as in \cref{def:population_graph}. As \ref{assumptions:additional} is violated, our theoretical results (\cref{thm:xist_limit}, and \cref{lem:xvst_theory} in the \nameref{appendix}) are not applicable. The question is whether it is still possible to derive an asymptotic distribution for $\xist(\mathcal{G}_n)$ as in \cref{thm:xist_limit}. In this simple scenario, the original RCut, NCut and CCut minimization problems are all solved by both the partition $S_{12}\coloneqq \{1,2\}$ and $S_{13}\coloneqq \{1,3\}$. In the following, we consider NCut due to its popularity in the literature, but similar results for the other cuts can be derived analogously.
    
    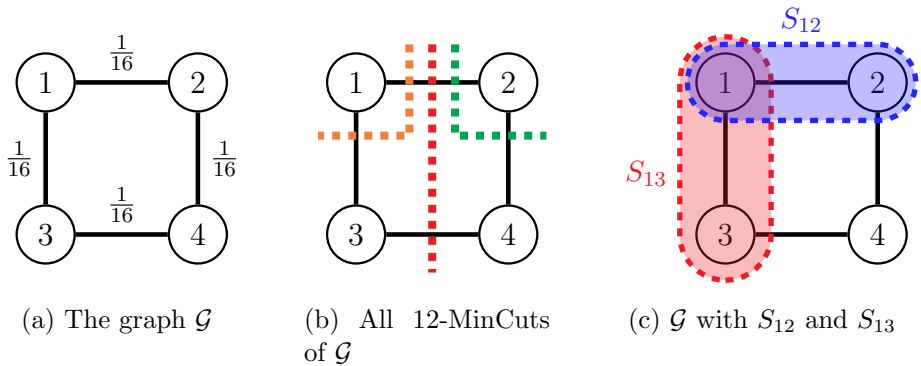
\begin{figure}[htpb]
	\centering
	\subfloat[The graph $\mathcal{G}$]{
	\centering
	\begin{tikzpicture}
    	\draw[draw=none] (1,-.6) circle (2mm);
    	\draw[line width=.3mm] (0,2) node[circle,draw,inner sep=1.5mm] (s) {\large $1$};
    	\draw[line width=.3mm] (2,2) node[circle,draw,inner sep=1.5mm] (t) {\large $2$};
    	\draw[line width=.3mm] (0,0) node[circle,draw,inner sep=1.5mm] (u) {\large $3$};
    	\draw[line width=.3mm] (2,0) node[circle,draw,inner sep=1.5mm] (v) {\large $4$};
    	\draw[line width=.6mm] (s) -- node[pos=.5,above] {$\tfrac{1}{16}$} (t) -- node[pos=.5,right] {$\tfrac{1}{16}$} (v) -- node[pos=.5,above] {$\tfrac{1}{16}$} (u) -- node[pos=.5,left] {$\tfrac{1}{16}$} (s);
	\end{tikzpicture}}
	\hspace{.6cm}
	\subfloat[All $12$-MinCuts of $\mathcal{G}$]{
	\centering
	\begin{tikzpicture}
    	\draw[draw=none] (1,-.6) circle (2mm);
    	\draw[line width=.3mm] (0,2) node[circle,draw,inner sep=1.5mm] (s) {\large $1$};
    	\draw[line width=.3mm] (2,2) node[circle,draw,inner sep=1.5mm] (t) {\large $2$};
    	\draw[line width=.3mm] (0,0) node[circle,draw,inner sep=1.5mm] (u) {\large $3$};
    	\draw[line width=.3mm] (2,0) node[circle,draw,inner sep=1.5mm] (v) {\large $4$};
    	\draw[line width=.6mm] (s) -- (t) -- (v) -- (u) -- (s);
    	\draw[dashed,line width=3pt,Red] (1,2.5) -- (1,-.5);
    	\draw[dashed,line width=3pt,Green] (1.3,2.5) -- (1.3,1.3) -- (2.5,1.3);
    	\draw[dashed,line width=3pt,Orange] (.7,2.5) -- (.7,1.3) -- (-.5,1.3);
	\end{tikzpicture}}
	\hspace{.6cm}
	\subfloat[$\mathcal{G}$ with $S_{12}$ and $S_{13}$]{
	\centering
	\begin{tikzpicture}
    	\draw[draw=none] (1,-.6) circle (2mm);
    	\draw[line width=.3mm] (0,2) node[circle,draw,inner sep=1.5mm] (s) {\large $1$};
    	\draw[line width=.3mm] (2,2) node[circle,draw,inner sep=1.5mm] (t) {\large $2$};
    	\draw[line width=.3mm] (0,0) node[circle,draw,inner sep=1.5mm] (u) {\large $3$};
    	\draw[line width=.3mm] (2,0) node[circle,draw,inner sep=1.5mm] (v) {\large $4$};
    	\draw[line width=.6mm] (s) -- (t) -- (v) -- (u) -- (s);
    	\filldraw[dashed,line width=2pt,draw=Red,fill=Red,fill opacity=.3] (0,-.6) to [out=0,in=270] (.6,0) -- (.6,2) to [out=90,in=0] (0,2.6) to [out=180,in=90] (-.6,2) -- node[pos=.6,left,Red,opacity=1] {$S_{13}$} (-.6,0) to [out=270,in=180] (0,-.6);
    	\filldraw[dashed,line width=2pt,draw=niceblue,fill=niceblue,fill opacity=.3] (-.5,2) to [out=270,in=180] (0,1.5) -- (2,1.5) to [out=0,in=270] (2.5,2) to [out=90,in=0] (2,2.5) -- node[midway,above,niceblue,opacity=1] {$S_{12}$} (0,2.5) to [out=180,in=90] (-.5,2);
	\end{tikzpicture}}
	\caption{Graph $\mathcal{G}=(V,E,\mat{W})$ with $V=\{1,2,3,4\}$, $E=\bigl\{\{1,2\}, \{1,3\}, \{2,4\}, \{3,4\}\bigr\}$, and weights $w_{ij}=\mathbbm{1}_{i\sim j}/16$.}
	\label{img:xvst_theory_example}
\end{figure}

    Since the graph $\mathcal{G}$ is highly symmetrical, any one $st$-MinCut value can be attained by multiple partitions. For instance, the $12$-MinCut is attained by $S_1 \coloneqq \{1\}$, $S_2 \coloneqq \{2\}$ and $S_{13} \coloneqq \{1,3\}$. 
    The (asymptotic) probabilities that the minimum is attained by either of the above partitions are:
    \begin{align*}
    \P\bigl(\MC(S_{13}) < \MC(S_1),\, \MC(S_{13}) < \MC(S_2)\bigr) &\;\nto\; 1/4,\\
    \P\bigl(\MC(S_1) < \MC(S_2),\, \MC(S_1) < \MC(S_{13})\bigr) &\;\nto\; 3/8,\\
    \P\bigl(\MC(S_2) < \MC(S_1),\, \MC(S_2) < \MC(S_{13})\bigr) &\;\nto\; 3/8.
    \end{align*}


    As there are only four nodes, the consideration of local maxima has nearly no imfluence on the computation time, in sharp contrast to large scale graphs. In this simple example, it can be shown that for the empirical graph $\mathcal{G}_n$, $\P(\abs{\Vloc} = 1)\to 2/3$ and $\P(\abs{\Vloc} = 2)\to 1/3$ as $n\to\infty$. That is, the \xistref algorithm will terminate without having computed any partition at all with an asymptotic probability of $2/3$. To avoid this, we redefine $\Vloc\coloneqq V$. Then, for the empirical graph, either $S_{12} = \{1,2\}$ or $S_{13} = \{1,3\}$ is considered by \xistref regardless of the empirical distribution $\vec{Y}$ (and therefore irrespective of the underlying distribution $\vec{p}$). 
    Since both $S_{12}$ and $S_{13}$ yield the same minimal NCut value, \xistref must return the (empirical equivalent of the) NCut value $\NC_{S_{12}}(\mathcal{G}) = \NC_{S_{13}}(\mathcal{G}) = 2$ for any $n\in\N$. 
    Consequently, it suffices to consider $\NC_{S_{12}}(\mathcal{G}_n)$ and ${\NC}_{S_{13}}(\mathcal{G}_n)$. Together, they converge weakly to $\vec{Z}^{\NC}\sim \mathcal{N}_2(\vec{0},\mat{0}) = \delta_{(0,0)}$ (and thus also in probability), where the degenerate covariance matrix $\mat{0}\in\R^{2\times 2}$ consists of zeros only, meaning that the asymptotic distribution of $\sqrt{n}(\xist(\mathcal{G}_n) - \xist(\mathcal{G}))$ is degenerate, and given by $\min \vec{Z}^{\NC} = \delta_{(0,0)}$.

    Towards a non-degenerate limit, we increase the scaling factor and consider instead $n\bigl(\xist(\mathcal{G}_n) - \xist(\mathcal{G})\bigr)$. Let $Z_1,Z_2,Z_3\iidsim\mathcal{N}(0,1)$. Considering $S_{12}$ and $S_{13}$ separately yields
    \begin{align}
         n\bigl(\NC_{S_{12}}(\mathcal{G}_n) - \NC_{S_{12}}(\mathcal{G})\bigr) &\dto -\tfrac{1}{6} Z_1^2 - \tfrac{7}{6} Z_2^2 + Z_3^2 \nonumber \\
         &\qquad- \sqrt{\tfrac{49}{18}}\, Z_1 Z_2 + \sqrt{\tfrac{50}{3}}\, Z_1 Z_3 - \sqrt{\tfrac{25}{3}}\, Z_2 Z_3, \label{eq:exmp:limit_distribution_S_12} \\
        \intertext{for $S_{12}$, and for $S_{13}$ analogously}
         n\bigl(\NC_{S_{13}}(\mathcal{G}_n) - \NC_{S_{13}}(\mathcal{G})\bigr) &\dto -\tfrac{1}{6} Z_1^2 + \tfrac{35}{12} Z_2^2 - \tfrac{13}{4} Z_3^2 \nonumber \\
         &\qquad+ \sqrt{\tfrac{3421}{45}}\, Z_1 Z_2 + \sqrt{\tfrac{3}{2}}\, Z_1 Z_3 - \sqrt{\tfrac{3}{4}}\, Z_2 Z_3, \label{eq:exmp:limit_distribution_S_13}
    \end{align}
    i.e., asymptotically, both $\NC_{S_{12}}(\mathcal{G}_n)$ and $\NC_{S_{13}}(\mathcal{G}_n)$ follow a generalized $\chi^2$-distribution.
    This can be extended to obtain an asymptotic distribution for \xistref through joint convergence, arriving at
    \begin{equation*}
    n\bigl(\xist(\mathcal{G}_n) - \xist(\mathcal{G})\bigr) = \min_{S \in\{ S_{12},\,S_{13}\}} n\bigl(\NC_S(\mathcal{G}_n) - \NC_S(\mathcal{G})\bigr) 
    \dto \min\{\chi_{S_{12}}^2, \,\chi_{S_{13}}^2\},
    \end{equation*}
    where $\chi_{S_{13}}^2$ and $\chi_{S_{13}}^2$ denote the limiting distributions from \eqref{eq:exmp:limit_distribution_S_12} and \eqref{eq:exmp:limit_distribution_S_13}, respectively. 
\end{example}


\subsection{Multiway cuts}
\label{sub:multiway_cuts}

\emph{Multiway cuts} are graph cuts that split the nodes of a graph into $k\in\N$ pairwise disjoint subsets. The multiway MCut, known as the \emph{$k$-cut problem}, can be solved in polynomial time \cite{GoldschmidtHochbaum1994}. If, however, one would prespecify $k$ vertices to belong to disjoint subsets, i.e.\ construct a version of the $st$-MinCut for $k\geq 3$ vertices, the resulting problem would be NP-complete in general \cite{Dahlhaus_etal1994}. As previously stated, even the case of $k=2$ is NP-complete if balancing terms are introduced. The literature on multiway cuts is rather limited and mostly consists of simple extensions of XCut results ($k=2$) to general $k\in\N$ (see e.g.\ \cite{vonLuxburg2007,Trillos_etal2016}), similar to the way we will proceed in the following.

\begin{definition}\label{def:stmincut:k}
    For a graph $\mathcal{G}=(V,E,\mat{W})$ and $k\in\N$, introduce the set of $k$-partitions of $V$ as
    \[
    \mathcal{S}^{k} \coloneqq \left\{ S = \{S_1,\ldots,S_k\}:\, \bigcup_{i=1}^k S_i = V,\text{ and }S_i\cap S_j=\varnothing\text{ for all } i\neq j\right\}.
    \]
   The \emph{balanced multiway graph cut} $\kXC$ of the graph $\mathcal{G}$ is defined as follows:
    \begin{align*}
	\kXC(\mathcal{G}) &\coloneqq \min_{S =\{S_1,\ldots,S_k\} \in\mathcal{S}^{k}}\kXC_S(\mathcal{G})\\
	\text{with}\quad\kXC_S(\mathcal{G}) &\coloneqq \frac{1}{2}\sum_{\ell=1}^k \sum_{i\in S_{\ell}, j\in \comp{S_{\ell}}}\frac{w_{ij}}{\bal(S_{\ell},\comp{S_{\ell}})} = \frac{1}{2} \sum_{\ell=1}^k \XC_{S_{\ell}}(\mathcal{G}).
	\end{align*}
\end{definition}
The factor $1/2$ in $\kXC$ is included to avoid counting each edge twice. The case $k=1$ is trivial. In case of $k=2$, $\XC_S(\mathcal{G}) = \kXC_{\{S,\comp{S}\}}(\mathcal{G})$ for any $S\in\mathcal{S}$ or equivalently for any $\{S,\comp{S}\} \in \mathcal{S}^2$. Thus, $\kXC$ is equivalent to the usual $\XC$ for $k=2$. Further, as $\kXC$ is merely a (scaled) sum of cuts, it is not difficult to adapt \cref{thm:conv_mrcut_S,cor:conv_xcut} for general $k\geq 2$.

\begin{theorem}
	\label{thm:conv_mrcut_S_k}
    Assume the same setup as in \cref{thm:conv_mrcut_S}, and let $\vec{Z}^{\XC} = (Z_S^{\XC})_{S\in\mathcal{S}}$ be the limiting distribution of $\bigl(\sqrt{n}\left(\XC_S(\mathcal{G}_n) - \XC_S(\mathcal{G})\right)\bigr)_{S\in\mathcal{S}}$ therein. Then:
	\begin{enumerate}[label=(\roman*)]
	\item As $n\to \infty$, it holds that \label{thm:conv_mrcut_S_k:Hadamard}
	\[
	\Bigl(\sqrt{n}\bigl({\kXC}_S(\mathcal{G}_n) - \kXC_S(\mathcal{G})\bigr)\Bigr)_{S\in\mathcal{S}^k}\quad \dto\quad \left(\frac{1}{2} \sum_{\ell=1}^k Z_{S_{\ell}}^{\XC}\right)_{S = \{S_1,\ldots,S_k\} \in \mathcal{S}^k}.
	\]
	\item If the balance functional $\ball_S^{\XC}(\cdot)$ is further differentiable at $\vec{p}$, then
	\[
	\Bigl(\sqrt{n}\bigl(\kXC_S(\mathcal{G}_n) - \kXC_S(\mathcal{G})\bigr)\Bigr)_{S\in\mathcal{S}^k} \quad\dto\quad\mathcal{N}_{\abs{\mathcal{S}^k}}\big(\vec{0},\mat{\Sigma}^{\kXC}\big),
	\]
	where the covariance matrix $\mat{\Sigma}^{\kXC} = (\mat{\Sigma}_{T,S}^{\kXC})_{T,S\in\mathcal{S}^k}\in\R^{\abs{\mathcal{S}^k}\times\abs{\mathcal{S}^k}}$ is given by
	\[
	\mat{\Sigma}_{T,S}^{\kXC} \coloneqq \frac{1}{4}\sum_{i=1}^k \sum_{j=1}^k \mat{\Sigma}_{T_i,S_j}^{\XC} \qquad\text{for}\quad T = (T_1,\ldots,T_k),\, S = (S_1,\ldots,S_k)\in\mathcal{S}^k,
	\]
	with $\mat{\Sigma}_{T,S}^{\XC}$, for $T,S\in\mathcal{S}$, given in \cref{thm:conv_mrcut_S}~(ii).\label{thm:conv_mrcut_S_k:diff}
	\item
	Let $\mathcal{S}_*^k \coloneqq \bigl\{S \in \mathcal{S}^k\,:\,\kXC_S(\mathcal{G}) = \kXC(\mathcal{G}) \bigr\}.$ Then
	$
	\sqrt{n}\bigl({\kXC}(\mathcal{G}_{n}) - \kXC(\mathcal{G})\bigr)  \dto  \frac{1}{2}\min_{S\in\mathcal{S}_{*}^k} \sum_{\ell=1}^k Z_{S_{\ell}}^{\XC}.
	$
	\end{enumerate}
\end{theorem}

Note that \cref{thm:conv_mrcut_S,cor:conv_xcut} can be seen as the particular case of \cref{thm:conv_mrcut_S_k} when $k = 2$. Further, one can derive the explicit limiting distributions for the multiway versions of MCut, RCut, NCut and CCut, by means of \cref{cor:mrnccut_explicit}.

\section{Simulations and applications}
\label{sec:simulations}

The codes for the simulations can be found on GitHub \citep{Xist_Github}.

\subsection{Empirical convergence behavior}
\label{sub:convergence_verification_simulations}

We validate the limiting distributions in \cref{thm:conv_mrcut_S,cor:conv_xcut} through simulation. For simplicity and computational feasibility (as we aim to compute all XCuts exactly) we restrict ourselves to a $3\times 3$ grid of $m=9$ nodes. Computing the $2^9\times 2^9$ covariance matrix $\mat{\Sigma}^{\XC}$ for $\ell=4$ took 22 hours on a 24-core computing cluster equipped with 512 gigabytes of RAM while already taking advantage of heavily parallelized code. Moreover, in order to draw the samples from the limiting distribution, it is also necessary to invert $\mat{\Sigma}^{\XC}$ which we were unable to do for $\ell=4$ as it does not even exhibit sparsity properties that would have sped up computation. 
It should be noted, however, that this does not imply that our discretization is useless in practice due to its exponential computational complexity,  see \cref{sec:applications} for the computational efficiency of the \xistref algorithm. 

\begin{figure}[htpb]
	\centering
	\subfloat[RCut / NCut]{
	\centering
	\begin{tikzpicture}
		\draw[draw=none] (2,2) circle (9pt);
		\foreach \x in {0,1,2} {
		    \foreach \y in {0,1,2} {
		        \node (\x;\y) at (\x,\y) {};
	        }
        }
		\draw[black!40, line width=2.5mm] (0;2) -- (1;2) -- (2;2);
		\draw[black!40, line width=0.3mm] (0;1) -- (1;1) -- (2;1);
		\draw[black!40, line width=1.875mm] (0;0) -- (1;0) -- (2;0);
		\draw[black!40, line width=0.7mm] (0;2) -- (0;1);
		\draw[black!40, line width=0.7mm] (1;2) -- (1;1);
		\draw[black!40, line width=0.7mm] (2;2) -- (2,1);
		\draw[black!40, line width=0.525mm] (0;1) -- (0;0);
		\draw[black!40, line width=0.525mm] (1;1) -- (1,0);
		\draw[black!40, line width=0.525mm] (2,1) -- (2;0);
		\foreach \x in {0,1,2} {
            \draw[draw=black!50, fill=nicegreen] (\x;2) circle (9pt) node[red] {\large $\bm{\lambda}$};
            \draw[draw=black!50, fill=nicegreen] (\x;1) circle (6pt) node[red] {\large $\bm{\eps}$};
            \draw[draw=black!50, fill=niceblue] (\x;0) circle (8pt) node[red] {\large $\bm{1}$};
        }
	\end{tikzpicture}}%
	\hspace{.6cm}
    \subfloat[CCut]{
	\centering
	\begin{tikzpicture}
		\draw[draw=none] (2,2) circle (9pt);
		\foreach \x in {0,1,2} {
		    \foreach \y in {0,1,2} {
		        \node (\x;\y) at (\x,\y) {};
	        }
        }
		\draw[black!40, line width=2.5mm] (0;2) -- (1;2) -- (2;2);
		\draw[black!40, line width=0.3mm] (0;1) -- (1;1) -- (2;1);
		\draw[black!40, line width=1.875mm] (0;0) -- (1;0) -- (2;0);
		\draw[black!40, line width=0.7mm] (0;2) -- (0;1);
		\draw[black!40, line width=0.7mm] (1;2) -- (1;1);
		\draw[black!40, line width=0.7mm] (2;2) -- (2,1);
		\draw[black!40, line width=0.525mm] (0;1) -- (0;0);
		\draw[black!40, line width=0.525mm] (1;1) -- (1,0);
		\draw[black!40, line width=0.525mm] (2,1) -- (2;0);
		\foreach \x in {0,1,2} {
            \draw[draw=black!50, fill=nicegreen] (\x;2) circle (9pt) node[red] {\large $\bm{\lambda}$};
            \draw[draw=black!50, fill=niceblue] (\x;1) circle (6pt) node[red] {\large $\bm{\eps}$};
            \draw[draw=black!50, fill=niceblue] (\x;0) circle (8pt) node[red] {\large $\bm{1}$};
        }
	\end{tikzpicture}}%
    \hspace{.6cm}
    \subfloat[RCut / NCut]{
	\centering
	\begin{tikzpicture}
		\draw[draw=none] (2,2) circle (9pt);
		\foreach \x in {0,1,2} {
		    \foreach \y in {0,1,2} {
		        \node (\x;\y) at (\x,\y) {};
	        }
        }
        \draw[black!40, line width=1.875mm, xstep=1, ystep=1] (0,0) grid (2,2);
		\foreach \x in {0,1,2} {
            \draw[draw=black!50, fill=nicegreen] (\x;2) circle (8pt) node[red] {\large $\bm{1}$};
            \draw[draw=black!50, fill=nicegreen] (\x;1) circle (8pt) node[red] {\large $\bm{1}$};
            \draw[draw=black!50, fill=niceblue] (\x;0) circle (8pt) node[red] {\large $\bm{1}$};
        }
	\end{tikzpicture}}%
	\hspace{.6cm}
    \subfloat[CCut]{
	\centering
	\begin{tikzpicture}
		\draw[draw=none] (2,2) circle (9pt);
		\foreach \x in {0,1,2} {
		    \foreach \y in {0,1,2} {
		        \node (\x;\y) at (\x,\y) {};
	        }
        }
        \draw[black!40, line width=1.875mm, xstep=1, ystep=1] (0,0) grid (2,2);
		\foreach \x in {0,1,2} {
            \draw[draw=black!50, fill=nicegreen] (\x;2) circle (8pt) node[red] {\large $\bm{1}$};
        }
        \foreach \y in {0,1} {
                \draw[draw=black!50, fill=niceblue] (0;\y) circle (8pt) node[red] {\large $\bm{1}$};
                \draw[draw=black!50, fill=niceblue] (1;\y) circle (8pt) node[red] {\large $\bm{1}$};
                \draw[draw=black!50, fill=nicegreen] (2;\y) circle (8pt) node[red] {\large $\bm{1}$};
        }
	\end{tikzpicture}}%
	\caption{Neighbourhood graph of a bimodal distribution and the uniform distribution on a grid with $m=9$ nodes and $t=1$ (i.e.\ only direct horizontal and vertical adjacency). In (a), the blue nodes represent the RCut minimizing partition $S_*$ (which also attains the optimal NCut) as computed by a brute-force algorithm, while (b) depicts the partition attaining the minimal CCut for the same distribution. Panels (c) and (d) depict the same for the uniform distribution. The number inside a node represents its (unnormalized) weight in the underlying probability distribution.}
	\label{img:bimodal+uniform_distribution_examples}
\end{figure}
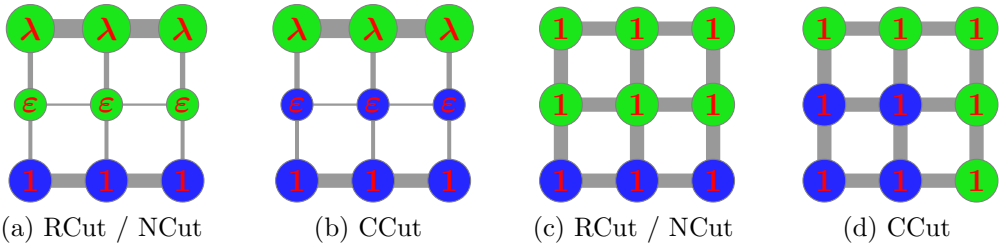

\begin{figure}[btph]
    \centering
    \subfloat[RCut statistic, bimodal distribution]{\includegraphics[width=0.43\linewidth]{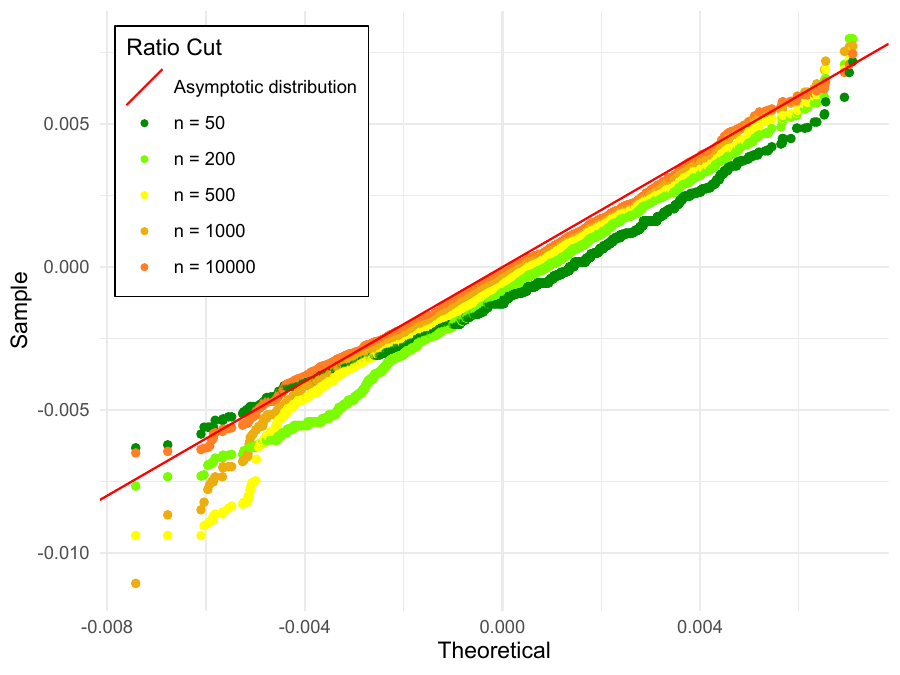}}%
    \subfloat[RCut statistic, uniform distribution]{\includegraphics[width=0.43\linewidth]{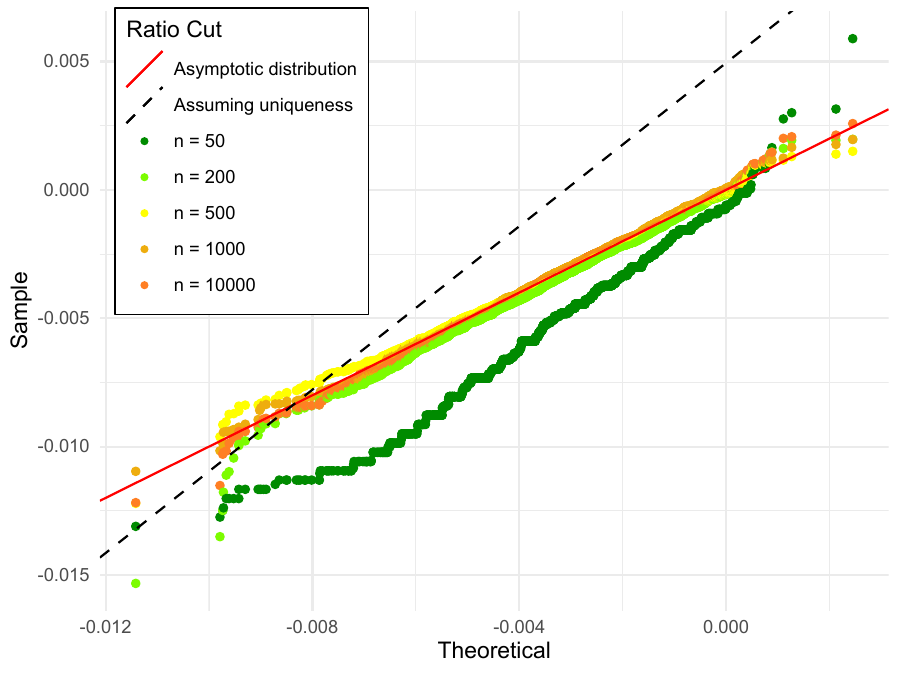}}%
    \\
    \subfloat[NCut statistic, bimodal distribution]{\includegraphics[width=0.43\linewidth]{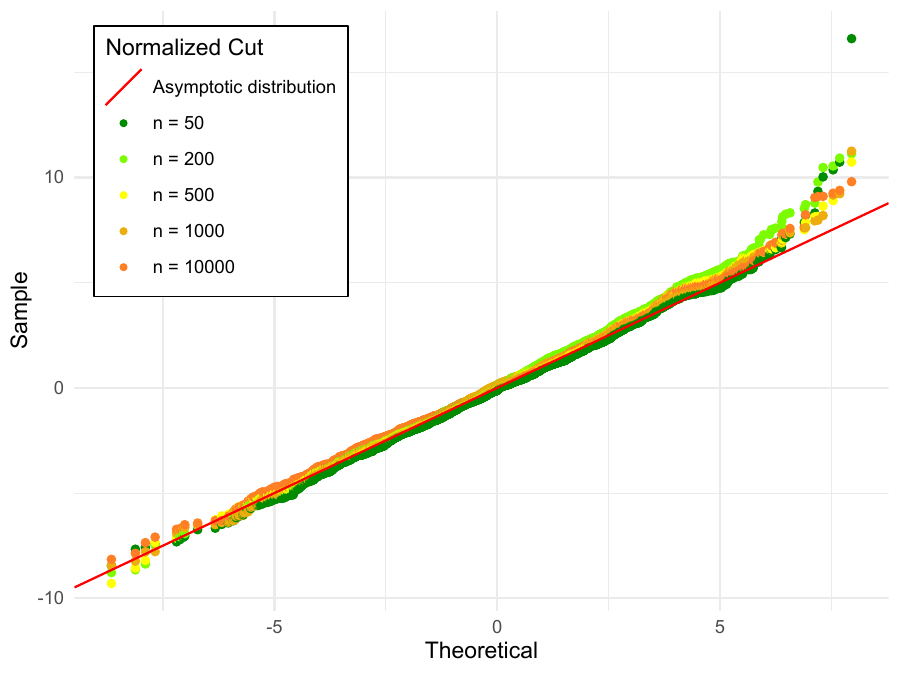}}%
    \subfloat[NCut statistic, uniform distribution]{\includegraphics[width=0.43\linewidth]{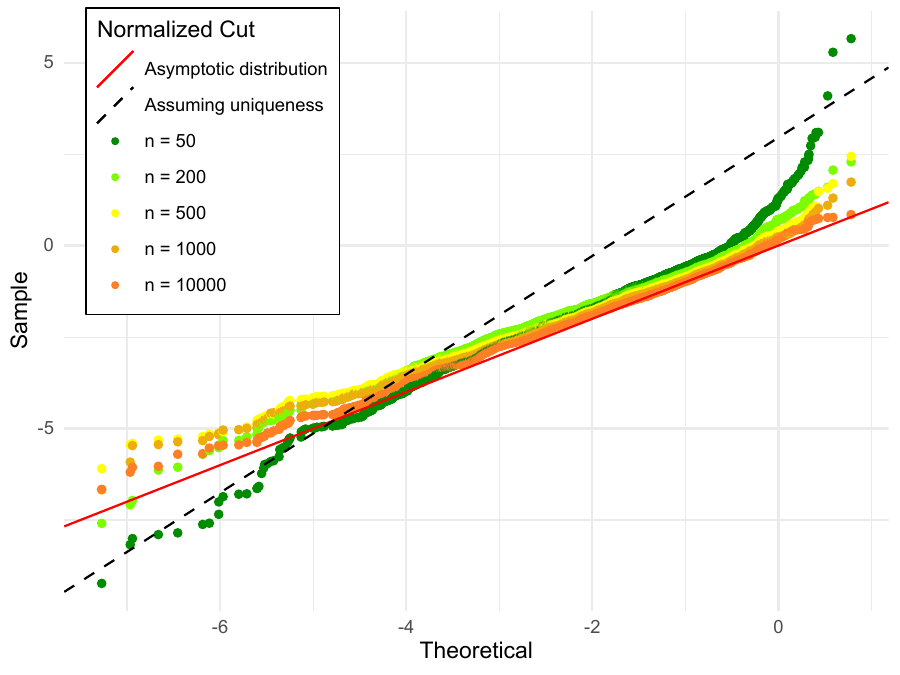}}%
    \\
    \subfloat[CCut statistic, bimodal distribution]{\includegraphics[width=0.43\linewidth]{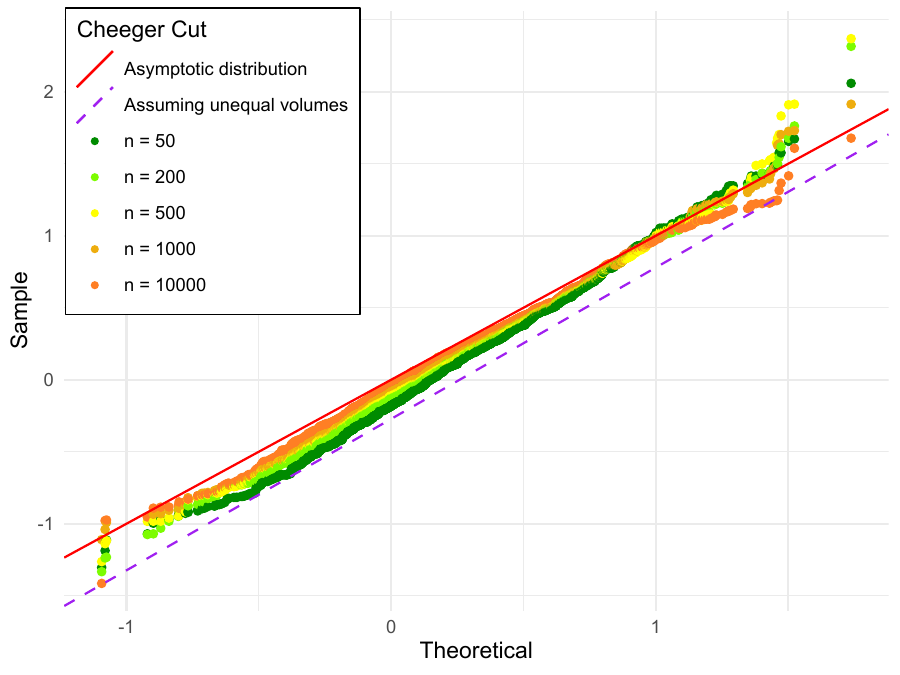}}%
    \subfloat[CCut statistic, uniform distribution]{\includegraphics[width=0.43\linewidth]{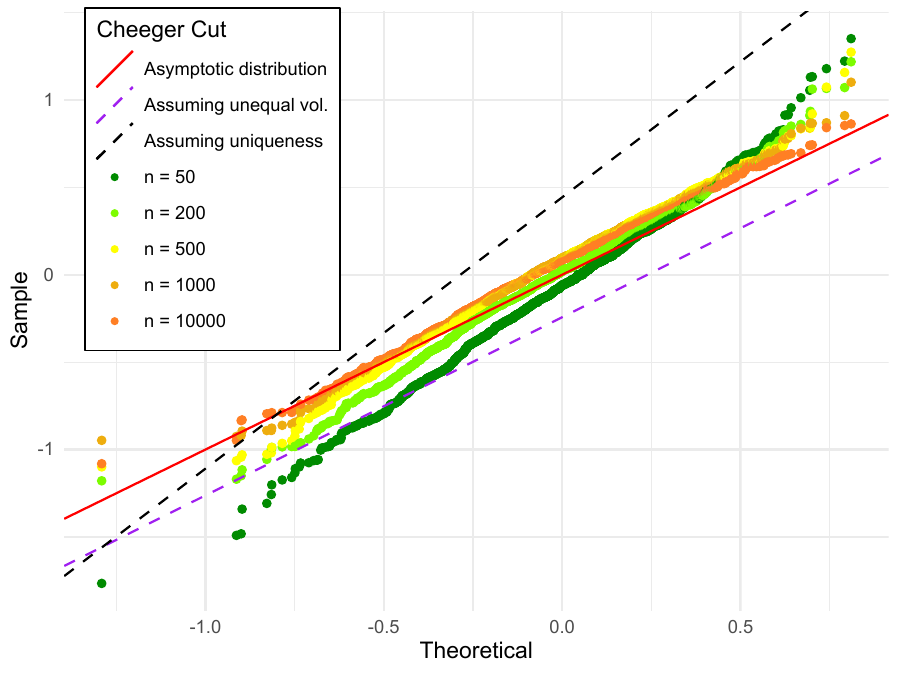}}%
    \caption{QQ plots of the XCut statistic $\sqrt{n}\big(\XC(\mathcal{G}_n) - \XC(\mathcal{G})\big)$ for various sample sizes $n$ against the limiting distribution in \cref{cor:conv_xcut} (using \cref{cor:mrnccut_explicit}) and for the uniform (right) and the bimodal (left) distribution (defined in \cref{img:bimodal+uniform_distribution_examples}). The red lines in the plots indicate the quantiles of the theoretical limiting distribution. The dashed black lines in (e) and (f) show the quantiles of the theoretical limiting distribution under the (false) assumption that $\vol(S_*)\neq\vol(S_*)$ for the partition $S_*$ that attains the minimal CCut. The dashed purple lines in (b), (d) and (f) indicate the limiting distribution under the (false) assumption that the XCut partition $S_*$ is the unique minimizer of $\XC_S(\mathcal{G})$ across $S\in\mathcal{S}$. All (dashed) lines were drawn through the $0.75$ and $0.25$ quantiles of the respective distributions.}%
    \label{img:qq_cbimodal_xcut_overn}%
\end{figure}

As the limiting distribution only depends on the probability vector $\vec{p}$, we consider two choices of $\vec{p}$ and compare how well and how fast the limiting scenario is attained.  One choice of $\vec{p}$ is the uniform distribution, with $p_i=1/m$, $i\in\{1,\ldots,m\}$, and the other is a bimodal distribution as in \cref{img:bimodal+uniform_distribution_examples}. The bimodal distribution is designed so that the optimal partition $S_*$ of CCut satisfies
\[
\vol(S_*)=\vol(\comp{S}_*)\iff 
\lambda=\sqrt{1+1.5\eps+\eps^2},
\]
in order to examine whether this fringe case indeed behaves in a non-Gaussian manner in accordance with \mysubref{cor:mrnccut_explicit}{cheeger}. We set $\eps=0.4$, and thus $\lambda = \sqrt{1+1.5\eps+\eps^2} \approx1.327$. \cref{img:qq_cbimodal_xcut_overn} shows a QQ plot comparing the sample distribution of $\sqrt{n}\big(\XC(\mathcal{G}_n) - \XC(\mathcal{G})\big)$ and its theoretical limiting distribution $\min_{S\in\mathcal{S}_{*}} Z_S^{\XC}$ as in \cref{cor:conv_xcut} for different sample sizes $n$ and the two distributions under consideration. While the weak convergence is reasonably fast as indicated by the red line for both distributions, this convergence is slower for the uniform distribution as can be seen by the different choices of $n$. The reason is that the uniform distribution has no unique optimal partition due to its symmetry in contrast to the bimodal distribution. If one were to, in the uniform case, wrongly assume that the minimizer is unique, the sample distribution would not converge to the supposed limiting distribution as indicated by the dashed black lines in \cref{img:qq_cbimodal_xcut_overn}~(b), (d) and (f).

Another mistake one is prone to make for CCut would be assuming that the volumes $\smash{\vol(S_*)}$ and $\smash{\vol(\comp{S}_*)}$ of the optimal partition $S_*$ are not equal when they actually are. Note that $\smash{\vol(S_*)} = \smash{\vol(\comp{S}_*)}$ for both the bimodal and the uniform distribution. Consequently, in \cref{img:qq_cbimodal_xcut_overn}~(e) the dashed black line indicates the Gaussian limiting distribution, namely the one obtained by assuming that $\vol(S_*)\neq\vol(\comp{S}_*)$. The dashed purple line in (f) indicates this case for the uniform distribution (though there, the supposed limit is not Gaussian). 

\begin{figure}[hbtp]
    \centering
    \subfloat[RCut statistic, bimodal distribution]{\includegraphics[width=0.43\linewidth]{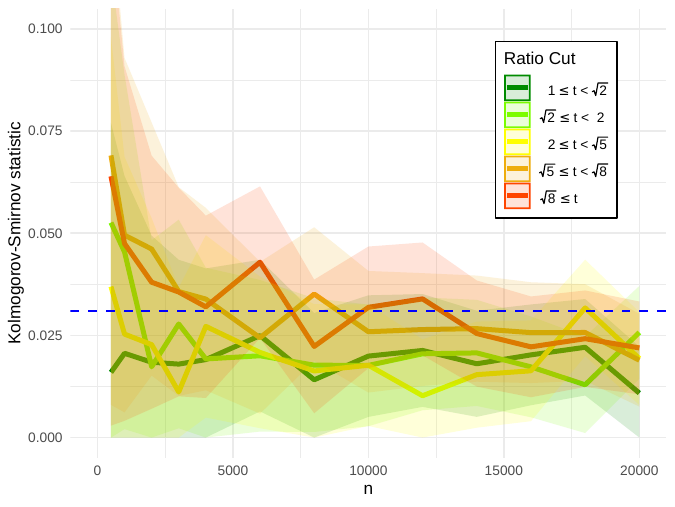}}%
    \subfloat[RCut statistic, uniform distribution]{\includegraphics[width=0.43\linewidth]{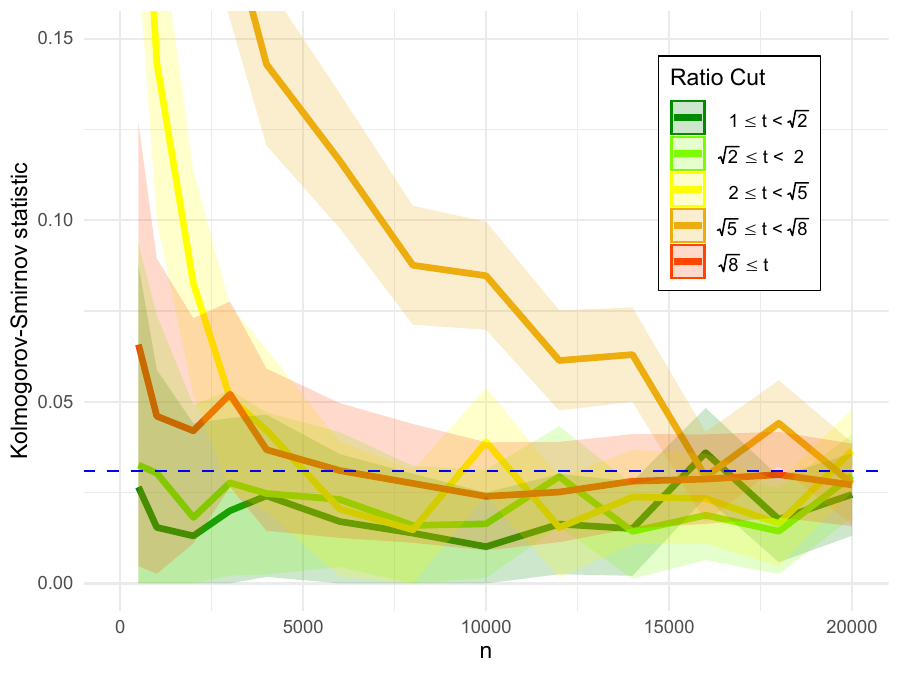}}%
    \\
    \subfloat[NCut statistic, bimodal distribution]{\includegraphics[width=0.43\linewidth]{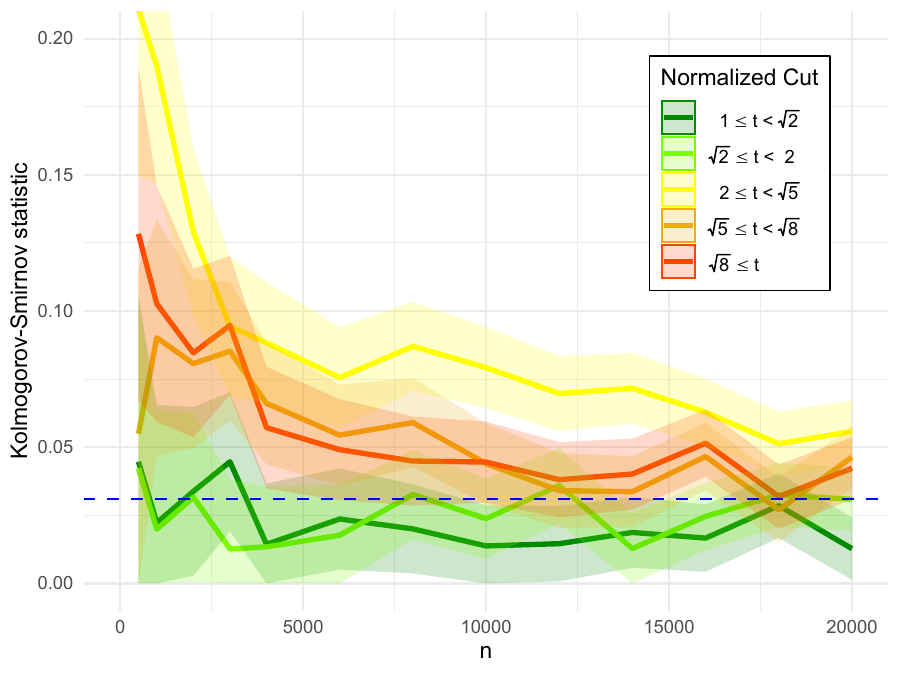}}%
    \subfloat[NCut statistic, uniform distribution]{\includegraphics[width=0.43\linewidth]{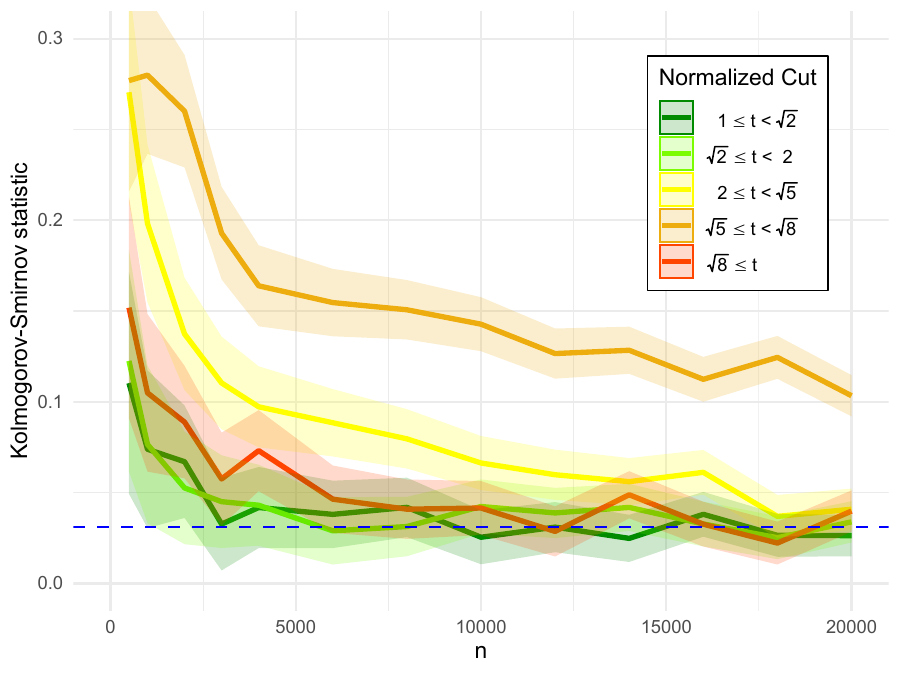}}%
    \\
    \subfloat[CCut statistic, bimodal distribution]{\includegraphics[width=0.43\linewidth]{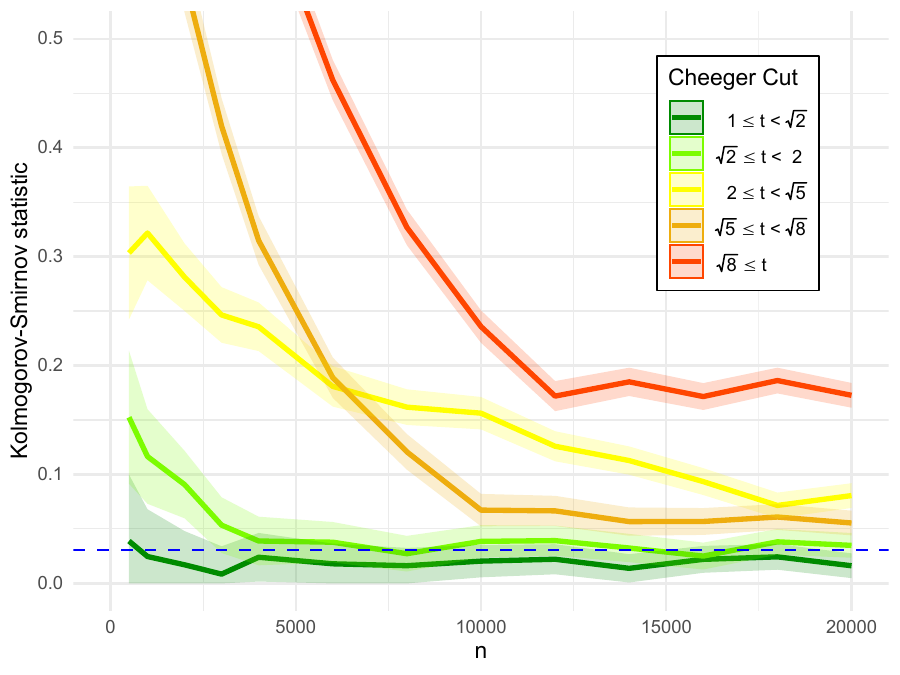}}%
    \subfloat[CCut statistic, uniform distribution]{\includegraphics[width=0.43\linewidth]{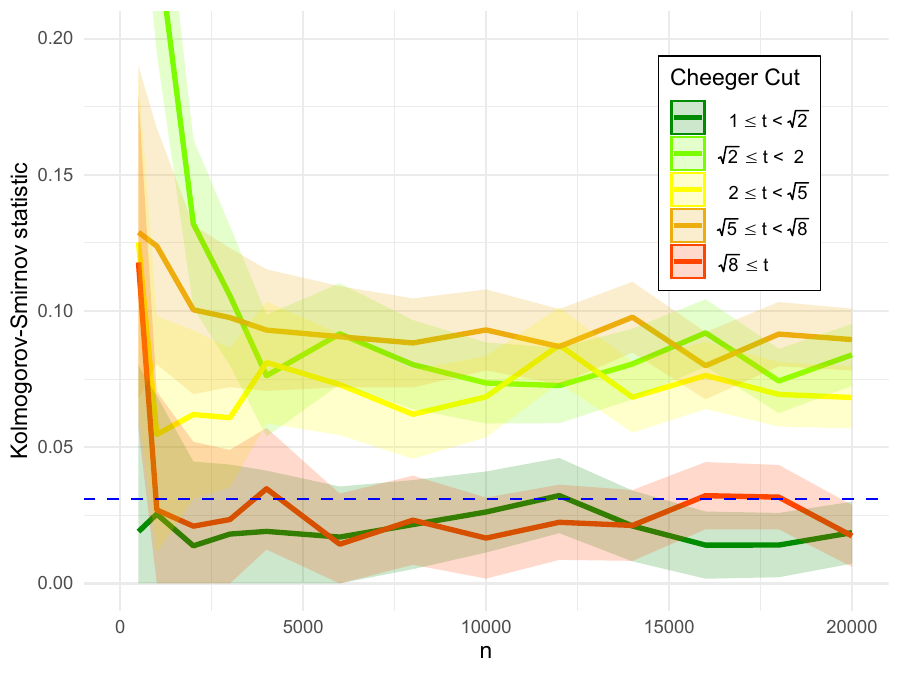}}%
    \caption{Kolmogorov--Smirnov distances $D_{n,t}$ between $\smash{\sqrt{n}\big(\XC(\mathcal{G}_{n,t}) - \XC(\mathcal{G}_t)\big)}$ and the respective limiting distributions of RCut, NCut and CCut for different $\smash{t\in\mathcal{T}=\{1,\,\sqrt{2},\, 2,\, \sqrt{5},\, \sqrt{8}\}}$ against the sample size $n$ (note that both $\mathcal{G}_{n,t} \coloneqq \mathcal{G}_n$ and $\mathcal{G}_t \coloneqq \mathcal{G}$ depend on $t$ by \cref{def:discretization}). We considered both the uniform and bimodal distributions that are given by \autoref{img:bimodal+uniform_distribution_examples}. The dashed blue line indicates the critical value of a Kolmogorov--Smirnov test as outlined in the text. For each $n$ and $t$, we repeat the simulation 100 times, and summarize the results by plotting the median of Kolmogorov--Smirnov distances in solid lines, and a two-sided 90\% confidence interval in shaped regions.}%
    \label{img:ks_xcut_overt}%
\end{figure}

\subsection{Influence of neighborhood distance}
\label{sub:neighbourhood_simulations}

Consider the same two distributions as in \cref{img:bimodal+uniform_distribution_examples}, but with different choices of the neighborhood distance $t$. The discrepancy between sample and limiting distribution will be measured by the Kolmogorov--Smirnov distance
$D_{n,t}^{\XC} \coloneqq \sup_{x\in\R}\bigl|\widehat{F}_{n,t}^{\XC}(x) - F^{\XC}_t(x)\bigr|,$
where $\widehat{F}_{n,t}^{\XC}$ is the distribution function of
$\sqrt{n}\big(\XC(\mathcal{G}_{n,t}) - \XC(\mathcal{G}_t)\big)$
with the subscript $t$ indicated the dependence on the neighborhood distance $t$, and $F_t^{\XC}$ is the distribution function of $\min_{S\in\mathcal{S}_{*}} Z_S^{\XC}$. We obtain $F_t^{\XC}$ by Monte Carlo simulation (generating 50,000 samples). With  $D_{n,t}^{\XC}$ one can construct a Kolmogorov--Smirnov test against the null hypothesis $H_0$ that $F_t^{\XC}$ is the underlying distribution of $\widehat{F}_{n,t}^{\XC}$. This test will reject $H_0$ if $D_{n,t}^{\XC}>q_{1-\alpha}$ for the $(1-\alpha)$-quantile $q_{1-\alpha}$ of the Kolmogorov distribution (which is the distribution of $\sup_{0\le t\le 1} \vert B(t)\vert$ with $B(t)$ the standard Brownian bridge).
We use a sample size of $n=2{,}000$ and an (asymptotic) significance level $\alpha=0.05$, see \cref{img:ks_xcut_overt}. 
%
%
Overall, RCut converges fastest and most consistently across for all choices of $t$. Also, for all three cuts, a small neighborhood of $t=1$, i.e.\ connecting only directly adjacent nodes, apparently yield distributions that come closest to the limit (which is not surprising since choosing the smallest possible $t$ is natural for only $m=9$ nodes). Another interesting observation is that in the uniform case for $t\in[\sqrt{5},\sqrt{8})$, the Kolmogorov--Smirnov distance converges rather slowly for both RCut and CCut. Overall, CCut seems to perform worst out of all three cuts. Its curious behavior, especially in \cref{img:ks_xcut_overt}~(f) where for three choices of $t$ the Kolmogorov--Smirnov distance does not seem to improve with increasing $n$, could be subject to a more thorough investigation in the future.

Another quantity of interest is the variance of the limiting distribution derived in \cref{thm:conv_mrcut_S} and how it behaves as the neighbourhood distance $t$ increases. In Appendix~\ref{apdx:sub:variance}, we investigate its behaviour empirically. In particular, as demonstrated in \cref{img:var_t_behaviour}, the variance does not follow a pattern in general, i.e.\ one can easily find cases where it is monotonically increasing (or decreasing) in $t$, but there are also cases where it is neither. Analyzing the variance of the limiting distribution more thoroughly from a theoretical perspective, in particular to find partitions and distributions that minimize it, remains a prospective avenue to be explored in a future work.

\subsection{Empirical validity of bootstrap}
\label{sub:bootstrap_simulations}

We verify the ramifications of incorporating the boostrap into our asymptotic theory, in particular \cref{thm:bootstrap_limit_xcut} by simulations. To elaborate, consider a $3\times 3$ grid, i.e.\ $m=9$, and the bimodal distribution introduced in \cref{img:bimodal+uniform_distribution_examples}. There are five choices of neighbourhood distance $t$ that each yield a different graph structure. However, for only two of them the partition attaining the minimum CCut is unique. For the first such choice of $t$, namely if $1\leq t<\sqrt{2}$, the design of the distribution $\vec{p}$ ensures that the volume $S$ of the CCut minimizing partition is equal to that of $\comp{S}$ (as discussed in \cref{sub:convergence_verification_simulations}). Consequently, the Hadamard directional derivative is not continuous at $\vec{p}$ for this choice of $t$ and there is no limiting distribution for the bootstrap (see \cite[Proposition~1]{Duembgen1993}). Indeed, in \cref{img:ccut_bootstrap}~(a), it is evident that the dark green line corresponds to the Kolmogorov--Smirnov distance between $\sqrt{n}(\CC(\mathcal{G}_{M,n}) - \CC(\mathcal{G}_n))$ and the supposed theoretical limit $Z_S^{\CC}$, where $S$ is the aforementioned CCut minimizing partition, has a markedly higher Kolmogorov--Smirnov distance than the second and third choices of $t$ do, i.e.\ where $\sqrt{2}\leq t<\sqrt{5}$. In both of the cases there is a unique partition $S_*$ minimizing CCut for which the volumes of $S_*$ and $\comp{S}_*$ are not equal, thus yielding Hadamard differentiablity with a linear Hadamard directional derivative. This explains why the light green and the yellow line in \cref{img:ccut_bootstrap}~(a) stay at small values for large $n$. The other remaining two choices of $t$, represented through the orange and red lines, both lead to multiple partitions attaining the minimum CCut, thus not yielding a limiting distribution (see again \cite[Proposition~1]{Duembgen1993}).

\begin{figure}[htpb]
	\centering
	\subfloat[$M=n$]{
	   \includegraphics[width=.43\linewidth]{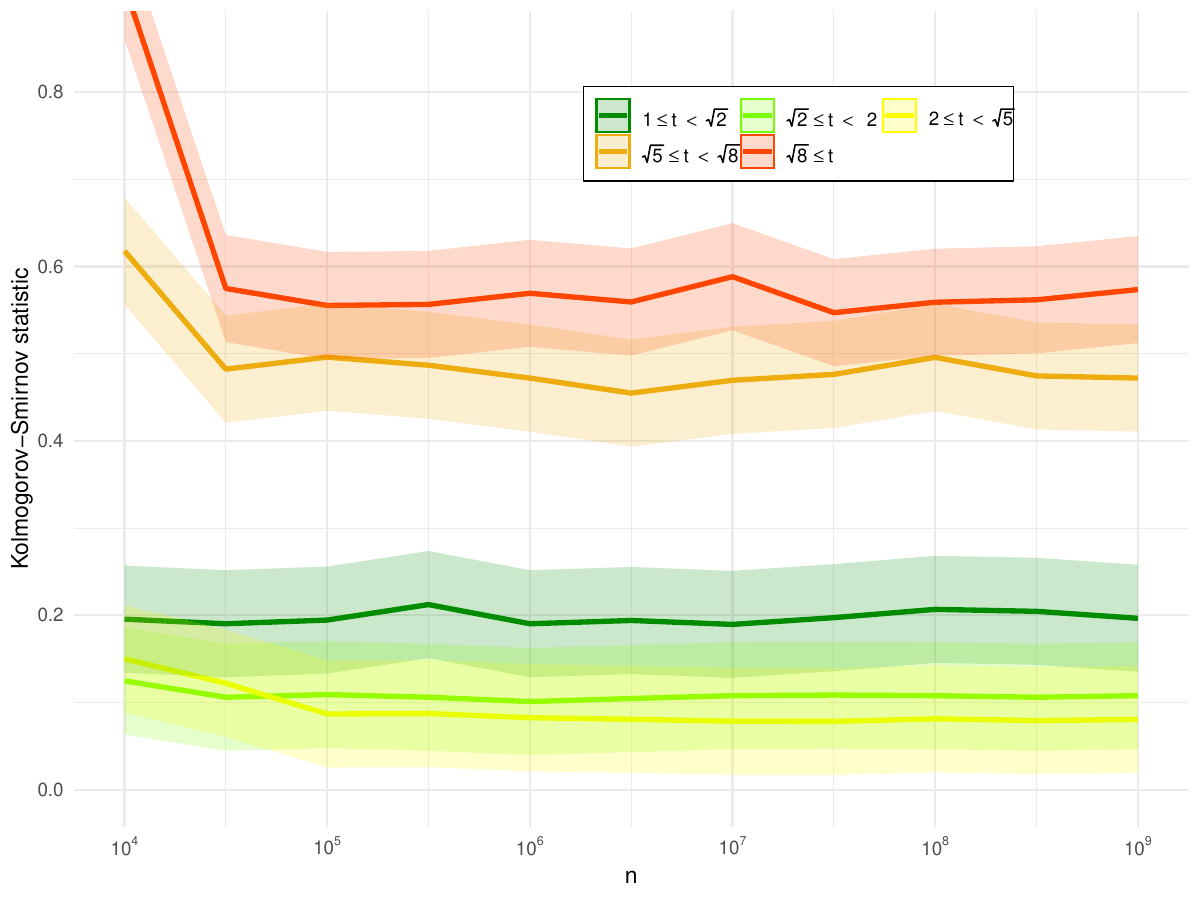}
	\label{img:ccut_bootstrap:meqn}
    }%
    \hspace{.3cm}
	\subfloat[$M=\sqrt{n}$]{
    \includegraphics[width=.43\linewidth]{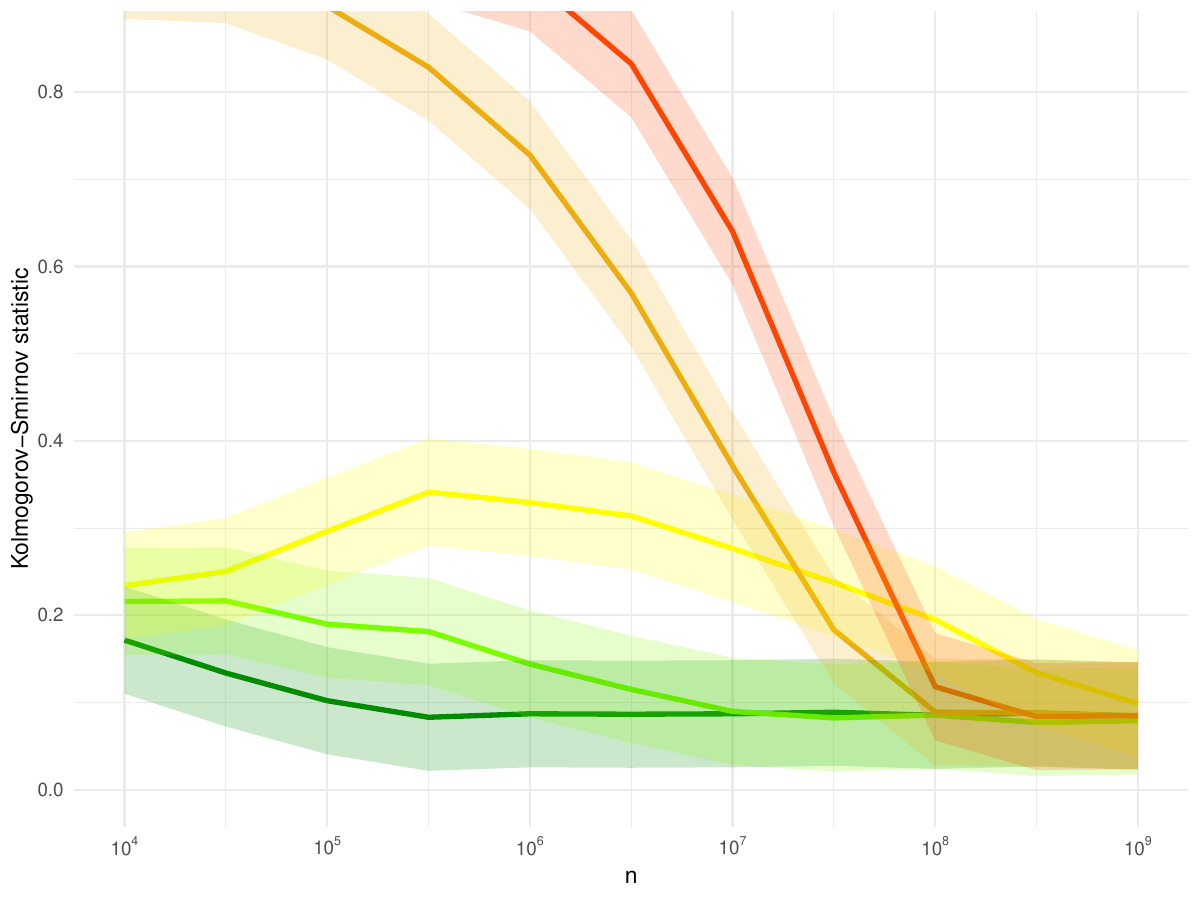}
	\label{img:ccut_bootstrap:meqsqrtn}
    }
	\caption{Kolmogorov--Smirnov distances of the CCut bootstrap estimator $\sqrt{n}(\CC(\mathcal{G}_{M,n}) - \CC(\mathcal{G}_n))$ for $M=n$ and $M=\sqrt{n}$ and the theoretical limit $\min_{S\in\mathcal{S}_{\vec{p}}}Z_S^{\CC}$ for the bimodal distribution depicted in \cref{img:bimodal+uniform_distribution_examples}. The limiting distribution was obtained using Monte Carlo simulations of 2,000 samples, and the distribution of the bootstrap estimator was obtained by sampling 100 times from $\mult(M,\vec{Y}/n)$). For each $n$ and $t$, we repeat the simulation 100 times, and summarize the results by plotting the median of Kolmogorov--Smirnov distances in solid lines, and a two-sided 90\% confidence interval in shaped regions.}
	\label{img:ccut_bootstrap}
\end{figure}


If we employ the $M_n$-out-of-$n$ bootstrap s.t.\ $M_n\to\infty$ and $n/M_n\to\infty$ as $n\to\infty$, $\sqrt{n}(\XC(\mathcal{G}_{M_n,n}) - \XC(\mathcal{G}_n))$ converges in distribution towards the theoretical limit from \cref{thm:conv_mrcut_S}. To illustrate this, choose $M\coloneqq \sqrt{n}$ and consider \cref{img:ccut_bootstrap} (b). Here, all choices of $t$ yield a distribution that clearly converges towards the aforementioned theoretical limit as indicated by the small Kolmogorov--Smirnov distance. While both the orange curve (i.e.\ when $\sqrt{5}\leq t<\sqrt{8}$) and the red one (i.e.\ when $t\geq\sqrt{8}$, meaning that the graph is complete) require much larger sample sizes $n$ in order for their Kolmogorov--Smirnov distance to decrease. 
This discrepancy in convergence rate can be explained by the fact that due to the large number of connections the minimal CCut is attained by 9 partitions for $\sqrt{5}\leq t<\sqrt{8}$, and by 27 partitions in the fully connected case. Consequently, the CCut limiting distribution is the minimum of a 9-variate (or 27-variate, respectively) Gaussian which may lead to a less accurate finite sample approximation. Interestingly, for $2\leq t<\sqrt{5}$, the Kolmogorov--Smirnov distance first increases slightly with $n$, only to decrease as for the other choices of $t$, though this behaviour is likely a sampling artefact and result of the small sample size of only 100 for the distribution of $\vec{Y}$.

\subsection{Asymptotic testing}
\label{sub:testing}

Asymptotic hypothesis tests are a prominent application for limit theorems, and graph cuts in particular offer interesting possibilities in this regard. 
Consider an arbitrary probability vector $\vec{p}$, and let $\mathcal{G}$ and $\mathcal{G}_n$ be the population and empirical graphs, respectively, constructed from (a sample from) $\vec{p}$. Further, let $Q_{\vec{p}}^{\XC}:[0,1]\to\R$ be the quantile function of the limiting distribution of $\sqrt{n}\bigl(\XC(\mathcal{G}_n) - \XC(\mathcal{G})\bigr)$ according to \cref{cor:conv_xcut}. To simplify notation, let $Q^\mathrm{unif}$ be the quantile function for the above limiting distribution for the uniform probability vector $\vec{p}^{\mathrm{unif}}\coloneqq (\tfrac{1}{m},\ldots,\tfrac{1}{m})^{\top}$, and let $\mathcal{G}^{\mathrm{unif}}$ and $\mathcal{G}_n^{\mathrm{unif}}$ be the corresponding population and empirical graphs. 
We now seek to test
\begin{equation}
H_0:\XC(\mathcal{G}) = \XC(\mathcal{G}^{\mathrm{unif}})\quad\text{vs.}\quad H_1: \XC(\mathcal{G}) \neq \XC(\mathcal{G}^{\mathrm{unif}}), \label{eq:asymptotic_test_xcut}
\end{equation}
i.e.\ whether cutting the distribution $\vec{p}$ yields a different cut value than the uniform distribution. 
If the underlying distribution has two or more modes, we expect the resulting cut value to be small since its associated partition separates the modes that are assumed to be balanced (i.e.\ (close to) maximizing the graph cut balancing term), thus yielding a lower cut value than if the same partition were applied to the uniform distribution. In the case of one mode, however, the resulting partition is expected to either remove a small subset of vertices near the border (if the single cluster is very pronounced) or cut through the cluster, thus sacrificing a . In contrast, for the uniform distribution, one expects a markedly different (i.e.\ more balanced) graph cut partition and therefore also a different cut value. This means that the test problem \eqref{eq:asymptotic_test_xcut} will reveal whether there is clustering structure in the underlying distribution. 

Fix $\alpha\in [0,1]$ as the asymptotic level of significance.  We reject in the testing problem \eqref{eq:asymptotic_test_xcut} if the underlying distribution satisfies $\XC(\mathcal{G}_n) < \XC(\mathcal{G}^{\mathrm{unif}}) + Q^{\mathrm{unif}}(\alpha/2)/\sqrt{n}$ or if $\XC(\mathcal{G}_n) > \XC(\mathcal{G}^{\mathrm{unif}}) + Q^{\mathrm{unif}}(1-\alpha/2)/\sqrt{n}$. This is because that under $H_0$, by \cref{cor:conv_xcut},
\begin{align*}
&\P\Bigl(\XC(\mathcal{G}_n)\in \Bigl[\XC(\mathcal{G}^{\mathrm{unif}}) + \frac{Q^\mathrm{unif}(\alpha/2)}{\sqrt{n}}, \XC(\mathcal{G}^{\mathrm{unif}}) + \frac{Q^\mathrm{unif}(1-\alpha/2)}{\sqrt{n}}\Bigr]\Bigr) \\
&\qquad= \P_{\vec{p}^{\mathrm{unif}}}\bigl(\sqrt{n}\bigl(\XC(\mathcal{G}_n) - \XC(\mathcal{G})\bigr) \in \bigl[Q^\mathrm{unif}(\alpha/2), Q^\mathrm{unif}(1-\alpha/2)\bigr]\bigr) \nto 1-\alpha,
\end{align*}
i.e.\ under $H_0$ the rejection probability is indeed asymptotically at most $\alpha$.

Consider now the alternative hypothesis $H_1$ in \eqref{eq:asymptotic_test_xcut}. If $\XC(\mathcal{G}) > \XC(\mathcal{G}^{\mathrm{unif}})$, centering around $\XC(\mathcal{G}^{\mathrm{unif}})$ will not yield a limiting distribution; in fact, \cref{cor:conv_xcut} together with Slutsky's lemma dictates that $\sqrt{n}\bigl(\XC(\mathcal{G}_n) - \XC(\mathcal{G})\bigr)$ will degenerate towards infinity in probability, so that the test above will reject with probability tending to 1. In contrast, if the sample stems from a distribution for which $\XC(\mathcal{G}) < \XC(\mathcal{G}^{\mathrm{unif}})$ holds, the centering around $\XC(\mathcal{G}^{\mathrm{unif}})$ also does not yield a limiting distribution which degenerates towards negative infinity. Hence, as $n\to\infty$, the considered test will not reject with probability approaching 1. This demonstrates that the above testing procedure provides a means to decide whether the distribution underlying a given sample is predominantly uniform, or whether it exhibits some kind of clustering structure, be it uni- or multimodal.

\begin{figure}[htpb]
    \centering
    \subfloat[NCut, $\eps\leq 1$]{
    	\centering
    	\begin{tikzpicture}
    		\draw[black!40, line width=2mm] (0,0) -- (3,0);
    		\draw[black!40, line width=0.7mm] (0,1) -- (3,1);
    		\draw[black!40, line width=0.7mm] (0,2) -- (3,2);
    		\draw[black!40, line width=2mm] (0,3) -- (3,3);
            \foreach \x in {0,1,2,3} {
                \draw[black!40, line width=1.2mm] (\x,0) -- (\x,1);
                \draw[black!40, line width=0.7mm] (\x,1) -- (\x,2);
                \draw[black!40, line width=1.2mm] (\x,2) -- (\x,3);
            }
            \foreach \x in {0,1,2,3} {
                \draw[draw=black!50, fill=niceblue] (\x,0) circle (8pt) node[red] {\large $\bm{1}$};
                \draw[draw=black!50, fill=niceblue] (\x,1) circle (5pt) node[red] {\large $\bm{\eps}$};
                \draw[draw=black!50, fill=nicegreen] (\x,2) circle (5pt) node[red] {\large $\bm{\eps}$};
                \draw[draw=black!50, fill=nicegreen] (\x,3) circle (8pt) node[red] {\large $1$};
            }
    	\end{tikzpicture}
     \label{img:asymptotic_test_distribution:eps_small}
    }%
    \hspace{8mm}
    \subfloat[NCut, $\eps\geq 1$]{
    	\centering
    	\begin{tikzpicture}
    		\draw[black!40, line width=2mm] (0,0) -- (3,0);
    		\draw[black!40, line width=3.6mm] (0,1) -- (3,1);
    		\draw[black!40, line width=3.6mm] (0,2) -- (3,2);
    		\draw[black!40, line width=2mm] (0,3) -- (3,3);
            \foreach \x in {0,1,2,3} {
                \draw[black!40, line width=2.7mm] (\x,0) -- (\x,1);
                \draw[black!40, line width=3.6mm] (\x,1) -- (\x,2);
                \draw[black!40, line width=2.7mm] (\x,2) -- (\x,3);
            }
            \foreach \x in {0,1} {
                \draw[draw=black!50, fill=niceblue] (\x,0) circle (8pt) node[red] {\large $\bm{1}$};
                \draw[draw=black!50, fill=niceblue] (\x,1) circle (11pt) node[red] {\large $\bm{\eps}$};
                \draw[draw=black!50, fill=niceblue] (\x,2) circle (11pt) node[red] {\large $\bm{\eps}$};
                \draw[draw=black!50, fill=niceblue] (\x,3) circle (8pt) node[red] {\large $1$};
            }
            \foreach \x in {2,3} {
                \draw[draw=black!50, fill=nicegreen] (\x,0) circle (8pt) node[red] {\large $\bm{1}$};
                \draw[draw=black!50, fill=nicegreen] (\x,1) circle (11pt) node[red] {\large $\bm{\eps}$};
                \draw[draw=black!50, fill=nicegreen] (\x,2) circle (11pt) node[red] {\large $\bm{\eps}$};
                \draw[draw=black!50, fill=nicegreen] (\x,3) circle (8pt) node[red] {\large $1$};
            }
    	\end{tikzpicture}
     \label{img:asymptotic_test_distribution:eps_large}
    }
    \caption{Bimodal and unimodal distributions over a $4\times 4$ grid where only horizontally and vertically adjacent nodes are connected. Edge and vertex thicknesses are drawn proportional to their weight and that of the underlying distribution vector $\vec{p}$, respectively. Note that both are probability distributions, i.e.\ normalized before graph construction to enable comparison across different values of $\eps$. The optimal NCut partitions are visualized by colors (green or blue).}
    \label{img:asymptotic_test_distribution}
\end{figure}
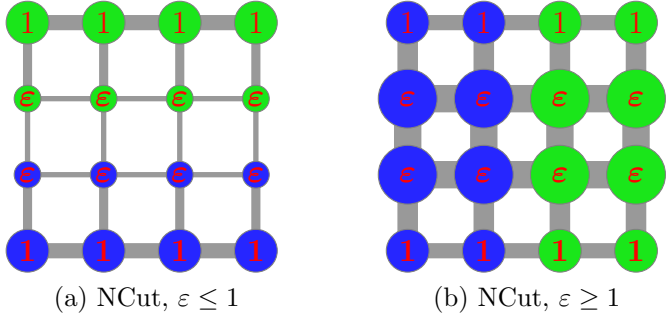

To illustrate the effectiveness of our testing procedure in distinguishing between uniform and a distribution with pronounced clusters, consider the bimodal distribution depicted in \cref{img:asymptotic_test_distribution}~(a). We seek to determine whether our test can correctly determine that this distribution is indeed non-uniform, and how this changes when $\eps$ approaches $1$, i.e.\ the underlying distribution is actually uniform, and when $\eps$ exceeds $1$, i.e.\ becomes unimodal with a large central cluster as in \cref{img:asymptotic_test_distribution}~(b).
As shown in \cref{img:asymptotic_test_over_eps}, the empirical rejection probability of our test is depicted over the parameter $\eps$ which determines the structure of the underlying distribution.
%
\begin{figure}
    \centering
    \includegraphics[width=0.6\linewidth]{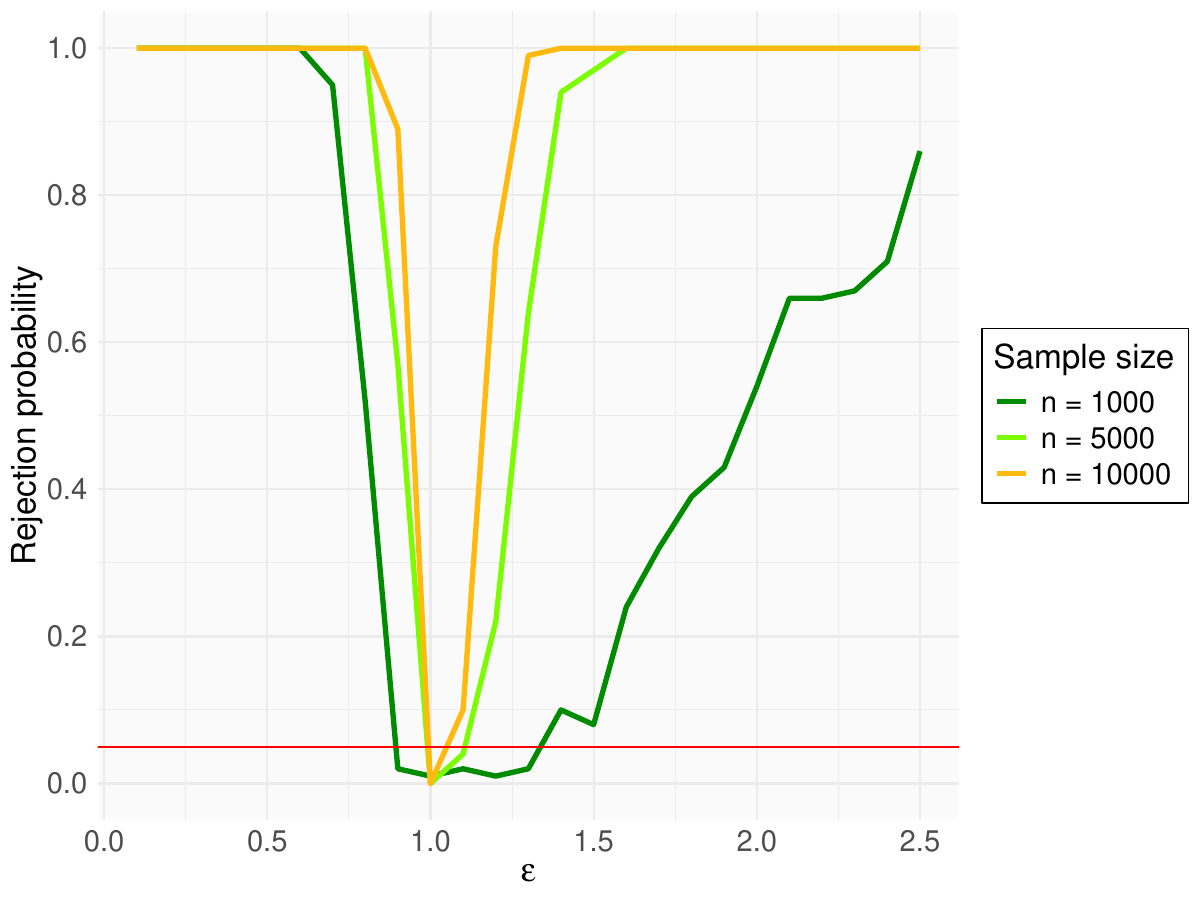}
    \caption{Empirical rejection probability of the asymptotic test \eqref{eq:asymptotic_test_xcut} for the distribution depicted in \cref{img:asymptotic_test_distribution}, plotted over the cluster separation parameter $\eps$, and separated according to sample size $n$. 
    The rejection probabilities are estimated by $100$ random repetitions. The horizontal red line marks the asymptotic significance level $\alpha = 0.05$.
    }
    \label{img:asymptotic_test_over_eps}
\end{figure}
Here, it is empirically evident that under $H_0$, i.e.\ for $\eps=1$, our test incorrectly rejects in less than $\alpha=0.05$ of the cases. 
The rejection probability returns to $1$ more slowly for large $\eps$ than for small $\eps$. The difference between the two cases is the transition between the two partitions $S$ and $T$ depicted in \cref{img:asymptotic_test_distribution}~(a) and \cref{img:asymptotic_test_distribution}~(b), respectively, that attain the NCut minimum for different values of $\eps$. The transition between $S$ and $T$ occurs at $\eps=1$ since for this choice of parameter the Minimum Cut value of both partitions is equal ($4\eps^2 = 2\eps^2 + 2$ up to normalization), and $\vol(S)=\vol(\comp{S})=\vol(T)=\vol(\comp{T})=\vol(V)/2$.
Hence, in \cref{img:asymptotic_test_over_eps} it is evident that our testing procedure can asymptotically detect the difference between two clusters and the uniform distribution very sharply, but is more hesitant to reject when there is one big central cluster which the resulting NCut cuts across. This demonstrates that balanced graph cuts are at their most powerful when used to distinguish between two clusters and either one cluster or no clear clustering structure.

\section{Conclusion and discussion}
\label{sec:outlook}
We have obtained explicit limiting
distributions for balanced graph cuts in a general discretized setup that is particularly suited to be make graph cuts more applicable. \cref{cor:conv_xcut} determined the limiting distribution for the most prominent graph cuts such as RCut (Ratio Cut), NCut (Normalized Cut) and CCut (Cheeger Cut), thereby revealing several case distinctions that decide whether the limit is the minimum of Gaussians or that of a random mixture of Gaussians. 
We stress that the discretization approach we presented can be easily generalized in several directions. Firstly, we chose edge weights $\widehat{w}_{ij}=\smash{Y_i Y_j\mathbbm{1}_{i\sim j}}$ to appropriately model the dependency between the discretized nodes. This notion can be extended to any Hadamard directionally differentiable function $f:\R^2\to\R$, i.e.\ by considering $\widehat{w}_{ij}=f(Y_i,Y_j)\mathbbm{1}_{i\sim j}$. Then, deriving the limiting distributions for these new edge weights is analogous to the procedure we introduced in the preceding as will become apparent by the proofs as presented in the \nameref{appendix}, although the resulting variances will obviously be different. What will not change, however, is the fact that the limit of $\sqrt{n}(\XC_S(\mathcal{G}_n) - \XC_S(\mathcal{G}))$ will still be Gaussian for MCut (Minimum Cut), RCut and NCut if $f$ has a linear Hadamard derivative. Again, CCut will not abide by a normal limit due to the non-differentiability of its balancing term. It is possible to generalize further, however, by only requiring $f$ to be Hadamard directionally differentiable. The appropriate approach then would be to apply the delta method for Hadamard directionally differentiable functions (\cref{prop:hadamard_delta_method}) analogously to the steps we undertook to handle CCut, thus possibly resulting in non-Gaussian limits for the aforementioned statistic. Making changes to the neighbourhood structure, i.e.\ replacing the $t$-neighborhood that influences the weights by means of $\mathbbm{1}_{i\sim j}$ also yields analogous results as long as the neighborhood does not depend on $\vec{Y}$. Otherwise, the neighborhood would need to be accounted for in the function $f$ as well, with similar implications (i.e.\ requiring Fr\'echet or Hadamard directional differentiability in order to apply our results).

Another direction is changing the discretization itself. In order to obtain explicit limits with the proposed approach it was necessary to fix the discretization $D=\sum_{i=1}^m D_i$. Now, as the sample size $n$ increases, it would be natural to increase the discretization resolution so that one is able to detect finer details, allowing for a more accurate partitioning of the resulting graph. Since this would violate the fixed structure of the underlying probability vector $\vec{p}\in[0,1]^m$ and fixed number $m\in\N$ of districts $D_i$, our multinomial model would not immediately apply. This scenario may be rather involved: Apart from apparent technical difficulties, it remains unclear whether there is still a reasonable limiting distribution for the cut values --- this will be postponed to subsequent work. 

\begin{appendices}
\section{Appendix}
\label{appendix}

In the following, the main results of \cref{sec:main_results,sec:applications} are proven. First, in \cref{apdx:sub:hadamard}, one of the main tools is the \nameref{prop:hadamard_delta_method}. In \cref{apdx:sub:main_proofs}, the main results of the paper, in particular \cref{thm:conv_mrcut_S} and \cref{cor:conv_xcut}, are shown as well as an extension of these results for the alternative NCut balancing terms in the form of \cref{cor:conv_ncut_alt} as mentioned at the beginning of \cref{sec:modelling}. Afterwards, in \cref{apdx:sub:proof_bootstrap_limit_xcut} we prove \cref{thm:bootstrap_limit_xcut}, i.e.\ the applicability of the bootstrap for our limit theorems. Finally, in \cref{apdx:sub:algorithm_proofs} the limit theorem for \xistref, \cref{thm:xist_limit}, is shown by first developing such a theorem for the \xvstnameref from \cite{SuchanLiMunk2023arXiv}. Additionally, the computation steps left out of \cref{exmp:xvst_practice} are shown in full detail.

\subsection{Hadamard directional differentiability}
\label{apdx:sub:hadamard}

\begin{definition}[\cite{Roemisch2006,Shapiro1990}]
	\label{def:hadamard_diff}
	A function $\vec{f}:\mathcal{D}\subseteq\R^{d}\to\R^{k}$ with $d,k \in \N$ is called \emph{Hadamard directionally differentiable} at $\vec{x}\in \dom$ in direction $\vec{h}\in\R^{d}$ if there is a (maybe non-linear) map $\vec{f}_{\vec{x}}^{\prime}:\R^{d}\to\R^{k}$ s.t.\
	\[
	\lim_{n\to\infty}\frac{\vec{f}(\vec{x}+t_n \vec{h_n}) - \vec{f}(\vec{x})}{t_n} = \vec{f}_{\vec{x}}^{\prime}(\vec{h})
	\]
	for all sequences $(\vec{h_n})_{n\in\N}$ converging to $\vec{h}$ in $\R^{d}$  and all sequences $(t_n)_{n\in\N}$ converging toward $0$ with $\vec{x} + t_n \vec{h_n}\in \dom$ and $t_n > 0$ for all $n\in\N$.
\end{definition}

In the following, whenever we call a function \emph{Hadamard differentiable at $\vec{x}$} we mean that it is Hadamard directionally differentiable at $\vec{x}$ in any direction $\vec{h}$. One important property of Hadamard directional differentiability is the fact that the derivative is not necessarily linear (see e.g.\ \mysubref{lem:min}{hadamard} later). Analogously to the usual notion of differentiability, though, one can state a \emph{chain rule} for Hadamard directionally differentiable functions. As we will utilize this chain rule, we will briefly state it here while referring to \citet{Shapiro1990} for a more detailed discussion about the differences between Hadamard and Fr\'{e}chet differentiability (which coincides with the usual differentiability on $\R^k$) in general.

\begin{lemma}[Chain rule for Hadamard directional derivatives]
	\label{lem:chain_rule_hadamard}
	Let $\vec{f}:\dom_{\vec{f}}\subseteq\R^{d}\to\R^{k}$ be Hadamard directionally differentiable at $\vec{x}\in \dom_{\vec{f}}$ and $\vec{g}:\dom_{\vec{g}}\subseteq\R^{k}\to\R^{l}$ be Hadamard directionally differentiable at $\vec{f}(\vec{x})\in \dom_{\vec{g}}$. Then $\vec{g}\circ \vec{f}:\dom_{\vec{f}}\to\R^{l}$ is Hadamard directionally differentiable at $\vec{x}$ with derivative
	\[
	(\vec{g}\circ \vec{f})_{\vec{x}}^{\prime} = \vec{g}_{\vec{f}(\vec{x})}^{\prime}\circ \vec{f}_{\vec{x}}^{\prime}\,.
	\]
\end{lemma}
\begin{proof}
    See \citet[Proposition~3.6 (i)]{Shapiro1990}.
\end{proof}

To reinforce our claim that we may not apply the usual delta method to the CCut functional, consider the following \cref{lem:min}. While its first two statements of continuity and differentiability are rather obvious (though still necessary to proceed), \mysubref{lem:min}{hadamard} computes the Hadamard directional derivative of the $\min$ function. Even though the proofs of this lemma are rather simple, we still present them for completeness. It should be noted that part \ref{lem:min:hadamard} also follows from the existing results in the literature (e.g.\ \citet[Proposition~1]{Roemisch2006} and \citet[Theorem~2.1]{CarcamoCuevasRodriguez2020}).

\begin{lemma}
	\label{lem:min}
	Let $d\in\N$. The minimum function 
	\[
	\min(\cdot):\R^d\to\R,\,\,\vec{x}=(x_1,\ldots,x_d)\mapsto\min\{x_1,\ldots,x_d\}
	\]
	has the following properties:
	\begin{enumerate}[label=(\roman*)]
		\item $\min(\cdot)$ is continuous. \label{lem:min:cont}
		\item $\min(\cdot)$ is Fr\'{e}chet differentiable at $\vec{x}=(x_1,\ldots,x_d)$ if and only if $\min\{x_1,\ldots,x_d\}$ is unique. \label{lem:min:diff}
		\item $\min(\cdot)$ is Hadamard directionally differentiable at every point $\vec{x}$ with derivative 
		\[
		\minn_{\vec{x}}^{\prime}(\vec{h}) = \min_{i=1,\ldots,d}\big\{h_i\,\,|\,\,x_i=\min(\vec{x})\big\}.
		\] \label{lem:min:hadamard}
	\end{enumerate}
	\vspace{-.8cm}
\end{lemma}
\begin{proof}
    \ref{lem:min:cont}: Trivial.

	\ref{lem:min:diff}: ($\impliedby$) We assume without loss of generality that $x_1$ is the unique minimum. Then in a sufficiently small neighborhood $U_{\vec{x}}$ of $\vec{x}$ it holds that $\min(\vec{z}) = z_1$ for $\vec{z} \in U_{\vec{x}}$, which is linear. Hence $\min(\cdot)$ is Fr\'{e}chet differentiable at $\vec{x}$.\\
	($\implies$) Assume that the minimum is not unique and without loss of generality is attained at $x_1$ and $x_2$. Then 
	\begin{align*}
	&\lim_{\eps \searrow 0} \frac{\min\{x_1+\eps, x_2,\ldots, x_d\} - \min(\vec{x})}{\eps} = 0 \\
    \text{and } &\lim_{\eps \searrow 0} \frac{\min\{x_1-\eps, x_2,\ldots, x_d\} - \min(\vec{x})}{\eps} = -1.
	\end{align*}
	This implies that $\min(\cdot)$ is not Fr\'{e}chet differentiable. 
	
	\ref{lem:min:hadamard}: For some arbitrary $\vec{x}\in\R^d$ we consider the limit
	\[
	\lim_{n\to\infty}\frac{\min(\vec{x}+t_n \vec{h}_n) - \min(\vec{x})}{t_n}
	\]
	where $\vec{h}_n\to\vec{h}$ and $t_n\to 0$ with $\vec{h}_n, \vec{h}\in\R^d$ and $t_n > 0$. Without loss of generality we assume that $\min(\vec{x})$ is attained at $x_1=\ldots=x_r$ for some $r\in\{1,\ldots,d\}$. Therefore, as $t_n > 0$, there exists an $N\in\N$ such that $\min(\vec{x} + t_n \vec{h}_n) = x_1 + t_n\min\{h_{n,j}\,\,|\,\,j=1,\ldots,r\}$ for all $n\geq N$, i.e.\
	\begin{align*}
	\minn_{\vec{x}}^{\prime}(\vec{h}) &= \lim_{n\to\infty}\frac{\min(\vec{x}+t_n \vec{h}_n) - \min(\vec{x})}{t_n} = \lim_{n\to\infty}\frac{t_n \min\{h_{n,j}\,|\,j=1,\ldots,r\}}{t_n}\\
	&= \min\{h_j\,|\,j=1,\ldots,r\}
	= \min_{i=1,\ldots,d}\big\{h_i\,|\,x_i=\min(\vec{x})\big\},
	\end{align*}
	where $h_{n,j}$ denotes the $j$-th component of $\vec{h}_n$ and $h_j$ is the $j$-th component of $\vec{h}$.
\end{proof}

The main application of Hadamard directional differentiability is the following generalization of the delta method which, as previously mentioned, we use to obtain a limiting result for CCut.

\begin{theorem}[Delta method for Hadamard directionally differentiable functions]
	\label{prop:hadamard_delta_method}
	Let $\vec{f}:\dom\subseteq\R^{d}\to\R^{k}$ be Hadamard directionally differentiable at $\vec{\theta}\in \dom$ with derivative $\vec{f}_{\vec{\theta}}^{\prime}$\,. Further let $(\vec{X}_n)_{n\in\N}$ be a sequence of random vectors with values in $\dom$ such that  
	\[
	r_n(\vec{X}_n - \vec{\theta})\dto \vec{X}
	\]
	for some random variable $\vec{X} \in \dom$  and real numbers $r_n\to\infty$. Then
	\[
	r_n	\big(\vec{f}(\vec{X}_n) - \vec{f}(\vec{\theta})\big)\dto \vec{f}_{\vec{\theta}}^{\prime}(\vec{X}).
	\]
\end{theorem}
\begin{proof}
    See \citet[Theorem~2.1]{Shapiro1991} or \citet[Proposition 1]{Roemisch2006}.
\end{proof}

\subsection{Proof of the main results}
\label{apdx:sub:main_proofs}

Now define
\[
S_{\ni i}\coloneqq \begin{cases}
S, &i\in S,\\
\comp{S}, &i\in \comp{S},
\end{cases}
\quad\text{and}\quad
S_{\not\ni i}\coloneqq \comp{(S_{\ni i})}=\begin{cases}
\comp{S}, &i\in S,\\
{S}, &i\in \comp{S}.
\end{cases}
\]
The usual delta method requires $\R^d$-differentiability of the function~$\vec{f}$. For MCut, RCut and NCut, their objective functions $g_S^{\XC}$ satisfy this condition. Consequently, it is sufficient to compute their partial derivatives, which yields the first part of the proof below. CCut is more involved as, according to \mysubref{lem:min}{diff}, its balancing term is not differentiable. However, we will show Hadamard directional differentiability of $g_S^{\CC}$ and compute the corresponding derivative. This will allow us to apply \cref{prop:hadamard_delta_method}.

\begin{proof}[Proof of \cref{thm:conv_mrcut_S}]
	We rewrite $\vec{Y} \sim \mult(n, \vec{p})$, $\vec{Y} \in\R^m$, with $\vec{p}=(p_1,\ldots, p_m)$ as a sum of $n$ independent multivariate Bernoulli variables $\vec{Z}_i \sim \mult(1, \vec{p})$ with $\vec{Z}_i\in \R^m$, i.e.\ $\vec{Y} \eqd \sum_{i=1}^n \vec{Z}_i$\,. Thus, the multivariate Central Limit Theorem yields
	\begin{equation}
	\sqrt{n}\Bigl(\tfrac{1}{n}\vec{Y} - \vec{p}\Bigr) = \sqrt{n}\biggl(\frac{1}{n}\sum_{i=1}^n \vec{Z}_i - \vec{p}\biggr) \dto \mathcal{N}_m(\vec{0},\mat{\Sigma}),
	\label{eq:clt_Y}
	\end{equation}
	where $\mat{\Sigma}=(\Sigma_{i,j})\in \R^{m \times m}$ is the covariance matrix of $\vec{Z}_1$, i.e.\
	\[
	\Sigma_{i,j} = \begin{cases}
	-p_i p_j & \text{if }\,i \neq j \\
	p_i(1-p_i) & \text{if }\,i = j
	\end{cases}
	\qquad \text{for}\quad i,j = 1, \ldots, m\,.
	\]
	Recall that for any XCut and any partition $S\in\mathcal{S}$ the XCut objective function is
	\begin{alignat*}{2}
	&\vec{g}^{\XC}:\R^m\to\R^{|\mathcal{S}|},\quad&&\vec{x}\,\mapsto\, \left(g_S^{\XC}(\vec{x})\right)_{S\in\mathcal{S}}\\
	\text{for}\quad&g_S^{\XC}:\R^m\to\R, &&\vec{x}\,\mapsto\,\frac{1}{\ball_S(\vec{x})} \sum_{\substack{i\in S, j\in \comp{S}\\i\sim j}} x_i x_j,
	\end{alignat*}
	where $\vec{x}=(x_1,\ldots,x_m)$, and $\ball:\R^m\to\R$ is the balancing term functional.
	
	\ref{thm:conv_mrcut_S:Hadamard}: We first treat the case where the XCut objective function $\smash{g_S^{\CC}}$ is Hadamard directionally differentiable at $\vec{p}$. Define
    \begin{alignat*}{2}
    \vec{\varphi} &:\,\R^{m}\to\R^{m+1},\; &&\vec{x}\mapsto \big(x_1,\ldots,x_m,\ball_S^{\XC}(\vec{x})\big) \\
	\text{ and }\quad\psi &:\,\R^{m+1}\to\R,\; &&\vec{x}\mapsto \sum_{{i\in S, j\in \comp{S}, \;i\sim j}} \frac{x_i x_j}{x_{m+1}}.
    \end{alignat*}
	in order to decompose the XCut objective function $g_S^{\XC}(\vec{x}) = (\psi\circ \vec{\varphi})(\vec{x})$ and apply \cref{lem:chain_rule_hadamard}. To that end, compute the Hadamard directional derivatives of $\psi$ and $\vec{\varphi}$ at $\vec{x}$ in direction $\vec{h}$:
	\begin{align*}
	\vec{\varphi}_{\vec{x}}^{\prime}(\vec{h}) &= \left(h_1,\,\ldots,\,h_m,\,\bigl(\ball_S^{\XC}\bigr)_{\vec{x}}'(\vec{h})\right),\\
	\psi_{\vec{x}}'(\vec{h}) &= \frac{1}{x_{m+1}^2}\sum_{{i\in S, j\in \comp{S},\; i\sim j}} \Big((h_i x_j + h_j x_i) x_{m+1} - h_{m+1} x_i x_j\Big),
	\end{align*}
	Finally,
	\begin{align*}
	(g_S^{\XC})_{\vec{x}}'(\vec{h}) &= (\psi\circ \vec{\varphi})_{\vec{x}}^{\prime}(\vec{h}) \\
    &= \psi_{\vec{\varphi}(\vec{x})}'\circ\vec{\varphi}_{\vec{x}}'(\vec{h}) \\
	&= \frac{1}{\ball_S^{\XC}(\vec{x})} \Biggl[\smash{\sum_{{i\in S, j\in \comp{S},\; i\sim j}}}\bigl( h_i x_j + h_j x_i\bigr) - g_S^{\XC}(\vec{x})\,\bigl(\ball_S^{\XC}\bigr)_{\vec{x}}'(\vec{h})\Biggr],
	\end{align*}
    so that \cref{prop:hadamard_delta_method} gives
	\[
	\sqrt{n}\left(g_S^{\XC}\left(\tfrac{1}{n}\vec{Y}\right) - g_S^{\XC}(\vec{p})\right)\dto \left((g_S^{\XC})_{\vec{p}}'(\vec{Z})\right)_{S\in\mathcal{S}}\quad\text{where}\quad \vec{Z}\sim\mathcal{N}_m(\vec{0},\mat{\Sigma}).
	\]
	As we already computed $(g_S^{\XC})_{\vec{x}}'(\vec{h})$, we simply plug in $\vec{Z}$ and $\vec{p}$ to obtain the limiting distribution of
	\begin{align*}
	(g_S^{\XC})_{\vec{p}}'(\vec{Z}) &= 
	\frac{1}{\bal(S,\comp{S})} \Biggl[\biggl(\sum_{i=1}^m Z_i q_{i,S}\biggr) - \XC_S(\mathcal{G})\,\bigl(\ball_S^{\XC}\bigr)_{\vec{p}}'(\vec{Z})\Biggr].
	\end{align*}
	
	\ref{thm:conv_mrcut_S:diff}: When $g_S^{\XC}$ is differentiable (in the usual sense), we can apply the standard delta method:
	\begin{align*}
	\left(\sqrt{n}\left(\XC_S(\mathcal{G}_n) - \XC_S(\mathcal{G})\right)\right)_{S\in\mathcal{S}} &= \left(\sqrt{n}\Big(g_S^{\XC}\left(\tfrac{1}{n}\vec{Y}\right)-g_S^{\XC}(\vec{p})\Big)\right)_{S\in\mathcal{S}} \\
    &\dto \mathcal{N}_{\abs{\mathcal{S}}}\left(\vec{0},\mat{\Sigma}^{\XC}\right),
	\end{align*}
	where $\mat{\Sigma}^{\XC}\coloneqq \mat{J}_{\vec{g}^{\XC}}\mat{\Sigma} \mat{J}_{\vec{g}^{\XC}}^{\top}$ with the Jacobian matrix $\mat{J}_{\vec{g}^{\XC}}\in \R^{\abs{\mathcal{S}}\times m}$ of $\vec{g}^{\XC}$. Consequently, a straightforward calculation reveals that for any $S,T\in\mathcal{S}$ the limiting covariance between $\sqrt{n}\bigl(\XC_S(\mathcal{G}_n) - \XC_S(\mathcal{G})\bigr)$ and $\sqrt{n}\bigl(\XC_T(\mathcal{G}_n) - \XC_T(\mathcal{G})\bigr)$ is given by  
	\begin{align*}
	\bigl(\mat{\Sigma}^{\XC}\bigr)_{S,T} &= 
	\left(\mat{J}_{\vec{g}^{\XC}}\mat{\Sigma} \mat{J}_{\vec{g}^{\XC}}\right)_{S,T} \\
    &= \sum_{i=1}^n p_i \frac{\partial g_S^{\XC}}{\partial p_r}(\vec{p}) \frac{\partial g_T^{\XC}}{\partial p_r}(\vec{p}) - \biggl(\sum_{i=1}^n p_i \frac{\partial g_S^{\XC}}{\partial p_r}(\vec{p})\biggr)\biggl(\sum_{i=1}^n p_i \frac{\partial g_T^{\XC}}{\partial p_r}(\vec{p})\biggr),
	\end{align*}
     using the fact that $\mat{\Sigma}=\diag(\vec{p}) - \vec{p}^{\top}\vec{p}$, where
    \[
    \frac{\partial g_S^{\XC}}{\partial p_r}(\vec{p}) = \frac{q_{r,S} - \left(\tfrac{\partial}{\partial p_r} \ball_S^{\XC}(\vec{p})\right)\XC_S(\mathcal{G})}{\ball_S^{\XC}(\vec{p})}.
    \]
	Now note that for any $S\in\mathcal{S}$,
	\[
	\frac{\sum_{i=1}^m p_i q_{i,S}}{\ball_S^{\XC}(\vec{p})} = \frac{\sum_{i \in S} p_i q_{i,S} + \sum_{i \in \comp{S}}p_i q_{i,S}}{\ball_S^{\XC}(\vec{p})} = \XC_S(\mathcal{G}) + \XC_S(\mathcal{G}) = 2\XC_T(\mathcal{G}).
	\]
    so that finally
	\begin{align*}
	&\nabla g_T^{\XC}(\vec{p})^{\top} \mat{\Sigma}\,\, \nabla g_S^{\XC}(\vec{p}) 
	\cdot\left(\frac{\sum_{i=1}^n p_i q_{i,S}}{\ball_S^{\XC}(\vec{p})} - \frac{\sum_{i=1}^n p_i \left(\tfrac{\partial}{\partial p_i} \ball_S^{\XC}(\vec{p})\right)\XC_S(\mathcal{G})}{\ball_S^{\XC}(\vec{p})}\right) \\
	= & \left(\sum_{i=1}^n p_i \left(\frac{q_{i,T} - \left(\tfrac{\partial}{\partial p_i} \ball_T^{\XC}(\vec{p})\right)\XC_T(\mathcal{G})}{\ball_T^{\XC}(\vec{p})}\right) \left(\frac{q_{i,S} - \left(\tfrac{\partial}{\partial p_i} \ball_S^{\XC}(\vec{p})\right)\XC_S(\mathcal{G})}{\ball_S^{\XC}(\vec{p})}\right)\right) \\
	&\; - \XC_T(\mathcal{G})\XC_S(\mathcal{G})\left(2-\frac{\displaystyle\sum_{i=1}^n p_i \left(\tfrac{\partial}{\partial p_i} \ball_T^{\XC}(\vec{p})\right)}{\ball_T^{\XC}(\vec{p})}\right) \left(2-\frac{\displaystyle\sum_{i=1}^n p_i \left(\tfrac{\partial}{\partial p_i} \ball_S^{\XC}(\vec{p})\right)}{\ball_S^{\XC}(\vec{p})}\right).
	\end{align*}
\end{proof}	


\begin{proof}[Proof of \cref{cor:mrnccut_explicit}]
    Consider now the limiting distribution from \cref{thm:conv_mrcut_S} for general balancing terms and recall the definitions of MCut and RCut from \cref{tbl:cuts}.
    
	\ref{cor:mrnccut_explicit:diff}: MCut and RCut satisfy $\tfrac{\partial}{\partial p_r} \ball_S^{\MC}(\vec{p}) = \tfrac{\partial}{\partial p_r} \ball_S^{\RC}(\vec{p}) = 0$ for $r\in\{1,\ldots,m\}$, $S\in\mathcal{S}$. Thus,
	\begin{align*}
	\mat{\Sigma}^{\MC}_{T,S} &= \sum_{i = 1}^m p_i q_{i,T}q_{i,S} - 4\MC_T(\mathcal{G})\MC_S(\mathcal{G}) \;
	\text{ and }\; \\
    \mat{\Sigma}^{\RC}_{T,S} &= \sum_{i=1}^m \frac{p_i q_{i,T} q_{i,S}}{\abs{T}\abs{\comp{T}}\abs{S}\abs{\comp{S}}} - 4 \RC_T(\mathcal{G})\RC_S(\mathcal{G})
	\end{align*}
	for any two partitions $S,T\in\mathcal{S}$. Similarly, for NCut and any $S\in\mathcal{S}$ we compute
	\[
	\frac{\partial}{\partial p_r} \ball_S^{\NC}(\vec{p}) = q_{r,S}\bigl(\vol(S_{\ni r}) - \vol(S_{\not\ni r})\bigr) + 2\vol(S_{\not\ni r})\sum_{j:\,j\sim r} p_j
	\]
	for any $r\in\{1,\ldots,m\}$ and note that
%
%
    \begin{align*}
        \sum_{i = 1}^m p_i \bigl(\tfrac{\partial}{\partial p_i} \ball_S^{\NC}(\vec{p})\bigr) &= \sum_{i=1}^m p_i \bigl(q_{i,S}\bigl(\vol(S_{\ni i}) - \vol(S_{\not\ni i})\bigr) + 2\vol(S_{\not\ni i})\sum_{j\sim i} p_j\bigr) \\
        & = \sum_{i \in S} p_i q_{i,S} \bigl(\vol(S) - \vol(\comp{S})\bigr) + \sum_{i \in \comp{S}} p_i q_{i,S} \bigl(\vol(\comp{S}) - \vol({S})\bigr) \\
        & \qquad + 2\vol(\comp{S})\sum_{i \in S}\sum_{j \sim i} p_i p_j+ 2\vol({S})\sum_{i \in \comp{S}}\sum_{j \sim i} p_i p_j \\
        & = \MC_S(\mathcal{G})\bigl(\vol(S) - \vol(\comp{S})\bigr) + \MC_S(\mathcal{G})\bigl(\vol(\comp{S}) - \vol({S})\bigr) \\
         & \qquad + 2\vol(\comp{S})\vol(S) + 2\vol({S})\vol(\comp{S}) \\
         & = 4 \vol(S)\vol(\comp{S}),
    \end{align*}
    for any $S\in\mathcal{S}$. Defining $r_{i,S} \coloneqq \vol(S_{\ni i}) + 3\vol(S_{\not\ni i})$ and plugging everything into the general expression from above we finally obtain
    \[
    \mat{\Sigma}_{T,S}^{\NC} = \sum_{i=1}^m\frac{p_i q_{i,T} q_{i,S}\bigl(1 - \NC_T(\mathcal{G})r_{i,T}\bigr)\bigl(1 - \NC_S(\mathcal{G})r_{i,S}\bigr)}{\vol(T)\vol(\comp{T})\vol(S)\vol(\comp{S})} - 4\NC_T(\mathcal{G})\NC_S(\mathcal{G}).
    \]
	
	\ref{cor:mrnccut_explicit:cheeger} To tackle the special case of CCut we need to compute $\bigl(\ball_S^{\CC}\bigr)_{\vec{x}}'(\vec{h})$. We do this by decomposing $\ball_S^{\CC}(\vec{x}) = (\vec{\xi}\circ\vec{\chi})(\vec{x})$, where
	\begin{alignat*}{2}
    \vec{\chi}:&\,\R^m\to\R^{m+2},\quad &&\vec{x}\mapsto \biggl(x_1,\ldots,x_m,\sum_{i\in S}\sum_{j:\,j\sim i} x_i x_j,\sum_{i\in \comp{S}}\sum_{j:\,j\sim i} x_i x_j\biggr) \\
	\text{and}\quad \vec{\xi}:&\,\R^{m+2}\to\R^{m+1},\quad &&\vec{x}\mapsto \big(x_1,\ldots,x_m,\min\{x_{m+1},x_{m+2}\}\big).
	\end{alignat*}
	Their Hadamard directional derivatives are give by
	\begin{align*}
	\vec{\chi}_{\vec{x}}^{\prime}(\vec{h}) &= \left(h_1,\,\ldots,\,h_d,\,\sum_{i\in S}\sum_{j\sim i} (h_i x_j + h_j x_i), \sum_{i\in \comp{S}}\sum_{j\sim i} (h_i x_j + h_j x_i)\right) \\
	\text{and}\quad \vec{\xi}_{\vec{x}}^{\prime}(\vec{h}) &= \left(h_1,\,\ldots,\,h_m,\,\minn_{(x_{m+1},x_{m+2})}^{\prime}(h_{m+1},h_{m+2})
	\right),
	\end{align*}
	where the Hadamard directional derivative of the minimum function is given by 
	\[
	\minn_{(x_{m+1},x_{m+2})}^{\prime}(h_{m+1},h_{m+2}) = 
	\begin{cases}
	h_{m+1}, & \text{if } x_{d+1}< x_{m+2},\\
	h_{m+2}, & \text{if } x_{d+1} > x_{m+2},\\
	\min\{h_{m+1}, h_{m+2}\}, & \text{if } x_{m+1} = x_{m+2},
	\end{cases}
	\]
	according to \mysubref{lem:min}{hadamard}. Combining the derivatives of $\vec{\chi}$ and $\vec{\xi}$ by means of \cref{prop:hadamard_delta_method} and plugging in $\vec{x}=\vec{p}$ and $\vec{h}=\vec{Z}$, we receive
	$\bigl(\ball_S^{\CC}\bigr)_{\vec{p}}'(\vec{Z}) =\,\,\vec{\xi}_{\vec{\chi}(\vec{p})}'(\chi_{\vec{p}}'(\vec{Z})) = \sum_{i\in S_{\vec{p},\vec{Z}}}\sum_{j\sim i} Z_i p_j + Z_j p_i $ for
 \begin{equation*}
     {S}_{\vec{p},\vec{Z}} \coloneqq 
	\begin{cases}
    	S, &\text{if }\vol(S) < \vol(\comp{S});\\
    	\comp{S}, &\text{if }\vol(S) > \vol(\comp{S});\\
    	\displaystyle\argmin_{T=S,\comp{S}}\Bigl\{\sum_{i\in T}\sum_{j\sim i} (Z_i p_j + Z_j p_i)\Bigr\}, &\text{if }\vol(S) = \vol(\comp{S}).
    \end{cases}
 \end{equation*}

	This means that $S_{\vec{p},\vec{Z}}$ assumes the role of the partition (either $S$ or $\comp{S}$) with minimal volume of the two. Only if both their volumes are equal, $S_{\vec{x},\vec{h}}$ will depend on the vectors $\vec{x}$ and $\vec{h}$. This fact will be important later. Finally, plugging in everything yields
    \begin{align*}
	Z_S^{\CC} \coloneqq  (g_S^{\CC})_{\vec{p}}^{\prime}(\vec{Z}) &= \frac{1}{\vol(S_{\vec{p},\vec{Z}})} \Biggl[\biggl(\sum_{i=1}^m Z_i q_{i,S}\biggr) - \CC_S(\mathcal{G})\,\Bigl(\sum_{i\in S_{\vec{p},\vec{Z}}}\sum_{j\sim i} Z_i p_j + Z_j p_i\Bigr)\Biggr] \\
	&= \sum_{i=1}^m \frac{Z_i}{\vol({S}_{\vec{p},\vec{Z}})} \Big(q_{i,S} - \CC_S(\mathcal{G}) (a_i\mathbbm{1}_{i\in{S}_{\vec{p},\vec{Z}}} + b_{i,{S}_{\vec{p},\vec{Z}}})\Big)
	\end{align*}
	for $a_{i} = \sum_{j \sim i} p_j$ and $b_{i,S} = \sum_{j \in S, \, j\sim i} p_j$.
	
	\ref{cor:mrnccut_explicit:cheegerunique}: Now assume that $\vol(S)\neq\vol(\comp{S})$ and $\vol(T)\neq\vol(\comp{T})$. Then, ${S}_{\vec{p},\vec{Z}}$ and ${T}_{\vec{p},\vec{Z}}$ do not depend on $\vec{Z}$ as is evident from their definition above, and are uniquely determined by the deterministic $\vec{p}$ alone. As the sum of deterministically weighted Gaussians is again normally distributed, and the expectation of $(g_S^{\CC})_{\vec{p}}^{\prime}(\vec{Z})$ is obviously zero, it remains to compute the covariance structure. Note that
    $\sum_{i=1}^m b_{i,S}\,p_i =	\sum_{j \in S} a_j p_j = \sum_{{{i\in V\, j\in S, i \sim j}}} p_i p_j = \vol(S).$
    Recall that $g_S^{\CC}(\vec{Z}) = \vec{v}_S \vec{Z}^{\top}$ with 
	\[
    \vec{v}_S \coloneqq  \frac{1}{\vol(S_{\vec{p},\vec{Z}})}\Bigl(q_{i,S_{\vec{p},\vec{Z}}} - g_S^{\CC}(\vec{p})\big(a_i\mathbbm{1}_{i\in S_{\vec{p},\vec{Z}}} + b_{i,S_{\vec{p},\vec{Z}}}\big)\Bigr)_{i=1,\ldots,m}\,.
    \]
    It holds that $g_S^{\MC}(\vec{p}) = g_{S_{\vec{p},\vec{Z}}}^{\MC}(\vec{p})$ since $S_{\vec{p},\vec{Z}}$ is either $S$ or $\comp{S}$. Thus,
	\begin{align*}
	\vec{v}_S\vec{p}^{\top} &= \frac{1}{\vol(S_{\vec{p},\vec{Z}})}\Big[\Big(\sum_{i=1}^m p_i q_{i,S_{\vec{p},\vec{Z}}}\Big) - \CC_S(\mathcal{G})(\vec{p})\Big(\sum_{i=1}^m p_i (a_i\mathbbm{1}_{i\in S_{\vec{p},\vec{Z}}} + b_{i,S_{\vec{p},\vec{Z}}})\Big)\Big]\\
	&= \frac{1}{\vol(S_{\vec{p},\vec{Z}})}\Big[2\MC_S(\mathcal{G}) - 2\vol(S_{\vec{p},\vec{Z}})\CC_S(\mathcal{G})\Big] = 0.
	\end{align*}
	Finally, we have
	\begin{align*}
	&\cov\left(Z_T^{\CC}, Z_S^{\CC}\right) = \vec{v}_T \mat{\Sigma} \vec{v}_S^{\top} = \vec{v}_T\, \diag(\vec{p})\,\vec{v}_S^{\top} - \vec{v}_T\, \vec{p}^{\top}\vec{p}\, \vec{v}_S^{\top} \\
	&\qquad= \frac{1}{\vol(T_{\vec{p},\vec{Z}})\vol(S_{\vec{p},\vec{Z}})} \sum_{i=1}^m p_i\Bigl(q_{i,T_{\vec{p},\vec{Z}}} - \CC_S(\mathcal{G})\bigl(a_i\mathbbm{1}_{i\in T_{\vec{p},\vec{Z}}} + b_{i,T_{\vec{p},\vec{Z}}}\bigr)\Bigr) \\
    &\hspace{6cm}\cdot\Bigl(q_{i,S_{\vec{p},\vec{Z}}} - \CC_S(\mathcal{G})\bigl(a_i\mathbbm{1}_{i\in S_{\vec{p},\vec{Z}}} + b_{i,S_{\vec{p},\vec{Z}}}\bigr)\Bigr).
	\end{align*}
\end{proof}

Notice that for CCut, if $\vol(S)=\vol(\comp{S})$ or $\vol(T)=\vol(\comp{T})$ for $S,T\in\mathcal{S}$, the asymptotic distribution from \mysubref{cor:mrnccut_explicit}{cheeger} is not Gaussian. This is due to the fact that for $\vol(S)\neq \vol(\comp{S})$,
\[
S_{\vec{p},\vec{Z}} = \displaystyle\argmin_{R=S,\comp{S}}\Bigl\{\sum_{i\in R}\sum_{j\sim i} (Z_i p_j + Z_j p_i)\Bigr\}
\]
is a non-deterministic random variable as it depends on $\vec{Z}\sim\mathcal{N}_m(0,\mat{\Sigma})$ itself. This results in a more complex covariance structure that we prefer not to compute explicitly.

\begin{proof}[Proof of \cref{cor:conv_xcut}]
	By \cref{thm:conv_mrcut_S} it holds that
	\[
	\left(\sqrt{n}\left(\XC_S(\mathcal{G}_n) - \XC_S(\mathcal{G})\right)\right)_{S\in\mathcal{S}}\dto \left(Z_S^{\XC}\right)_{S\in\mathcal{S}}.
	\]
	Now define $d\coloneqq |\mathcal{S}|\equiv\#\{S\in\mathcal{S}\}$ and consider the function $\min(\cdot):\R^d\to\R,\;(x_1,\ldots,x_d)\mapsto\min_{i=1,\ldots,d}x_i\,.$
	By \mysubref{lem:min}{hadamard}
	the Hadamard directional derivative of $\min(\cdot)$ at $\vec{u}\in\R^d$ is 
	\[
	\minn_{\vec{u}}^{\prime}(\vec{h}) = \min_{i=1,\ldots,d}\big\{h_i : u_i = \min(\vec{u})\big\}.
	\]
	Thus, the generalized version of delta method (\cref{prop:hadamard_delta_method}) yields
	\[
	\sqrt{n}\left(\XC(\mathcal{G}_n) - \XC(\mathcal{G})\right) \dto \min_{S\in\mathcal{S}_{*}} Z^{\XC}_S
	\quad \text{for}\quad\mathcal{S}_{*} \coloneqq  \{S\in\mathcal{S} :  \XC_S(\mathcal{G})=\XC(\mathcal{G})\}.
	\]
\end{proof}

Now, we state \cref{thm:conv_mrcut_S,cor:conv_xcut} for the alternative version of NCut, mentioned in \cref{sec:modelling}.

\begin{corollary}
Consider NCut in its alternative form, i.e.\ with cut value
	\begin{align*}
	\NC_S^*(\mathcal{G})&\coloneqq \sum_{\mathclap{\substack{i\in S, j\in \comp{S}\\i\sim j}}} p_i p_j \left(\tfrac{1}{\vol(S)} + \tfrac{1}{\vol(\comp{S})}\right) = \vol(V) \NC_S(\mathcal{G}). \\
	\intertext{and its empirical version}
	\NC_S^*(\mathcal{G}_n)&\coloneqq \sum_{\mathclap{\substack{i\in S, j\in \comp{S}\\i\sim j}}} Y_i Y_j \left(\tfrac{1}{\widehat{\vol}(S)} + \tfrac{1}{\widehat{\vol}(\comp{S})}\right) = \frac{\widehat{\vol}(V)}{n^2} \,\NC_S(\mathcal{G}_n)\,,
	\end{align*}
	where $\widehat\vol(S) \coloneqq  \sum_{i \in S}\sum_{j \sim i} Y_i Y_j$ and, as before, $\vol(S) = \sum_{i \in S}\sum_{j \sim i} p_i p_j$ for $S\in\mathcal{S}$\,. For this (scale invariant) version of NCut, the statements of \cref{thm:conv_mrcut_S,cor:conv_xcut} still hold true, although the limiting covariance is slightly different, namely
	\begin{align*}
	\cov(Z_T^{\NC^*}, Z_S^{\NC^*}) &= \vol^2(V) \cov(Z_T^{\NC}, Z_S^{\NC}) \\
	&= \Biggl(\sum_{i=1}^m p_i q_{i,T} q_{i,S} \bigl(1 - \NC_T(\mathcal{G}) r_{i,T}\bigr)\bigl(1 - \NC_S(\mathcal{G}) r_{i,S}\bigr)\Biggr) \\
    &\quad\cdot\left(\tfrac{1}{\vol(T)} + \tfrac{1}{\vol(\comp{T})}\right) \left(\tfrac{1}{\vol(S)} + \tfrac{1}{\vol(\comp{S})}\right) - 4 \NC_T^*(\mathcal{G})\NC_S^*(\mathcal{G})\,.
	\end{align*}\\[-1cm]
	\label{cor:conv_ncut_alt}
\end{corollary}
\begin{proof}
    In \cref{thm:conv_mrcut_S} (and, by extension, \cref{cor:conv_xcut}) we have shown that
	\[
	\left(\sqrt{n}\big(\NC_S(\mathcal{G}_n) - \NC_S(\mathcal{G})\big)\right)_{S\in\mathcal{S}} \dto \left(Z_S^{\NC}\right)_{S\in\mathcal{S}} =: \vec{Z}^{\NC} \sim \mathcal{N}_{\abs{\mathcal{S}}}(\vec{0},\mat{\Sigma}^{\NC})\,.
	\]
	The law of large numbers dictates that $n^2\,\widehat{\vol}(S)\stackrel{\text{a.s.}}{\longrightarrow}\vol(S)$\,. Combining both of these results using Slutsky's lemma leads to
	\[
	\biggl(\sqrt{n}\big(\NC_S^*(\mathcal{G}_n) - \frac{\widehat\vol(V)}{n^2 \vol(V)}\NC_S^*(\mathcal{G})\big)\biggr)_{S\in\mathcal{S}} \dto \;\vol(V)\vec{Z}^{\NC},
	\]
    where we used the fact that $\NC_S^*(\mathcal{G}) / \vol(V) = \NC_S(\mathcal{G})$\,. We now utilize the delta method in an analogous manner to the proof of \cref{thm:conv_mrcut_S}, i.e.\ by exploiting differentiability of the (slightly modified) objective function
    \[
    g_S^{\NC^*}:\R^m\to\R,\quad (x_1,\ldots,x_m)\mapsto\left(\frac{1}{\sum_{i\in S}\sum_{j\sim i} x_i x_j} + \frac{1}{\sum_{i\in\comp{S}}\sum_{j\sim i} x_i x_j}\right)\hspace{.2cm} \sum_{\mathclap{\substack{i\in S,\,j\in \comp{S}\\i\sim j}}} x_i x_j.
    \]
    This, together with convergence of $\vec{Y}/n\dto\vec{p}$ and the aforementioned delta method yields
	\[
	\left(\sqrt{n}\big(\NC_S^*(\mathcal{G}_n) - \NC_S^*(\mathcal{G})\big)\right)_{S\in\mathcal{S}} \dto \widetilde{\vec{Z}}\,,
	\]
	for some random vector $\widetilde{\vec{Z}}$. Now, by the law of large numbers and the continuous mapping theorem,
	${\widehat\vol(V)}/\bigl({n^2 \vol(V)}\bigr) \to 1$ a.s.
	Then, by the convergence of types theorem (e.g.\ \cite[Theorem~8.32]{Breiman1968}),
	\[
	\widetilde{\vec{Z}} \eqd \vol(V)\,\vec{Z}^{\NC} \quad\text{and thus}\quad \cov(Z_T^{\NC^*}, Z_S^{\NC^*}) = \vol^2(V) \cov(Z_T^{\NC}, Z_S^{\NC})\,.
	\]
\end{proof}

\subsection{Proof of \texorpdfstring{\cref{thm:bootstrap_limit_xcut}}{Theorem \ref{thm:bootstrap_limit_xcut}}}
\label{apdx:sub:proof_bootstrap_limit_xcut}


In order to show \cref{thm:bootstrap_limit_xcut} we utilize two theorems by \citet{Duembgen1993} as well as an auxiliary result provided by \citet{PolitisRomanoWolf1999} in their comprehensive book on subsampling. We will state all these results for completeness and to transfer notation, in particular to make \citet{Duembgen1993} more accessible. The latter also requires the concept of \emph{weak convergence in probability} which we will briefly introduce here first.

\begin{definition}
    Let $(\Omega_n,\P_n)$ be probability spaces, and let $\widehat{L}_n \coloneqq  \widehat{L}_n(\cdot\mid\omega_n)$ be random distributions on the metric space $(\mathcal{D},d)$, that is, $\widehat{L}_n$ is a probability distribution on $\mathcal{D}$ for $\P_n$-almost all $\omega_n\in\Omega_n$. Then, $\widehat{L}_n$ is defined to \emph{converge weakly in probability} towards a Borel distribution $L$ on $\mathcal{D}$ if
    \begin{equation}
    \E_{Z_n\sim\widehat{L}_n}\bigl[f(Z_n)\bigr] \stackrel{\P_n}{\longrightarrow} \E_{Z\sim L}\bigl[f(Z)\bigr]\quad\text{for all }f\in C(\mathcal{D},[0,1]),
    \label{apdx:eq:weak_convergence_in_probability}
    \end{equation}
    where $C(\mathcal{D},[0,1])$ denotes the set of continuous functions from $\mathcal{D}$ to $[0,1]$, and $\stackrel{\P_n}{\longrightarrow}$ is the convergence in probability w.r.t.\ $\P_n$.
    \label{apdx:def:weak_convergence_in_probability}
\end{definition}

By the Portmanteau theorem, the above convergence of expectations of continuous functions could also be expressed as convergence in distribution of $\widehat{L}_n$ to $L$, although the fact that this only needs to hold in probability w.r.t.\ $L_n$ necessitates the slightly more inconvenient notation in terms of all $f\in C(\mathcal{D},[0,1])$ as presented above.

With this definition we are ready to state the necessary auxiliary results for \cref{thm:bootstrap_limit_xcut}.

\begin{proposition}
    Let \firstassumptions hold and assume that the law of $\sqrt{M_n}(\vec{Y}^*/M_n - \vec{Y}/n)$ (conditioned on $\vec{Y}$) converges weakly in probability towards $\vec{Z}\sim\mathcal{N}_m(\vec{0},\mat{\Sigma})$. Also let $f:\R^m\to\R$ be any function that is Hadamard directionally differentiable at $\vec{p}$ in almost any direction. Then, there exists a sequence of sets $(A_n)_{n\in\N}\subseteq\R^d$ with $\P(\vec{Y}\in A_n)\to 1$ as $n\to\infty$ such that the following holds:
    \begin{enumerate}[label=(\roman*)]
        \item Let $M_n\to\infty$ as $n\to\infty$ with $M_n=o(n)$. Then
        \begin{equation}
            \sqrt{M_n}\Bigl(f\bigl(\vec{Y}^*/M_n\bigr) - f(\vec{Y}/n)\Bigl) \quad \dto \quad f'_{\vec{p}}(\vec{Z})
            \label{apdx:eq:xcut_bootstrap_convergence}
    	\end{equation}
    	where $\vec{Z}^{\XC} = (Z_S^{\XC})_{S\in\mathcal{S}}$ is the limiting distribution from \cref{thm:conv_mrcut_S}. \label{apdx:prop:duembgen_bootstrap:Mn}
    	\item Let $M_n \coloneqq  n$ and assume that $f$ is fully Hadamard differentiable at $\vec{p}$ (in almost any direction). Then \eqref{apdx:eq:xcut_bootstrap_convergence} still holds. \label{apdx:prop:duembgen_bootstrap:n}
	\end{enumerate}
    \label{apdx:prop:duembgen_bootstrap}
\end{proposition}
\begin{proof}
    For \ref{apdx:prop:duembgen_bootstrap:Mn} and \ref{apdx:prop:duembgen_bootstrap:n} see \citet[Propositions 2 and 1, respectively]{Duembgen1993}. 
    Note that all their basic assumptions are satisfied in our case:
    \begin{itemize}
        \item Their assumption A1 is satisfied by the central limit theorem: $\sqrt{n}(\vec{Y}/n - \vec{p})$ converges in distribution towards the tight Borel distribution $\mathcal{N}_m(\vec{0},\mat{\Sigma})$ on $\R^m$.
        \item Their assumption A2 is equivalent to our assumption of weak convergence in probability here. 
        \item Their assumption B1 is trivially satisfied since $\vec{p}$ does not depend on $n$.
        \item Their assumption B2 is a slightly weaker version of our assumption on Hadamard directional differentiability at $\vec{p}$.
    \end{itemize}
    Notice that, as the proof of \citet[Propositions 1 and 2]{Duembgen1993} reveals, both claims only hold on a sequence of sets $(A_n)_{n\in\N}\subset\R^d$ s.t.\ $\P(\vec{Y}\notin A_n)\to 0$ as $n\to\infty$.
\end{proof}

\begin{proposition}
    Let \firstassumptions hold and let $M_n\to\infty$. Obtain $\vec{Y}^*$ by sampling $M_n$ times out of $\vec{Y}$. Then, for almost every sample $\vec{Y}$, the distribution of $\sqrt{M_n}(\vec{Y}^*/M_n - \vec{Y}/n)$, conditioned on $\vec{Y}$, converges weakly in probability towards $\mathcal{N}_m(\vec{0},\mat{\Sigma})$.
    \label{apdx:prop:multinomial_bootstrap_weak_convergence}
\end{proposition}
\begin{proof}
    See \citet[Proposition 1.4.2~(iii)]{PolitisRomanoWolf1999}. 
    Note that the covariance convergence condition they impose is trivially satisfied in our case since we sample from a constant distribution $\nu$ (or rather $n$ times from $\mathrm{Bin}(1,\vec{p})$) which does not depend on $n$.
    
    Also notice that their Proposition 1.4.2~(iii) even yields weak convergence \emph{almost surely} (i.e.\ \cref{apdx:def:weak_convergence_in_probability} with almost sure convergence replacing convergence in probability in \eqref{apdx:eq:weak_convergence_in_probability}) which obviously implies weak convergence in probability. For our purposes, the latter will be sufficient. A more accessible version of this theorem can be found in \citet[Theorem~2.2~(a)]{BickelFreedman1981} who also show almost sure convergence.
\end{proof}


\begin{proof}[Proof of \cref{thm:bootstrap_limit_xcut}]
    Under \firstassumptions, the weak convergence assumption of \cref{apdx:prop:duembgen_bootstrap} is satisfied by \cref{apdx:prop:multinomial_bootstrap_weak_convergence}. Note that the XCut functional $g_S^{\XC}$ as defined in \eqref{eq:xcut_functional} is Hadamard directionally differentiable at $\vec{p}$ in any direction by \assumptionref{assumptions:diff}, and the minimum function $\minn:\R^{\abs{S}}\to\R$ is Hadamard directionally differentiable anywhere in any direction, see \mysubref{lem:min}{hadamard}. Consequently, the functional
    \[
    f:\R^m\to\R,\quad \vec{x}=(x_1,\ldots,x_m)\mapsto \min_{S\in\mathcal{S}} \sum_{{i\in S, j\in\comp{S},\;i\sim j}} \frac{x_i x_j}{\ball_S^{\XC}(x_1,\ldots,x_m)}
    \]
    is Hadamard directionally differentiable at $\vec{p}$ in $\vec{X}$-almost any direction. Thus, by \mysubref{apdx:prop:duembgen_bootstrap}{Mn}, 
    \[
    \Bigl(\sqrt{M_n}\bigl({\XC}_S(\mathcal{G}_{M_n}^*) - \XC_S(\mathcal{G}_n)\bigr)\Bigr)_{S\in\mathcal{S}} \quad \dto \quad \vec{Z}^{\XC}
    \]
    and therefore \mysubref{thm:bootstrap_limit_xcut}{Mn}. Analogously, claim~\ref{thm:bootstrap_limit_xcut:n} follows from \mysubref{apdx:prop:duembgen_bootstrap}{n}, because $f$ is fully Hadamard directionally differentiable at $\vec{p}$ if and only if the balancing terms $\ball_S^{\XC}$ is, too, \emph{and} the partition that attains the minimum XCut value is unique (see \mysubref{lem:min}{hadamard}).
\end{proof}

\subsection{Behaviour of the asymptotic variance}
\label{apdx:sub:variance}

Consider the variance $\var(Z_{S,t}^{\XC})$ of the limiting distribution for a fixed $S\in\mathcal{S}$. We are interested in its behaviour as a function of $t$, i.e.\ of whether there is a discernible attribute characterizing its course, such as monotonicity. For this, again consider the regular $3\times 3$ grid and define four distributions and the partition $S$ as in \cref{img:var_t_behaviour_distributions}. These distributions were picked to exemplify cases where the variance of the limiting distribution $Z_{S,t}^{\XC}$ corresponding to a \enquote{reasonable} partition $S\in\mathcal{S}$ behaves in completely different ways when considered as a function of $t$. In fact, $S$ attains the minimal RCut value for all the distributions is displayed in (b), (c) and (d) if $t=1$.

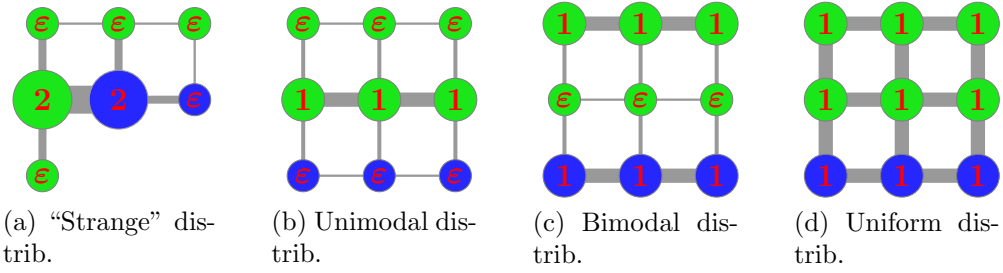
\begin{figure}[hbpt]
	\centering
	\subfloat[\enquote{Strange} distrib.]{
	\centering
	\begin{tikzpicture}[baseline]
		\draw[draw=none] (0,0) circle (8pt);
		\draw[draw=none] (2,2) circle (8pt);
		\foreach \x in {0,1,2} {
		    \foreach \y in {0,1,2} {
		        \node (\x;\y) at (\x,\y) {};
	        }
        }
		\draw[black!40, line width=0.3mm] (0;2) -- (1;2) -- (2;2) -- (2;1);
		\draw[black!40, line width=1.05mm] (0;2) -- (0;1);
		\draw[black!40, line width=1.05mm] (1;2) -- (1;1);
		\draw[black!40, line width=1.05mm] (1;1) -- (2;1);
		\draw[black!40, line width=1.05mm] (0;1) -- (0;0);
		\draw[black!40, line width=3.675mm] (0;1) -- (1;1);
		\draw[draw=black!50, fill=nicegreen] (0;2) circle (6pt) node[red] {\large $\bm{\eps}$};
		\draw[draw=black!50, fill=nicegreen] (1;2) circle (6pt) node[red] {\large $\bm{\eps}$};
		\draw[draw=black!50, fill=nicegreen] (2;2) circle (6pt) node[red] {\large $\bm{\eps}$};
		\draw[draw=black!50, fill=nicegreen] (0;1) circle (11pt) node[red] {\large $\bm{2}$};
	    \draw[draw=black!40, fill=niceblue] (1;1) circle (11pt) node[red] {\large $\bm{2}$};
		\draw[draw=black!50, fill=niceblue] (2;1) circle (6pt) node[red] {\large $\bm{\eps}$};
		\draw[draw=black!50, fill=nicegreen] (0;0) circle (6pt) node[red] {\large $\bm{\eps}$};
	\end{tikzpicture}}%
	\hspace{.6cm}
	\subfloat[Unimodal distrib.]{
	\centering
	\begin{tikzpicture}[baseline]
    	\draw[draw=none] (0,0) circle (8pt);
		\draw[draw=none] (2,2) circle (8pt);
		\foreach \x in {0,1,2} {
		    \foreach \y in {0,1,2} {
		        \node (\x;\y) at (\x,\y) {};
	        }
        }
		\draw[black!40, line width=0.3mm] (0;2) -- (1;2) -- (2;2);
		\draw[black!40, line width=1.875mm] (0;1) -- (1;1) -- (2;1);
		\draw[black!40, line width=0.3mm] (0;0) -- (1;0) -- (2;0);
		\draw[black!40, line width=0.525mm] (0;2) -- (0;1) -- (0;0);
		\draw[black!40, line width=0.525mm] (1;2) -- (1;1) -- (1,0);
		\draw[black!40, line width=0.525mm] (2;2) -- (2,1) -- (2;0);
		\foreach \x in {0,1,2} {
            \draw[draw=black!50, fill=nicegreen] (\x;2) circle (6pt) node[red] {\large $\bm{\eps}$};
            \draw[draw=black!50, fill=nicegreen] (\x;1) circle (8pt) node[red] {\large $\bm{1}$};
            \draw[draw=black!50, fill=niceblue] (\x;0) circle (6pt) node[red] {\large $\bm{\eps}$};
        }
	\end{tikzpicture}}%
	\hspace{.6cm}
	\subfloat[Bimodal distrib.]{
	\centering
	\begin{tikzpicture}[baseline]
		\draw[draw=none] (0,0) circle (8pt);
		\draw[draw=none] (2,2) circle (8pt);
		\foreach \x in {0,1,2} {
		    \foreach \y in {0,1,2} {
		        \node (\x;\y) at (\x,\y) {};
	        }
        }
		\draw[black!40, line width=1.875mm] (0;2) -- (1;2) -- (2;2);
		\draw[black!40, line width=0.3mm] (0;1) -- (1;1) -- (2;1);
		\draw[black!40, line width=1.875mm] (0;0) -- (1;0) -- (2;0);
		\draw[black!40, line width=0.525mm] (0;2) -- (0;1) -- (0;0);
		\draw[black!40, line width=0.525mm] (1;2) -- (1;1) -- (1,0);
		\draw[black!40, line width=0.525mm] (2;2) -- (2,1) -- (2;0);
		\foreach \x in {0,1,2} {
            \draw[draw=black!50, fill=nicegreen] (\x;2) circle (8pt) node[red] {\large $\bm{1}$};
            \draw[draw=black!50, fill=nicegreen] (\x;1) circle (6pt) node[red] {\large $\bm{\eps}$};
            \draw[draw=black!50, fill=niceblue] (\x;0) circle (8pt) node[red] {\large $\bm{1}$};
        }
	\end{tikzpicture}}%
	\hspace{.6cm}
	\subfloat[Uniform distrib.]{
	\centering
	\begin{tikzpicture}[baseline]
	    \draw[draw=none] (0,0) circle (8pt);
		\draw[draw=none] (2,2) circle (8pt);
		\foreach \x in {0,1,2} {
		    \foreach \y in {0,1,2} {
		        \node (\x;\y) at (\x,\y) {};
	        }
        }
        \draw[black!40, line width=1.875mm, xstep=1, ystep=1] (0,0) grid (2,2);
		\foreach \x in {0,1,2} {
            \draw[draw=black!50, fill=nicegreen] (\x;2) circle (8pt) node[red] {\large $\bm{1}$};
            \draw[draw=black!50, fill=nicegreen] (\x;1) circle (8pt) node[red] {\large $\bm{1}$};
            \draw[draw=black!50, fill=niceblue] (\x;0) circle (8pt) node[red] {\large $\bm{1}$};
        }
	\end{tikzpicture}}%
	\caption{Neighbourhood graph of four distributions on a $3\times 3$ grid (on $D=[1,3]^2$ for convenience) with $m=9$ nodes ($m=7$ in (a)), $t=1$ and $\eps=0.4$. The blue nodes represent the fixed partition $S$ which minimizes the RCut value. The number inside of a node represents its weight in the underlying probability distribution without normalization to better showcase their relation. The vertex and edge widths are proportional to their respective weights.}
	\label{img:var_t_behaviour_distributions}
\end{figure}

\cref{img:var_t_behaviour} shows that even in the simple case of a very small grid the limiting variance follows no clear pattern as the neighbourhood distance $t$ increases. It should be mentioned, however, that finding an example of a distribution (and partition thereon) where the variance increased from $t=1$ to $t=\sqrt{2}$ was more difficult and resulted in a more artificial example. This could be an indication that a choice of $t=1$ is sufficient to obtain a graph that can be cut to obtain reasonably \enquote{good} partitions. It would certainly be an interesting avenue for further research to assess whether and how the quality of a cut is related to its asymptotic variance. On another note, in \cref{img:var_t_behaviour} we only addressed the case of RCut here for simplicity, though similar examples of distributions can be easily found for the other cuts as well.

\begin{figure}[hptb]
	\centering
	\includegraphics[width=0.6\linewidth]{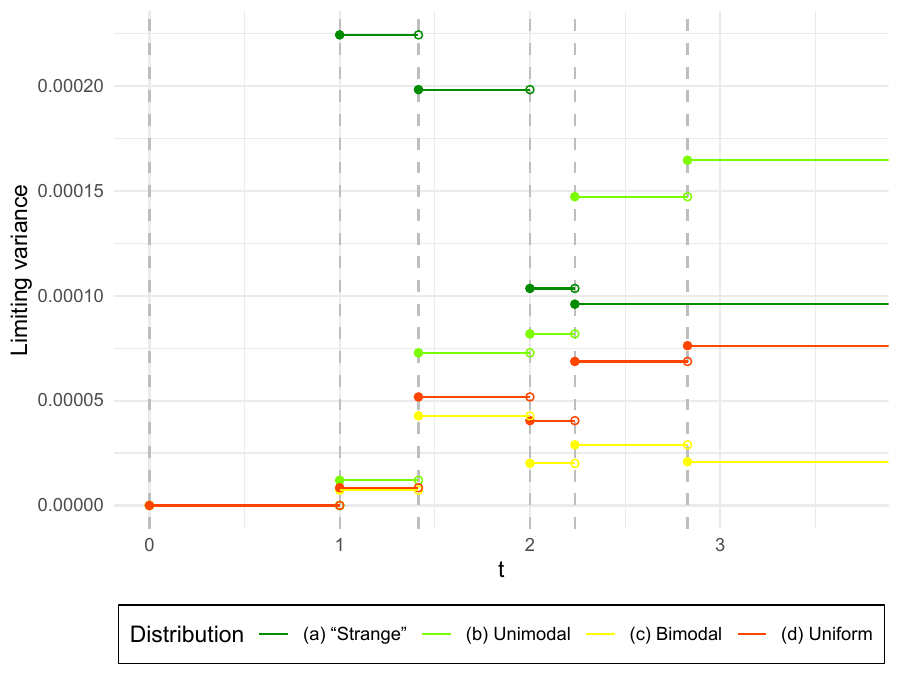}
	\caption{Variance of the limiting distribution of $\smash{\sqrt{n}\big(\RC_S(\mathcal{G}_n) - \RC_S(\mathcal{G})\big)}$ for the four distributions and the partition $S$ displayed in \cref{img:var_t_behaviour_distributions}, plotted as a function over the neighbourhood distance $t$, whereupon $\mathcal{G}_n$ and $\mathcal{G}$ depend by means of the graph structure, see \cref{def:discretization}.}
	\label{img:var_t_behaviour}
\end{figure}

\subsection{Revisiting \texorpdfstring{\cref{exmp:xvst_practice}}{Example 1}}
\label{apdx:sub:example}

In this section we revisit \cref{exmp:xvst_practice} and fill in all the details properly that have been previously omitted. First, however, we require a small auxiliary result that is also called upon in the main section to argue that asymptotically, the components of a multinomial distribution are almost surely distinct, namely the following \cref{lem:multinomial_components_distinct}.

\begin{lemma}
    Let $\vec{Y}\sim\mult(n,\vec{p})$ for $n\in\N$ and some probability vector $\vec{p}\in\R^m$, $m\in\N_{\geq 2}$. Then, for any $i,j\in V$, $i\neq j$, with $p_i,p_j\in [0,1]$ (but not both zero), it holds that $\P(Y_i=Y_j)\to 0$ as $n\to\infty$.
    \label{lem:multinomial_components_distinct}
\end{lemma}
%
\begin{proof}
    Let $i,j\in V$, $i\neq j$, be arbitrary, but fixed. If $p_i\neq p_j$, the claim holds simply by the law of large numbers, so w.l.o.g.\ let $p_i = p_j$. We apply the continuous mapping theorem on the function $\langle\vec{e}_i - \vec{e}_j,\cdot\rangle$ and obtain
    \[
     \frac{Y_i - Y_j}{\sqrt{n}} = \left\langle\vec{e}_i - \vec{e}_j, \sqrt{n}\bigl(\tfrac{\vec{Y}}{n} - \vec{p}\bigr)\right\rangle \dto \mathcal{N}\bigl(0, (\vec{e}_i - \vec{e}_j)^{\top} \mat{\Sigma} (\vec{e}_i - \vec{e}_j)\bigr).
    \]
   Note that $ (\vec{e}_i - \vec{e}_j)^\top \mat{\Sigma} (\vec{e}_i - \vec{e}_j) = p_i (1-p_i) + p_j (1-p_j) + 2p_i p_j > 0$ 
    since $0 < p_i,p_j < 1$ (if $p_i$ or $p_j$ were equal to $1$, the claim would be trivial) and not both are equal to zero by assumption. This shows that the limiting distribution above is non-degenerate. Consequently, for any $\eps>0$
    \begin{align*}
    \P(Y_i = Y_j) \leq \P\left(-\eps < \frac{Y_i - Y_j}{\sqrt{n}} \leq \eps\right) &= F_n(\eps) - F_n(-\eps) \\
    &\nto F(\eps) - F(-\eps) =  \P(-\eps < Z_{ij} \leq \eps),
    \end{align*}
    where $F_n$ and $F$ denote the cdf of $(Y_i - Y_j)/\sqrt{n}$ and its limiting distribution $Z_{ij}:\sim \mathcal{N}\bigl(0, p_i + p_j - (p_i - p_j)^2\bigr)$, respectively. Since the right hand side of the above equation becomes arbitrarily small as $\eps\to 0$, we have $\P(Y_i = Y_j)\to 0$ as $n\to\infty$ for any $i,j\in V$, $i\neq j$.
\end{proof}

We are ready to provide the full details of \cref{exmp:xvst_practice}. Consider the graph $\mathcal{G}=(V,E,\mat{W})$ depicted in \cref{img:xvst_theory_example}, i.e.\ $V\coloneqq \{1,2,3,4\}$ and $E\coloneqq \bigl\{\{1,2\},\{1,3\},\{2,4\},\{3,4\}\bigr\}$, where the probability vector $\vec{p} = (1/4,1/4,1/4,1/4)$ (see also \cref{img:xvst_theory_example} for an illustration).

Before computing any limits, it is important to consider the case of $\abs{\Vloc}=1$, where \xistref does not return any clustering. We seek to calculate the asymptotic probability of this event. Note that
\begin{align}
\P(\Vloc = \{Y_1\}) &= \P\bigl(Y_1 > Y_2, Y_1 > Y_3, \max(Y_2,Y_3) > Y_4\bigr) \nonumber\\
&= \P\bigl(\{Y_1 > Y_2, Y_1 > Y_3\}\setminus \{Y_1 > Y_2, Y_1 > Y_3, Y_4 > Y_2, Y_4 > Y_3\}\bigr) \nonumber\\
&= \P(Y_1 > Y_2, Y_1 > Y_3) - \P(Y_1 > Y_2, Y_1 > Y_3, Y_4 > Y_2, Y_4 > Y_3),\label{apdx:eq:exmp:just_one_locmax}
\end{align}
where we used that $\P(Y_i = Y_j)\nto 0$ by \cref{lem:multinomial_components_distinct}. (In the following, this fact will be used implicitly.) We calculate the first probability in \eqref{apdx:eq:exmp:just_one_locmax} by defining
\[
\mat{A} \coloneqq  \begin{bmatrix}
-1 & 1 & 0 & 0 \\
-1 & 0 & 1 & 0 \\
\end{bmatrix}\quad\text{s.t.}\quad
\frac{1}{\sqrt{n}}\begin{bmatrix}Y_2 - Y_1 \\ Y_3 - Y_1\end{bmatrix} = \mat{A} \sqrt{n}\left(\frac{\vec{Y}}{n} - \vec{p}\right) \dto\mathcal{N}(\vec{0},\mat{A}^{\top} \mat{\Sigma}\mat{A})
\]
by the central limit theorem for $\vec{Y}/n$, where $\mat{\Sigma}$ is the covariance matrix of $\mult(1,\vec{p})$. Consequently, denoting the distribution function of $\mathcal{N}(\vec{0},\mat{A}^{\top} \mat{\Sigma}\mat{A})$ by $F_{\mat{A}\vec{Z}}$ and setting $\mat{B} \coloneqq  \bigl(\mat{A}^{\top} \mat{\Sigma}\mat{A}\bigr)^{-1}$,
\begin{align*}
\P(Y_1 > Y_2, Y_1 > Y_3) = & \P\biggl(\frac{1}{\sqrt{n}}\begin{bmatrix}Y_2 - Y_1 \\ Y_3 - Y_1\end{bmatrix} < \vec{0}\biggr) \\
&\nto F_{\mat{A}\vec{Z}}(\vec{0}) \\
&= \int_{(-\infty,0)^2} \frac{1}{2\pi\sqrt{\det\mat{A}^{\top} \mat{\Sigma}\mat{A}}} \exp\Bigl(-\tfrac{1}{2}\bigl(\vec{x}^{\top} \bigl(\mat{A}^{\top} \mat{\Sigma}\mat{A}\bigr)^{-1}\vec{x}\bigr)\Bigr) d\vec{x} \\
&= \int_{-\infty}^0 \int_{-\infty}^0 \frac{1}{2\pi\sqrt{\tfrac{16}{3}}} \exp\Bigl(-\tfrac{1}{2}\bigl(\tfrac{8}{3}x_1^2 - \tfrac{8}{3}x_1 x_2 + \tfrac{8}{3} x_2^2\bigr)\Bigr) d x_1 d x_2 \\
&= \frac{1}{3}.
\end{align*}
To compute the second term in \eqref{apdx:eq:exmp:just_one_locmax}, rewrite
\begin{align}
    &\P(Y_1 > Y_2, Y_1 > Y_3, Y_4 > Y_2, Y_4 > Y_3) \nonumber\\
    &\qquad = \P(Y_1 > Y_2, Y_1 > Y_3, Y_4 > Y_2) - \P(Y_1 > Y_2, Y_1 > Y_3, Y_4 > Y_2, Y_3 > Y_4) \nonumber\\
    &\qquad = \P(Y_1 > Y_2, Y_1 > Y_3, Y_4 > Y_2) - \P(Y_1 > Y_3, Y_3 > Y_4, Y_4 > Y_2) \label{apdx:eq:exmp:just_one_locmax_2}
\end{align}
and compute both probabilities separately. This can be done by choosing
\[
\mat{A} \coloneqq  \begin{bmatrix}
    -1 & 1 & 0 & 0 \\
    -1 & 0 & 1 & 0 \\
    0 & 1 & 0 & -1
\end{bmatrix}\,\text{ for the first term, and }\, \mat{A} \coloneqq  \begin{bmatrix}
    -1 & 0 & 1 & 0 \\
    0 & 1 & 0 & -1 \\
    0 & 0 & -1 & 1
\end{bmatrix}
\]
for the second, and then proceeding as shown above. For brevity, we will omit the corresponding calculations as they are analogous to the above. Putting everything together, we arrive at
\begin{multline*}
\P(\Vloc = \{Y_1\}) 
= \P(Y_1 > Y_2, Y_1 > Y_3) \\-\bigl(P(Y_1 > Y_2, Y_1 > Y_3, Y_4 > Y_2) - \P(Y_1 > Y_3, Y_3 > Y_4, Y_4 > Y_2)\bigr) 
\nto
\frac{1}{6}.
\end{multline*}
Analogously, any of the other vertices could be the sole local maximum, so that finally
\[
\P(\abs{\Vloc} = 1) = \sum_{i=1}^4 \P(\Vloc = \{Y_i\}) = 4\,\P(\Vloc = \{Y_1\}) \nto \frac{2}{3}.
\]
Consequently, it makes no sense to use the standard version of \xistref for such a simple example; instead redefine $\Vloc\coloneqq V$ (i.e.\ simply iterate over all vertices instead of only local maxima).

Consider the fact that the $12$-MinCut of the population graph is attained by $S_1 \coloneqq  \{1\}$, $S_2 \coloneqq  \{2\}$ and $S_{13} \coloneqq  \{1,3\}$, although only the latter yields the output $\xist(\mathcal{G})=2$ of the \xistref algorithm for NCut. In this simple case, we can compute the (asymptotic) probabilities that the minimum is attained by either of the above partitions manually: As $n\to \infty$,
\begin{align*}
&\P\bigl(\MC(S_{13}) < \MC(S_1), \MC(S_{13}) < \MC(S_2)\bigr)\\
&\qquad=\P(Y_1 Y_2 + Y_3 Y_4 < Y_1 Y_2 + Y_1 Y_3, Y_1 Y_2 + Y_3 Y_4 < Y_1 Y_2 + Y_2 Y_4) \\
&\qquad= \P(0 < Y_4 < Y_1,\,\, 0 < Y_3 < Y_2) \to \frac{1}{4}, \\
&\P\bigl(\MC(S_1) < \MC(S_2), \MC(S_1) < \MC(S_{13})\bigr) \\
&\qquad=\P(\MC(S_2) < \MC(S_1), \MC(S_2) < \MC(S_{13})) \to \frac{3}{8},
\end{align*}
by symmetry and the fact that all three have to sum up to $1$ since $\P(Y_i = Y_j)\to 0$ for any $i,j\in V$, $i\neq j$, by \cref{lem:multinomial_components_distinct}.

Returning to the $st$-MinCut probabilities, one obtains the same probabilities (for the appropriately re-indexed partitions) for the $13$-, $24$- and $34$-MinCuts. The diagonal partitions, i.e.\ the $14$- and $23$-MinCuts, can be computed similarly, e.g.\ for the for the $23$-MinCut, $S_2 = \{2\}$ and $S_3 = \{3\}$ occur with probability $1/2$ each, and $S_{12} \coloneqq  \{1,2\}$ and $S_{13} = \{1,3\}$ are almost surely not attained for large~$n$.

In order to determine the limiting distribution notice that under the condition that $Y_i\neq Y_j$ for $i\neq j$ (which is satisfied asymptotically by \cref{lem:multinomial_components_distinct}), either the partition $S_{12} = \{1,2\}$ or $S_{13} = \{1,3\}$ attains the $ij$-MinCut for at least one pair $i,j\in\{1,2,3,4\}$, $i\neq j$, that is considered by \xistref. To see why, assume w.l.o.g.\ that \xistref starts with the $21$-MinCut. We distinguish four cases:
\begin{itemize}
    \item $Y_1 > Y_3, Y_2 > Y_4$: \xistref computes the $21$-MinCut partition $\{1,3\}$, updates $\tau=(1,2,2,1)$, then the $32$-MinCut partition $\{3\}$, updates $\tau=(1,2,3,1)$, and finally the $41$-MinCut partition $\{4\}$.
    \item $Y_1 > Y_3, Y_2 < Y_4$: \xistref computes the $21$-MinCut partition $\{2\}$, updates $\tau=(1,2,1,1)$, then the $31$-MinCut partition $\{3\}$, updates $\tau=(1,2,3,1)$, and finally the $41$-MinCut partition $\{1,2\}$.
    \item $Y_1 < Y_3, Y_2 > Y_4$: \xistref computes the $21$-MinCut partition $\{1\}$, updates $\tau=(1,2,2,2)$, then the $32$-MinCut partition $\{1,2\}$, updates $\tau=(1,2,3,3)$, and finally the $43$-MinCut partition $\{4\}$.
    \item $Y_1 < Y_3, Y_2 < Y_4$: If $Y_1 Y_4 > Y_2 Y_3$, \xistref computes the $21$-MinCut partition $\{2\}$, updates $\tau=(1,2,1,1)$, then the $31$-MinCut partition $\{1\}$, then updates $\tau=(1,2,3,3)$. Otherwise, i.e.\ if $Y_1 Y_4 < Y_2 Y_3$, \xistref computes the $21$-MinCut partition $\{1\}$, updates $\tau=(1,2,2,2)$, then the $32$-MinCut partition $\{2\}$, then updates $\tau=(1,2,3,3)$. Finally, in both cases (so also if $Y_1 Y_4 = Y_2 Y_3$), the $43$-MinCut partition $\{1,4\}$ is computed.
\end{itemize}

Hence, \xistref considers either $S_{12}\coloneqq \{1,2\}$ or $S_{13}\coloneqq \{1,3\}$ in either case, and, asymptotically, both are returned with probability $1/2$ each as all the above cases have the same probability of occurring. Since both $S_{12}$ and $S_{13}$ yield the same minimal NCut value, \xistref must return the empirical equivalent of the optimal population NCut value $\NC_{S_{12}}(\mathcal{G}) = \NC_{S_{13}}(\mathcal{G}) = 2$. As a side note, the \xvstnameref from \cite{SuchanLiMunk2023arXiv} that will be introduced in the following \cref{apdx:sub:algorithm_proofs} returns the same output as \xistref, so both algorithms attain the same limiting distribution in this example.

Continuing, it thus suffices to consider the limiting distribution of $\NC_{S_{12}}(\mathcal{G}_n)$ and ${\NC}_{S_{13}}(\mathcal{G}_n)$. Together, they converge in distribution to $\vec{Z}^{\NC}:\sim \mathcal{N}_2(\vec{0},\mat{0}) = \delta_{(0,0)}$ (and thus also in probability), where the degenerate covariance matrix $\mat{0}\in\R^{2\times 2}$ consists of zeros only. By the Hadamard delta method (i.e.\ \cref{prop:hadamard_delta_method}), the limiting distribution of $\sqrt{n}(\NC(\mathcal{G}_n) - \NC(\mathcal{G}))$ is  $\min \vec{Z}^{\NC} = \delta_0$.

Note that in this scenario we may not apply \cref{lem:xvst_theory} since \ref{assumptions:additional} is violated: Both $S_{12}$ and $S_{13}$ satisfy $q_{i,S_{12}} = q_{i,S_{13}}$ for all $i\in V$, and both $S_{12}$ and $S_{13}$ attain the $14$-MinCut. Interestingly, despite this, neither $S_{12}$ nor $S_{13}$ can ever attain the empirical $14$-MinCut since (by \cref{lem:multinomial_components_distinct})
\begin{align*}
    \MC_{S_{12}}(\mathcal{G}_n) = \MC^{14}(\mathcal{G}_n) &\iff Y_1 Y_3 + Y_2 Y_4 < Y_1 Y_3 + Y_1 Y_2 \\
    &\phantom{\iff}\qquad\text{ and }Y_1 Y_3 + Y_2 Y_4 < Y_3 Y_4 + Y_2 Y_4 \\
    &\iff Y_4 < Y_1, Y_1 < Y_4
\end{align*}
which can never be satisfied. By symmetry, the same result holds for $S_{13}$.

Since the limiting distribution of $\sqrt{n}\bigl(\NC(\mathcal{G}_n) - \NC(\mathcal{G})\bigr)$ is degenerate, it is natural to increase the scaling, i.e.\ by considering $n\bigl(\NC(\mathcal{G}_n) - \NC(\mathcal{G})\bigr)$. For this, however, we require the second-order delta method (see e.g.\ \cite[Theorem~3.3B]{Serfling1980}). Let $S$ now denote either $S_{12}$ or $S_{13}$. The Taylor expansion of the NCut functional $g_S^{\NC}(\cdot)$ (recall its definition in \eqref{eq:xcut_functional})  around $\vec{Y}/n$ is given by
\begin{equation*}
    g_S^{\NC}\bigl(\tfrac{\vec{Y}}{n}\bigr) = g_S^{\NC}(\vec{p}) + (\nabla g_S^{\NC}(\vec{p}))^{\top} \bigl(\tfrac{\vec{Y}}{n} - \vec{p}\bigr) 
    + \frac{1}{2} \bigl(\tfrac{\vec{Y}}{n} - \vec{p}\bigr)^{\top} \mat{H}g_S^{\NC} \bigl(\tfrac{\vec{Y}}{n} - \vec{p}\bigr) + {\scriptstyle \mathcal{O}}_P\bigl(\norm{\tfrac{\vec{Y}}{n} - \vec{p}}^2\bigr),
\end{equation*}
with $\mat{H}g_S^{\NC}\in\R^{m\times m}$ the Hessian of the function $g_S^{\NC}$. As $\nabla g_S^{\NC}(\vec{p}) = -4\cdot\vec{1}$, we have $\bigl(g_S^{\NC}(\vec{p})\bigr)^{\top} (\vec{Y}/n - \vec{p}) = 0$ by $\vec{Y}^{\top}\vec{1} = n$ and $\vec{p}^{\top}\vec{1} = 1$. Then, the convergence of $g_S^{\NC}(\vec{Y}/n) - g_S^{\NC}(\vec{p})$ depends on the square term in the above Taylor expansion. Scaling by $n$ (instead of the usual $\sqrt{n}$) now yields a convergence to a nondegenerate distribution: Since $\sqrt{n}(\vec{Y}/n - \vec{p})$ converges in distribution to $\widetilde{\vec{Z}} :\sim\mathcal{N}(\vec{0}, \mat{\Sigma})$ and $\mat{H}g_S^{\NC}$ is symmetric, by the aforementioned second-order delta method we obtain
\begin{equation}
    n\bigl(g_S^{\NC}\bigl(\tfrac{\vec{Y}}{n}\bigr) - g_S^{\NC}(\vec{p})\bigr) = \frac{1}{2} \bigl(\tfrac{\vec{Y}}{n} - \vec{p}\bigr)^{\top} \mat{H}g_S^{\NC} \bigl(\tfrac{\vec{Y}}{n} - \vec{p}\bigr) + {\scriptstyle \mathcal{O}}_{\P}\bigl(\norm{\tfrac{\vec{Y}}{n} - \vec{p}}^2\bigr) 
\dto \frac{1}{2} \widetilde{\vec{Z}}^{\top} \mat{H}g_{S_{12}}^{\NC} \widetilde{\vec{Z}}. \label{apdx:eq:exmp:taylor}
\end{equation}
For any $S\in\mathcal{S}$, the Hessian $\mat{H}g_S^{\NC}(\vec{p})$ evaluates to
\begin{align*}
\bigl(\mat{H}g_S^{\NC}\bigr)_{k,\ell} &= \frac{\mathbbm{1}_{k\in S_{\not\ni\ell}, k\sim\ell} - \Bigl(\frac{\partial}{\partial p_{\ell}} d_{k,S}\Bigr) \NC_S(\mathcal{G}) - d_{k,S}\, q_{\ell,S}}{\vol S \vol\comp{S}} \\
&\quad- \frac{(q_{k,S} - d_{k,S}\NC_S(\mathcal{G})) d_{\ell,S}}{(\vol S \vol\comp{S})^2} \\
\text{for}\quad d_{r,S} &\coloneqq \frac{\partial}{\partial p_r} \vol S \vol\comp{S} = q_{r,S} (\vol S_{\ni r} - \vol S_{\not\ni r}) + 2\vol S_{\not\ni r} a_r \quad\text{for } r\in V\\
\text{s.t.}\quad \frac{\partial}{\partial p_{\ell}} d_{k,S} &= \begin{cases}
    2(q_{k,S} (a_{\ell} - q_{\ell,S}) + q_{\ell,S} a_k + \vol S_{\not\ni k} \mathbbm{1}_{\ell\sim k}), & S_{\ni k} = S_{\ni\ell} \\
    (2a_k - q_{k,S})(2a_{\ell} - q_{\ell,S}) + q_{k,S} q_{\ell,S} + \vol V \mathbbm{1}_{\ell\sim k}, & S_{\ni k} = S_{\not\ni\ell},
\end{cases}
\end{align*}
where we used the previously introduced abbreviation $a_r\coloneqq \sum_{j\sim r} p_j$. Plugging in $\vec{p}=\tfrac{1}{4}\cdot\mathbf{1}$ yields $\vol S = \tfrac{1}{4}$, $d_{r,S} = \tfrac{1}{4}$, $\tfrac{\partial}{\partial p_{\ell}} d_{k,S} = \tfrac{3}{16} + \tfrac{1}{4}\mathbbm{1}_{k\sim\ell}$ if $k\in S_{\ni\ell}$ and $\tfrac{10}{16} + \tfrac{1}{2}\mathbbm{1}_{k\sim\ell}$ otherwise, everything for both $S = S_{12}$ and $S = S_{13}$. With this, we obtain the two Hessians
\[
\mat{H} g_{S_{12}}^{\NC} = \begin{bmatrix}
    9 & 1 & 35 & -5 \\
    1 & 9 & -5 & 35 \\
    35 & -5 & 9 & 1 \\
    -5 & 35 & 1 & 9
\end{bmatrix}\quad\text{and}\quad \mat{H} g_{S_{13}}^{\NC} = \begin{bmatrix}
    9 & 35 & 1 & -5 \\
    35 & 9 & -5 & 1 \\
    1 & -5 & 9 & 35 \\
    -5 & 1 & 35 & 9
\end{bmatrix}
\]
of $S_{12}$ and $S_{13}$, respectively. We first derive the asymptotic distribution of a single partition $S\in\{S_{12},S_{13}\}$ by noticing that $\widetilde{\vec{Z}}\sim\mathcal{N}(\vec{0},\mat{\Sigma})$ can be written as $\widetilde{\vec{Z}} = \mat{L}\vec{Z}$ for $\vec{Z}\sim\mathcal{N}(\vec{0},\mat{I}_4)$ and the lower triangular matrix $\mat{L}$ resultant from the Cholesky decomposition of $\mat{\Sigma} = \mat{L}\mat{L}^{\top}$ (see \cite[Corollary 3]{TanabeSagae1992} for a general formula of $\mat{L}$ for any multinomial covariance matrix $\Sigma$). Hence, \eqref{apdx:eq:exmp:taylor} becomes
\begin{align}
    n\bigl(\NC_{S_{12}}(\mathcal{G}_n) - \NC_{S_{12}}(\mathcal{G})\bigr) &\dto \frac{1}{2} \widetilde{\vec{Z}}^{\top} \mat{H}g_{S_{12}}^{\NC} \widetilde{\vec{Z}} = \frac{1}{2} \vec{Z}^{\top} \mat{L}^{\top} \mat{H}g_{S_{12}}^{\NC}\mat{L} \vec{Z} \nonumber \\
    \begin{split}
    &= -\tfrac{1}{6} Z_1^2 - \tfrac{4}{3} Z_2^2 + Z_3^2 \label{apdx:eq:exmp:limit_distribution_S_12} \\
    &\quad - \sqrt{\tfrac{98}{9}}\, Z_1 Z_2 + \sqrt{\tfrac{200}{3}}\, Z_1 Z_3 - \sqrt{\tfrac{100}{3}}\, Z_2 Z_3.
    \end{split} \\
    \intertext{Analogously, $\NC_{S_{13}}(\mathcal{G}_n)$ follows an asymptotic distribution of}
    \begin{split}
    n\bigl(\NC_{S_{13}}(\mathcal{G}_n) - \NC_{S_{13}}(\mathcal{G})\bigr) &\dto -\tfrac{1}{6} Z_1^2 + \tfrac{35}{12} Z_2^2 - \tfrac{13}{4} Z_3^2 \label{apdx:eq:exmp:limit_distribution_S_13} \\
    &\quad + \sqrt{\tfrac{1369}{18}}\, Z_1 Z_2 + \sqrt{\tfrac{3}{2}}\, Z_1 Z_3 - \sqrt{\tfrac{3}{4}}\, Z_2 Z_3,
    \end{split}
\end{align}
so, asymptotically, both $\NC_{S_{12}}(\mathcal{G}_n)$ and $\NC_{S_{13}}(\mathcal{G}_n)$ follow a generalized chi-squared distribution with 3 degrees of freedom. Notice that the asymptotic distributions \eqref{apdx:eq:exmp:limit_distribution_S_12} and \eqref{apdx:eq:exmp:limit_distribution_S_13} of $\NC_{S_{12}}(\mathcal{G}_n)$ and $\NC_{S_{13}}(\mathcal{G}_n)$ are equal by symmetry.


We can extend these results to obtain an asymptotic distribution for \xistref by considering the terms in \eqref{apdx:eq:exmp:taylor} jointly over $S\in\{S_{12}, S_{13}\}$, such that
\begin{align*}
    \begin{bmatrix}
        n\bigl(g_{S_{12}}^{\NC}\bigl(\tfrac{\vec{Y}}{n}\bigr) - g_{S_{12}}^{\NC}(\vec{p})\bigr) \\
        n\bigl(g_{S_{13}}^{\NC}\bigl(\tfrac{\vec{Y}}{n}\bigr) - g_{S_{13}}^{\NC}(\vec{p})\bigr)
    \end{bmatrix} 
    &\dto \frac{1}{2} \begin{bmatrix}
        \vec{Z}^{\top} \mat{H}g_{S_{12}}^{\NC} \vec{Z} \\
        \vec{Z}^{\top} \mat{H}g_{S_{13}}^{\NC} \vec{Z}
    \end{bmatrix}.
\end{align*}
Finally, using the delta method for Hadamard directionally differentiable functions (\cref{prop:hadamard_delta_method}) on the $\min(\cdot)$ functional via \mysubref{lem:min}{hadamard} yields
\begin{align*}
n\bigl(\xist(\mathcal{G}_n) - \xist(\mathcal{G})\bigr) &= \min_{S \in\{ S_{12},\,S_{13}\}} n\bigl(g_S^{\NC}\bigl(\tfrac{\vec{Y}}{n}\bigr) - g_S^{\NC}(\vec{p})\bigr) \\
&\dto \min_{S \in \{S_{12},\,S_{13}\}} \frac{1}{2}\vec{Z}^{\top} \mat{H}g_S^{\NC} \vec{Z},
\end{align*}
where explicit expressions for $\tfrac{1}{2}\vec{Z}^{\top} \mat{H}g_{S_{12}}^{\NC} \vec{Z}$ and $\tfrac{1}{2}\vec{Z}^{\top} \mat{H}g_{S_{13}}^{\NC} \vec{Z}$ are derived in \eqref{apdx:eq:exmp:limit_distribution_S_12} and \eqref{apdx:eq:exmp:limit_distribution_S_13}.

\subsection{Consistency of the \texorpdfstring{\xistref}{Xist} algorithm}
\label{apdx:sub:algorithm_proofs}

In the following we derive a limiting distribution for the \xistref algorithm by showing \cref{thm:xist_limit}. In order to showcase the idea behind the proof of \cref{thm:xist_limit}, we first briefly introduce the \xvstnameref from \cite{SuchanLiMunk2023arXiv} and proof a CLT-type result for it in \cref{lem:xvst_theory}. After then deriving the limiting distribution of \xistref, the proof of \cref{thm:conv_mrcut_S_k} is presented.
%
%

\begin{algorithm}
\label{alg:xvst}
\SetAlgorithmName{Basic Xvst algorithm}{Xvst}
\SetAlgoLined
\DontPrintSemicolon
\SetKwInOut{Input}{input}
\SetKwInOut{Output}{output}
\Input{weighted graph $G=(V,E,\mat{W})$}
\Output{XC value $c_{\min}$ with associated partition $S_{\min}$}
\BlankLine

set $c_{\min}\leftarrow\infty$ and $S_{\min}\leftarrow\emptyset$\;

\For{$\{s,t\}\subset V$ with $s\neq t$\label{alg:xvst:while_loop}}{
    compute an $st$-MinCut partition $S_{st}$ on $G$ \label{alg:xvst:st-mincut}\;
    compute the graph cut value $\XC_{\comp{S}_{st}}(G)$ of partition $S_{st}$ \label{alg:xvst:xcut}\;
	\If{$\XC_{S_{st}}(G) < c_{\min}$}{
        $c_{\min} \leftarrow \XC_{S_{st}}(G)$\;
        $S_{\min} \leftarrow S_{st}$\;
    }
}\label{alg:xvst:while_end}
\caption{$\XC(G)$ via $st$-MinCuts}
\end{algorithm}

The \xvstnameref contains the essence of \xistref, namely the principle of choosing the best XCut out of a subset of (or, in this case, all) $st$-MinCuts. In contrast to \xistref, however, the \xvstnameref does not restrict itself to $\Vloc$, and it does not contain any Gomory--Hu-type construction to improve its computational complexity. By \citet[Theorem 3.2]{SuchanLiMunk2023arXiv}, the \xvstnameref runs in $\BO(\abs{V}^3\abs{E})$, so one order of magnitude slower than \xistref if $\abs{V}$ and $\abs{\Vloc}$ are of the same order, e.g.\ if $p_1=\cdots=p_m$ for $\mathcal{G}$. In practice, however, this case is not likely to happen, especially since $\P(Y_i = Y_j)\to 0$ for any $i\neq j$ as $n\to\infty$ by \cref{lem:multinomial_components_distinct}.

In order to gauge the computational impact that the restriction to $\Vloc$ has on \xistref compared to the \xvstnameref, the expected number of local maxima is of interest, particularly if an underlying discretized distribution exhibits $p_i=p_j$ for many connected vertex pairs $i,j\in V$. This quantity is difficult to compute exactly in general, however, due to $\mat{\Sigma}$ being non-invertible, meaning that the limiting distribution of $\sqrt{n}(\vec{Y}/n - \vec{p})$ has no density. While this could be alleviated by restriction onto a subspace, it does not help with the second issue of having to compute a large number of integrals since the information whether vertex $i$ is local maximum has a \enquote{cascading effect}, i.e.\ it introduces a dependency structure in the form of $Y_i$ on all connected vertices. Consequently, the only feasible way to gauge the advantages of \xistref is via Monte Carlo simulation. Consider a $\ell\times \ell$ grid with $p_1=\cdots=p_m=1/m$ for $m=\ell^2$ --- the law of large numbers allows us to ignore any non-uniform distributions --- with edges only between directly horizontally and vertically adjacent vertices. Simulations show that the expected number of local maxima is approximately $4/3$ for $\ell=2$, $8.6$ for $\ell=5$, $22.5$ for $\ell=10$ and $49.4$ for $\ell=15$, meaning that it seems to scale linearly with $m$, and, on average, \xistref only considers around one fifth of all vertices, leading to a runtime that is 25 times faster than the \xvstnameref, notably on a uniform distribution which is can be considered the \enquote{worst-case scenario} for the improved algorithm. Moreover, when the number of edges increases, the number of local maxima decreases, meaning \xistref performs even better compared to the \xvstnameref. Exemplified on the $\ell\times\ell$-grid with uniform distribution above, adding edges between directly diagonally adjacent vertices decreases the expected number of local maxima to around an eighth of all vertices. For a more thorough comparison of \xistref and the \xvstnameref, in particular regarding the impact of and explanation behind the $st$-MinCut selection procedure through $\tau$, see \cite{SuchanLiMunk2023arXiv}.

Before deriving the limiting distribution of \xvstnameref, we require a small auxiliary result that reformulates \assumptionref{assumptions:additional} to make it better usable in the following proofs.

\begin{lemma}
    \label{lem:q_equality_uniqueness}
    \assumptionref{assumptions:additional} holds if and only if the all components corresponding to elements in each $\mathcal{S}_{st}$, $s,t\in V$, $s\neq t$, of the MCut limiting distribution $\vec{Z}^{\MC}$ are almost surely non-equal, i.e.\
    \[
    q_{i,S} = q_{i,T}\;\text{ for all } i\in V \iff \P(Z_S^{\MC} = Z_T^{\MC}) = 0.
    \]
\end{lemma}
\begin{proof}
    By \mysubref{thm:conv_mrcut_S_k}{diff} the MCut limit $\vec{Z}^{\MC}$ is distributed according to the Gaussian $\mathcal{N}_{\abs{\mathcal{S}}}(\vec{0},\mat{\Sigma}^{\MC})$ with $\Sigma_{S,T}^{\MC} = \sum_{i=1}^n p_i q_{i,S} q_{i,T} - 4\MC_S(\mathcal{G})\MC_T(\mathcal{G})$. Fix now $s,t\in V$, $s\neq t$ as well as $S,T\in\mathcal{S}_{st}$ (although the following holds for any two partitions $S$ and $T$). Consider the vector $\vec{a} \coloneqq  \vec{e}_S - \vec{e}_T\in\R^{|S|}$, where $\vec{e}_S$ and $\vec{e}_T$ are unit vectors in $\R^{|S|}$ with one in the $S$-th and $T$-th component, respectively. Then, 
    $\vec{a}^{\top} \vec{X}^{\MC} \sim \mathcal{N}(0,\vec{a}^{\top} \mat{\Sigma}^{\MC} \vec{a}).$
    This, however, means that $\P(Z_S^{\MC} = Z_T^{\MC}) = \P(Z_S^{\MC} - Z_T^{\MC} = 0)$ is zero if and only if the variance of the Gaussian above is zero as well. We rewrite this as follows:
    \begin{align*}
        0 \stackrel{!}{=} \vec{a}^{\top} \mat{\Sigma}^{\MC} \vec{a} &= \Sigma_{S,S}^{\MC} - 2 \Sigma_{S,T}^{\MC} + \Sigma_{T,T}^{\MC} \\
        &= \sum_{i=1}^m p_i (q_{i,S}^2 - 2 q_{i,S} q_{i,T} + q_{i,T}^2) \\
        &\qquad - 4 \bigl(\MC_S(\mathcal{G})^2 - 2\MC_S(\mathcal{G})\MC_T(\mathcal{G}) + \MC_T(\mathcal{G})^2\bigr) \\
        &= \biggl(\sum_{i=1}^m p_i (q_{i,S} - q_{i,T})^2\biggr) - \left(\biggl(\sum_{i=1}^m p_i q_{i,S}\biggr) - \biggl(\sum_{i=1}^m p_i q_{i,T}\biggr)\right)^2,
    \end{align*}
    where in the last step we used the fact that $\sum_{i=1}^m p_i q_{i,R} = 2\MC(R)$ for any partition $R\in\mathcal{S}$. As the $p_i$'s are all positive and sum up to 1, and $x\mapsto x^2$ is convex, Jensen's inequality dictates that
    \[
    \sum_{i=1}^m p_i (q_{i,S} - q_{i,T})^2 \geq \biggl(\sum_{i=1}^m p_i (q_{i,S} - q_{i,T})\biggr)^2
    \]
    with equality if and only if $q_{i,S} - q_{i,T} = q_{j,S} - q_{j,T}$ for all $i,j\in V$. This new condition is now obviously satisfied if $q_{i,S} = q_{i,T}$, so it remains to show the converse.
    
    To show this assume that $q_{i,S} - q_{i,T} = q_{j,S} - q_{j,T}$ for all $i,j\in V$, and consider the fact that 
    by \assumptionref{assumptions:additional} both $S$ and $T$ attain the minimal $st$-MinCut. Define $\lambda \coloneqq  q_{i,S} - q_{i,T}$, which is by assumption is independent of $i\in V$. Thus
    \begin{align*}
        0 = 2\MC(S) - 2\MC(T) &= \bigl(\sum_{i=1}^m p_i q_{i,S}\bigr) - \bigl(\sum_{i=1}^m p_i q_{i,T}\bigr) \\
        &= \sum_{i=1}^m p_i (q_{i,S} - q_{i,T}) = \lambda \sum_{i=1}^m p_i = \lambda,
    \end{align*}
    showing $q_{i,S} = q_{i,T}$ for all $i\in V$, and finishing the proof.
\end{proof}

With this, we are ready to derive an asymptotic distribution for \xistref. The idea behind the following proof is loosely based on a more general result by \citet[Lemma 5.5]{KlattMunkZemel2022}, who show existence of a limiting distribution within the framework of basis solutions for linear programs.

\begin{theorem}
    \label{lem:xvst_theory}
    Assume the same setup as in \cref{thm:xist_limit}, and let $\xvst(\cdot)$ denote the output of the \xvstnameref.
    \begin{enumerate}[label=(\roman*)]
        \item Then
        \begin{align}
        &\sqrt{n}\bigl(\xvst(\mathcal{G}_n) - \XC_{S_{*,n}}(\mathcal{G})\bigr) \nonumber\\
        &\qquad\dto \min_{(s,t,S)\in\mathcal{S}_*^{\MC}} \bigl\{Z_S^{\XC} : Z_S^{\MC} < Z_T^{\MC}\;\text{ for all }T\in\mathcal{S}_{st}\setminus\{S\}\bigr\},\label{eq:xvst_theory_limit}
        \end{align}
        where $S_{*,n}\in\mathcal{S}$ denotes the partition that attains the output $\xvst(\mathcal{G}_n)$ of the \xvstnameref applied to the empirical graph $\mathcal{G}_n$. \label{lem:xvst_theory:without_uniqueness}
        \item If additionally \assumptionref{assumptions:uniqueness} without the restriction to $\mathrm{V}_{\mathcal{G}}^{\mathrm{loc}}$ (i.e.\ setting $s,t\in V$ instead of $s,t\in\mathrm{V}_{\mathcal{G}}^{\mathrm{loc}}$) holds, $\sqrt{n}\bigl(\xvst(\mathcal{G}_n) - \xvst(\mathcal{G})\bigr)\dto \mathcal{N}\bigl(0,\Sigma_{S_*,S_*}^{\XC}\bigr)$, where the variance $\Sigma_{S_*,S_*}^{\XC}$ is given by \mysubref{cor:mrnccut_explicit}{diff}, for the (unique) minimizing partition $S_*$ of $\xvst(\mathcal{G})$.
        \label{lem:xvst_theory:with_uniqueness}
    \end{enumerate}
\end{theorem}

\begin{proof}[Proof of \cref{lem:xvst_theory}]
    We first show part~\ref{lem:xvst_theory:without_uniqueness}. Set $d\coloneqq \abs{\mathcal{S}} = 2^{m-1}-1$. A slightly modified version of \mysubref{thm:conv_mrcut_S}{Hadamard} shows that the combined XCut and MCut vectors converge toward a limiting distribution:
    \[
    \vec{v}_n \coloneqq  \begin{bmatrix}\vec{v}^{\XC} \\
    \vec{v}^{\MC}\end{bmatrix} \coloneqq  \begin{bmatrix}
    \sqrt{n}(\XC_{S_1}(\mathcal{G}_n) - \XC_{S_1}(\mathcal{G})) \\
    \vdots \\
    \sqrt{n}(\XC_{S_d}(\mathcal{G}_n) - \XC_{S_d}(\mathcal{G})) \\
    \sqrt{n}(\MC_{S_1}(\mathcal{G}_n) - \MC_{S_1}(\mathcal{G})) \\
    \vdots \\
    \sqrt{n}(\MC_{S_d}(\mathcal{G}_n) - \MC_{S_d}(\mathcal{G}))
    \end{bmatrix}\dto \begin{bmatrix}\vec{Z}^{\XC} \\ \vec{Z}^{\MC}\end{bmatrix} =: \vec{Z}.
    \]
    This fact can be easily seen by considering the Hadamard directionally differentiable functional $\vec{x}\mapsto (g_S^{\XC}(\vec{x}), g_S^{\MC})(\vec{x})$ and applying the delta method as outlined in the proof of \cref{thm:conv_mrcut_S}. Consequently, the random variables $\vec{Z}^{\MC}$ and $\vec{Z}^{\XC}$ are the same as in \cref{thm:conv_mrcut_S}; they are (in general) not independent, however. The marginal distribution of $\vec{Z}^{\MC}$ is the multivariate Gaussian $\mathcal{N}_d(\vec{0},\mat{\Sigma}^{\MC})$ by \mysubref{cor:mrnccut_explicit}{diff}.

    Define the map $\vec{\varphi}:\R^{2d}\to\R^m,\vec{v}\mapsto\vec{e}_{\vec{v},\mathcal{U}}$, where $\vec{e}_{\vec{v},\mathcal{U}}$ is defined as
    \[
    (\vec{e}_{\vec{v},\mathcal{U}})_k \coloneqq  \begin{cases}
        1, &\text{if}\; v_k^{\MC} < v_j^{\MC}\;\text{ for all } j\in\mathcal{U}\;\text{and}\; k\in\mathcal{U},\\
        0, &\text{else.}
    \end{cases}
    \]
    By \cref{lem:q_equality_uniqueness}, under \assumptionref{assumptions:additional} it holds that $P(Z_S^{\MC} = Z_T^{\MC}) = 0$ for any $S,T\in\mathcal{S}_{st}$, meaning that asymptotically, the $st$-MinCut partition $S_{st}\in\mathcal{S}_{st}$ is unique. This implies that $\vec{\varphi}$ is continuous at $Z$-a.e.\ $\vec{v}\in\R^{2d}$ since each such $\vec{v}$ has exactly one component $\vec{v}_S^{\MC}$ satisfying $v_S^{\MC} < v_T^{\MC}$ for all $T\in\mathcal{U}, T\neq S$. Consequently, the corresponding $\vec{e}_{\vec{v},\mathcal{U}}$ is uniquely determined by $\vec{v}$ to be equal to the unit vector in $\R^d$ corresponding to the partition $S\in\mathcal{S}$. Since the condition that $v_S^{\MC} < v_T^{\MC}$ for all $T\in\mathcal{U}, T\neq S$ still holds if any component of $\vec{v}$ is perturbed by a small enough $\eps>0$, $\vec{\varphi}$ is continuous at each such $\vec{v}$.
    
    This also means that the function $\widetilde{\vec{\varphi}}:\R^{2d}\to\R^{3d},\vec{v}\mapsto(\vec{e}_{\vec{v},\mathcal{U}}, \vec{v})$ is continuous as well (where we abused notation slightly to denote the image vector in $\R^{3d}$ more easily). Setting $\mathcal{U}\coloneqq \mathcal{S}_{st}$ we use the continuous mapping theorem to establish that $\widetilde{\vec{\varphi}}(\vec{v}_n)$ converges in distribution towards $\widetilde{\vec{\varphi}}(\vec{Z}) = (\vec{e}_{\vec{Z},\mathcal{U}}, \vec{Z})^{\top}$. We then define the functional $\psi:\R^{3d}\to\R,\,\,(x_1,\ldots,x_{3d})^{\top}\mapsto \sum_{i=1}^d x_i x_{d+i}$ and apply continuous mapping again to receive
    \[
    \sqrt{n}\big(\XC_{\mathcal{S}_{st}^{(n)}}(\mathcal{G}_n) - \XC_{S_{n,st}}(\mathcal{G})\big) = \psi(\widetilde{\vec{\varphi}}(\vec{v}_n)) \dto \psi(\widetilde{\vec{\varphi}}(\vec{Z})) =: \widetilde{Z}_{st}^{\xvst},
    \]
    where $S_{n,st}$ is the ($\vec{Z}$-a.s.\ unique) partition attaining the empirical $st$-MinCut of $\mathcal{G}_n$. Repeating this procedure for all pairs $s,t\in V$, $s\neq t$, and applying \cref{prop:hadamard_delta_method} to the minimum functional (using \mysubref{lem:min}{hadamard} to obtain its Hadamard directional derivative) finally yields the desired result:
    \begin{align*}
    &\sqrt{n}\bigl(\xvst(\mathcal{G}_n) - \min_{{s,t\in V,\;s\neq t}}\XC_{S_{n,st}}(\mathcal{G})\bigr) \\
    &\qquad\dto \min_{{s,t\in V,\;s\neq t}} \Bigl\{\widetilde{Z}_{st}^{\xvst}\bigm| \MC_S(\mathcal{G}) = \MC^{st}(\mathcal{G})\Bigr\} \\
    &\qquad= \min_{{s,t\in V,\; s\neq t}} \Bigl\{Z_S^{\XC} \bigm| \MC_S(\mathcal{G}) = \MC^{st}(\mathcal{G}), Z_S^{\MC} < Z_T^{\MC}\;\text{ for all }T\in\mathcal{S}_{st}\Bigr\},
    \end{align*}
    where $\MC^{st}(\mathcal{G})$ denotes the $st$-MinCut value of the underlying graph $\mathcal{G}$.
    
    We can derive \mysubref{lem:xvst_theory}{with_uniqueness} directly from part~\ref{lem:xvst_theory:without_uniqueness} using \ref{assumptions:uniqueness}, where the latter is not restricted to $\mathrm{V}_{\mathcal{G}}^{\mathrm{loc}}$. It asserts that, for large enough $n$, $\abs{\mathcal{S}_{st}} = 1$, so that asymptotically, $\{Z_S^{\XC} : Z_S^{\MC} < Z_T^{\MC}\;\text{ for all }T\in\mathcal{S}_{st}\setminus\{S\}\bigr\} = \{S_*\}$ almost surely. Additionally, \ref{assumptions:uniqueness} states that as $n\to\infty$, $S_*$ uniquely attains the minimum XCut value over all $st$-MinCuts. In part~\ref{lem:xvst_theory:without_uniqueness}, the convergence \eqref{eq:xvst_theory_limit} was shown, whence the limiting distribution reduces to the component $Z_{S_*}^{\XC}$. From \mysubref{cor:mrnccut_explicit}{diff} the claim follows as it has already been shown there that $Z_{S_*}^{\XC}\sim \mathcal{N}(0,\Sigma_{S_*,S_*}^{\XC})$.
\end{proof}

We now show \cref{thm:xist_limit}. Although \xistref is a modified version of the \xvstnameref, we require a slightly different proof strategy compared to \cref{lem:xvst_theory}, namely the Portmanteau theorem to characterize distributional convergence.

\begin{proof}[Proof of \cref{thm:xist_limit}]


We first show \ref{thm:xist_limit:without_uniqueness}. \xistref takes into account those vertices $i$ that are local maxima, i.e.\ such that $Y_j\leq Y_i$ for all $j\sim i$. By \cref{lem:multinomial_components_distinct} this can be reduced to a strict inequality since $\P(Y_i = Y_j)\to 0$ as $n\to\infty$. If $p_i > p_j$ for all $j$ with $j\sim i$, $Y_i > Y_j$ for all $n\geq N$ (for some $N\in\N$) and all $j$ with $j\sim i$. Denote the set of vertices that satisfy this by $V_1\subset V$. Additionally, we have to consider those $i\in V$ with $p_i \geq p_j$ for all $j\sim i$ and $p_i = p_k$ for some $k\sim i$. Denote this set of vertices by $V_2\subseteq V$, so that together
\begin{align*}
    V_1 &\coloneqq \{r\in V : p_r > p_i\;\text{for all }i\sim r\}\quad\text{and}\\
    \widehat{V}_2 &\coloneqq \{r\in V\setminus V_1 : p_r > p_i\text{ or } p_r = p_i \text{ and }Z_r^{\vec{p}} > Z_i^{\vec{p}}\; \text{ for all }i\sim r\}.
    \end{align*}
Consequently, the set of local maxima is contained in $V_1\cup V_2$ for large enough $n$. Note that as $n\to\infty$, every $s\in V_1$ is almost surely a local maximum. The same cannot be said for $V_2$, although all $i,j\in V_2$ satisfy $p_i = p_j$ if $i\sim j$. This means that $Y_i > Y_j$ is equivalent to $\sqrt{n}(Y_i - p_i) > \sqrt{n}(Y_j - p_j)$ for all $i,j\in V_2$ with $i\sim j$. Since we know that any $i\in V_2$ are guaranteed to be local maxima for large $n$, only consider those in $V_2\coloneqq \{i_1,\ldots,i_{|V_2|}\}$.
    
    Now, enumerating $\mathcal{S}\coloneqq \{S_1,\ldots,S_d\}$, a slightly modified version of the proof of \cref{thm:conv_mrcut_S} shows the following convergence:
    \begin{equation}
    \label{eq:xvst_loc_max_join_convergence}
    \vec{v}_n \coloneqq  \begin{bmatrix}\vec{v}^{\XC} \\
    \vec{v}^{\MC} \\ \vec{v}^{\vec{p}}\end{bmatrix} \coloneqq  \begin{bmatrix}
    \sqrt{n}(\XC_{S_1}(\mathcal{G}_n) - \XC_{S_1}(\mathcal{G})) \\
    \vdots \\
    \sqrt{n}(\XC_{S_d}(\mathcal{G}_n) - \XC_{S_d}(\mathcal{G})) \\
    \sqrt{n}(\MC_{S_1}(\mathcal{G}_n) - \MC_{S_1}(\mathcal{G})) \\
    \vdots \\
    \sqrt{n}(\MC_{S_d}(\mathcal{G}_n) - \MC_{S_d}(\mathcal{G})) \\
    \sqrt{n}(Y_1 - p_1) \\
    \vdots \\
    \sqrt{n}(Y_m - p_m) \\
    \end{bmatrix}\dto 
    \begin{bmatrix}\vec{Z}^{\XC} \\ \vec{Z}^{\MC} \\ \vec{Z}^{\vec{p}}\end{bmatrix} =: \vec{Z},
    \end{equation}
    where the random variable $\vec{Z}$ can be expressed 
    It should be noted that when the balancing term $\bal_S(\cdot)$ is differentiable at $\vec{p}$, $\vec{Z}$ will be a multivariate (centered) Gaussian whose covariance matrix we can compute explicitly (analogously to the aforementioned proof). For now, however, we are content with the above limit, remembering that its components $\vec{Z}^{\MC}$, $\vec{Z}^{\XC}$ and $\vec{Z}^{\vec{p}}$ are dependent. 
    
    Referring back to our above comments regarding the definitions of $V_1$ and $V_2$, for $s,t\in V$, $s\neq t$, a partition $S\in\mathcal{S}_{st}$ is considered by \xistref  (when applied to the empirical graph $\mathcal{G}_n$) if and only if $\MC_S(\mathcal{G}_n) \leq \MC_T(\mathcal{G}_n)$ for all $T\in\mathcal{S}_{st}\setminus \{T\}$ and, for $r=s,t$, $r \in V_1$ or $r\in V_2$ and $Y_r > Y_i$ for all $i\in V_2$ with $i\sim r$. Note that we may disregard the cases $\MC_S(\mathcal{G}_n) = \MC_T(\mathcal{G}_n)$ and $Y_i = Y_j$ by \mysubref{cor:mrnccut_explicit}{diff} and \cref{lem:multinomial_components_distinct}: By definitions of $\mathcal{S}_{st}$, $V_1$ and $V_2$ we have the equivalence
    \begin{align*}
        \MC_S(\mathcal{G}_n) < \MC_T(\mathcal{G}_n) \iff v_S^{\MC} < v_T^{\MC} \quad\text{and}\quad
        Y_r > Y_i \iff v_r^{\vec{p}} > v_i^{\vec{p}}
    \end{align*}
    for $S,T\in\mathcal{S}_{st}$ and $r,i\in V_2$ with $i\sim r$. 
    Set $\mathcal{S}^{\MC} \coloneqq  \{(s,t,S)\in V\times V\times\mathcal{S}\,\,|\,\,s\neq t, \MC_S(\mathcal{G}) = \MC^{st}(\mathcal{G})\}$ and define the random variable
    \begin{align*}
    Z^{\xvst} \coloneqq  \min_{(s,t,S)\in\mathcal{S}^{\MC}} \biggl\{Z_S^{\XC} : &\MC_S(\mathcal{G}) = \MC^{st}(\mathcal{G}), Z_S^{\MC} < Z_T^{\MC}\;\text{ for all }T\in\mathcal{S}_{st}\setminus\{S\}, \\
    &\XC_S(\mathcal{G})=\min_{(u,v,T)\in\mathcal{S}^{\MC}}\XC_T(\mathcal{G}),\;\text{ and, for } r=s,t: \\
    &\;r \in V_1\;\text{ or }\; (r\in V_2\,\text{ and }\, Z_r^{\vec{p}} > Z_i^{\vec{p}}\;\text{ for all }i\in V_2, i\sim r)\biggr\},
    \end{align*}
    and let $G\in B(\R)$ be an arbitrary $Z^{\xvst}$-continuity set. Further, for fixed $s,t\in V$, $s\neq t$, and $S\in\mathcal{S}$ set
    \begin{align*}
        H_{s,t,S} &\coloneqq  \Bigl\{\vec{u}\coloneqq (\vec{u}^{\XC},\vec{u}^{\MC},\vec{u}^{\vec{p}})\in\R^d\times\R^d\times\R^m : u_S^{\MC} < u_T^{\MC}\;\text{ for all }T\in\mathcal{S}_{st}\setminus \{S\},\\[-2mm]
        &\hspace{1.5cm} r\in V_1\,\text{ or }\, (r\in V_2\,\text{ and }\, u_r^{\vec{p}} > u_i^{\vec{p}}\;\text{ for all } i\in V_2\text{ with } i\sim r)\;\text{ for }r=s,t
        \Bigr\}\\
        A_{s,t,S}^{\XC} &\coloneqq  \Bigl\{\vec{u}^{\XC}\in\R^d :  u_S^{\XC} \leq u_T^{\XC} \;\text{ for all } (v,w,T)\in \mathcal{S}^{\MC} \text{ with } T\neq S\text{ and } \vec{u}\in H_{v,w,T}\Bigr\},\\
        G_{s,t,S} &\coloneqq  \Bigl\{\vec{u}\in\R^{2d+m} : \vec{u}\in H_{s,t,S},\, \vec{u}_S^{\XC}\in G\cap A_{s,t,S}^{\XC}\Bigr\}.
    \end{align*}
    Naturally, $H_{s,t,S}$, $A_{s,t,S}$ and $G_{s,t,S}$ are Borel-measurable on $\R^{2d+m}$. The boundary of $H_{s,t,S}$ satisfies
    \begin{align*}
    \partial H_{s,t,S} &\subseteq \Bigl(\bigcup_{T\in\mathcal{S}_{st}\setminus \{S\}} \bigl\{u_S^{\MC} = u_T^{\MC}\bigr\}\Bigr) \cup \Bigl(\bigcup_{\substack{i:i\sim r\\r=s,t}} \bigl\{u_r^{\vec{p}} = u_i^{\vec{p}}\bigr\}\Bigr), \\
    \intertext{meaning that we can capture the boundary of $G_{s,t,S}$ as follows:}
    \partial G_{s,t,S} &\subseteq \partial H_{s,t,S} \;\cup\;\partial \bigl(G \cap A_{s,t,S}^{\XC} \bigr)\times\R^{d+m}\; \cap H_{s,t,S}.
    \end{align*}
    Since $\{Z_S^{\XC} = Z_T^{\XC}\} \subseteq \{Z^{\xvst}\in\partial G\}$ for any $(s,t,S),(u,v,T)\in\mathcal{S}^{\MC}$ with $S\neq T$ and $Z\in H_{s,t,S}\cap H_{u,v,T}$, the inclusion above establishes $\bigl\{\vec{Z}\in \partial G_{s,t,S}\bigr\} \subseteq \bigl\{\vec{Z}\in \partial H_{s,t,S}\bigr\} \cup \bigl\{Z^{\xvst}\in \partial G\bigr\}.$
    Consequently,
    \begin{align*}
    &\P(\vec{Z}\in\partial G_{s,t,S}) \\
    &\quad\leq \P\Bigl(\bigl\{Z^{\xvst}\in \partial G\bigr\} \cup \Bigl\{\vec{Z}\in \bigcup_{T\in\mathcal{S}_{st}\setminus \{S\}} \bigl\{u_S^{\MC} = u_T^{\MC}\bigr\} \cup \bigcup_{\substack{i:i\sim r\\r=s,t}} \bigl\{u_r^{\vec{p}} = u_i^{\vec{p}}\bigr\}\Bigr\}\Bigr) \\
    &\quad= \P\Bigl(\bigl\{Z^{\xvst}\in \partial G\bigr\} \cup \bigcup_{T\in\mathcal{S}_{st}\setminus \{S\}} \bigl\{Z_S^{\MC} = Z_T^{\MC}\bigr\} \cup \bigcup_{\substack{i:i\sim r\\r=s,t}} \bigl\{Z_r^{\vec{p}} = Z_i^{\vec{p}}\bigr\}\Bigr) \\
    &\quad\leq \P\bigl(\vec{Z}^{\xvst}\in \partial G\bigr) + \sum_{T\in\mathcal{S}_{st}\setminus \{S\}} \P\bigl(Z_S^{\MC} = Z_T^{\MC}\bigr) + \sum_{\substack{i:i\sim r\\r=s,t}} \P\bigl(Z_r^{\vec{p}} = Z_i^{\vec{p}}\bigr) = 0,
    \end{align*}
    where in the last step we used the following facts: Firstly, $G$ is a $\vec{Z}^{\xvst}$-continuity set, secondly, $\P(Z_S^{\MC} = Z_T^{\MC}) = 0$ for any two distinct partitions $S,T\in\mathcal{S}$ by \assumptionref{assumptions:additional} and \cref{lem:q_equality_uniqueness}, and thirdly $\P(Z_i^{\vec{p}} = Z_j^{\vec{p}}) = 0$ holds for any $i,j\in V$ by the proof of \cref{lem:multinomial_components_distinct}. Hence, $G_{s,t,S}$ is a $\vec{Z}$-continuity set.

    To finish the proof, define $S_{*,n}$ to be the partition that attains the output $\xist(\mathcal{G}_n)$ of \xistref applied to $\mathcal{G}_n$. We first apply the Portmanteau theorem, then the law of total probability and afterwards Portmanteau a second time, yielding (for fixed $s,t\in V$, $s\neq t$):
    \begin{multline*}
    \P(\sqrt{n}(\xist(\mathcal{G}_n) - \XC_{S_{*,n}}(\mathcal{G})))\in G) = \sum_{\substack{s,t\in V\\s\neq t}}
    \sum_{S\in\mathcal{S}_{st}} \P(\vec{v}_n\in G_{s,t,S}) \\
    \dto \sum_{\substack{s,t\in V\\s\neq t}}
    \sum_{S\in\mathcal{S}_{st}} \P(\vec{Z}\in G_{s,t,S})
    = \P(Z^{\xvst}\in G)
    \end{multline*}
    due to the definition of $Z^{\xvst}$. This shows \mysubref{thm:xist_limit}{without_uniqueness} since $G$ was arbitrary.
    
    Now \mysubref{thm:xist_limit}{with_uniqueness} follows directly from part~\ref{thm:xist_limit:without_uniqueness} using \assumptionref{assumptions:uniqueness}. The latter asserts that the set $\{Z_S^{\XC} : Z_S^{\MC} < Z_T^{\MC}\;\text{ for all }T\in\mathcal{S}_{st}\setminus\{S\}\bigr\}$ contains exactly one element for large enough $n$, namely $Z_{S_*}^{\XC}$ because it is the unique XCut minimizer, and it uniquely attains its corresponding $st$-MinCut, such that $\mathcal{S}_{st} = \{S_*\}$. Moreover, under \assumptionref{assumptions:uniqueness}, the minimum XCut value over $st$-MinCuts for vertices in $V_1\cap \widehat{V}_2$ is asymptotically almost surely attained by $S_*$. Consequently, the limit from part~\ref{thm:xist_limit:without_uniqueness} simplifies to the component $Z_{S_*}^{\XC}$ which, by \mysubref{cor:mrnccut_explicit}{diff}, is normally distributed with variance $\Sigma_{S_*,S_*}^{\XC}$.
\end{proof}

Lastly, we show \cref{thm:conv_mrcut_S_k} which is merely an extension of \cref{thm:conv_mrcut_S} for the multiway cut that separates the nodes of the graph into $k\geq 2$ partitions.

\begin{proof}[Proof of \cref{thm:conv_mrcut_S_k}]
    We prove this using the same procedure we showed \cref{thm:conv_mrcut_S} with. First, recall again that $\sqrt{n}(\vec{Y}/n - \vec{p}) \dto \mathcal{N}_m(\vec{0},\mat{\Sigma})$, where $\mat{\Sigma}=(\Sigma_{i,j})\in \R^{m \times m}$ is the covariance matrix of the multinomial distribution $\mult(1,\vec{p})$. Now, for XCut and any partition $S\in\mathcal{S}$ we define the function
 	\begin{alignat*}{2}
    &\vec{g}^{\kXC}:\R^m\to\R^{|\mathcal{S}|},\quad&&\vec{x}\,\mapsto\, \left(g_S^{\kXC}(\vec{x})\right)_{S\in\mathcal{S}}\\
 	\text{for}\quad &g_S^{\kXC}:\R^m\to\R, &&\vec{x}\,\mapsto\,\frac{1}{2}\sum_{\ell=1}^k g_{S_{\ell}}^{\XC}(\vec{x}),
 	\end{alignat*}
	where $\vec{x}=(x_1,\ldots,x_m)$, $S=\{S_1,\ldots,S_k\}\in\mathcal{S}^{k}$, and $g_S^{\XC}:\R^m\to\R$ is the XCut functional we used throughout this paper.
	
	\ref{thm:conv_mrcut_S_k:Hadamard}: If $g_S^{\XC}$ is Hadamard differentiable, we apply the delta method for Hadamard directionally differentiable functions (\cref{prop:hadamard_delta_method}) to receive:
	\[
	\sqrt{n}\Big(\kXC_S(\mathcal{G}_n) - \kXC_S(\mathcal{G})\Big) = \sqrt{n}\left(g_S^{\kXC}\left(\tfrac{\vec{Y}}{n}\right) - g_S^{\kXC}(\vec{p})\right)\dto \left((g_S^{\kXC})_{\vec{p}}'(\vec{Z})\right)_{S\in\mathcal{S}},
	\]
	where $\vec{Z}\sim\mathcal{N}_m(\vec{0},\mat{\Sigma})$. It is evident from \cref{def:hadamard_diff} that taking the Hadamard directional derivative at any fixed point in any fixed direction is a linear operator. Consequently, we receive a limiting distribution of
	\[
	(g_S^{\kXC})_{\vec{p}}'(\vec{Z}) = \frac{1}{2} \sum_{\ell=1}^k (g_{S_{\ell}}^{\XC})_{\vec{p}}'(\vec{Z}) = \frac{1}{2} \sum_{\ell=1}^k Z_{S_{\ell}}^{\XC}\quad\text{for any}\quad S=\{S_1,\ldots,S_k\}\in\mathcal{S}^{k}.
	\]
	Here $Z_{S_{\ell}}^{\XC}$ is the limiting distribution of $\sqrt{n}(\XC_{S_{\ell}}(\mathcal{G}_n) - \XC_{S_{\ell}}(\mathcal{G}))$ from \cref{thm:conv_mrcut_S}.
	
	\ref{thm:conv_mrcut_S_k:diff}: If $g_S^{\XC}$ is $\R^m$-differentiable, we may apply the standard delta method:
	\begin{align*}
	\Bigl(\sqrt{n}\bigl(\kXC_R(\mathcal{G}_n) - \kXC_R(\mathcal{G})\bigr)\Bigr)_{R\in\mathcal{S}} &= \Bigl(\sqrt{n}\bigl(g_R^{\kXC}\bigl(\tfrac{\vec{Y}}{n}\bigr)-g_R^{\kXC}(\vec{p})\bigr)\Bigr)_{R\in\mathcal{S}} \\
    &\dto \mathcal{N}_{\abs{\mathcal{S}}}\bigl(\vec{0},\mat{\Sigma}^{\kXC}\bigr),
	\end{align*}
	so that for $T=\{T_1,\ldots,T_k\},\, S=\{S_1,\ldots,S_k\}\in\mathcal{S}^{k}$ the limiting covariance matrix $\Sigma^{\kXC}$ is given by
	\begin{align*}
	\bigl(\mat{\Sigma}^{\kXC}\bigr)_{T,S} &= \nabla g_T^{\kXC}(\vec{p})^{\top} \mat{\Sigma}\,\, \nabla g_S^{\kXC}(\vec{p}) \\
    &= \frac{1}{4} \sum_{i=1}^k \sum_{j=1}^k g_{T_i}^{\XC}(\vec{p})^{\top} \mat{\Sigma}\,\, \nabla g_{S_j}^{\XC}(\vec{p}) \\
    &= \frac{1}{4} \sum_{i=1}^k \sum_{j=1}^k \cov(Z_{T_i}^{\XC}, Z_{S_j}^{\XC}),
	\end{align*}
	where, for any $R\in\mathcal{S}$, $Z_R^{\XC}$ denotes the limiting distribution of $\sqrt{n}\bigl(\XC_R(\mathcal{G}_n) - \XC_R(\mathcal{G})\bigr)$ that we have already computed in \cref{thm:conv_mrcut_S}.
\end{proof}

\end{appendices}
%
%
%

\section*{Acknowledgements}
LS is supported by the DFG (German Research Foundation) under project GRK 2088: \enquote{Discovering Structure in Complex Data}, subproject A1. HL is funded and AM is supported by the DFG under Germany’s Excellence Strategy, project EXC 2067: \enquote{Multiscale Bioimaging: from Molecular Machines to Networks of Excitable Cells} (MBExC). AM and HL are supported by DFG CRC 1456 \enquote{Mathematics of Experiment}, and AM is supported by DFG RU 5381 \enquote{Mathematical Statistics in the Information Age -- Statistical Efficiency and Computational Tractability}. The authors especially thank Max Wardetzky (University of Göttingen) for various helpful discussions and comments. 



\bibliographystyle{plainnat}
\bibliography{ma-bib}       

\end{document}